\documentclass{article} 

\usepackage{silence}
\usepackage[margin=1in]{geometry}

\usepackage[T1]{fontenc}
\usepackage[utf8]{inputenc}
\usepackage{newunicodechar}

\usepackage{amsmath}
\usepackage{amssymb}
\usepackage{amsthm}
\usepackage[mathscr]{euscript}
\usepackage{adjustbox}

\usepackage{tikz}
\usetikzlibrary{cd}
\usetikzlibrary{decorations.markings}
\usetikzlibrary{decorations.pathmorphing}
\usetikzlibrary{decorations.pathreplacing,calc,tikzmark}
\usetikzlibrary{positioning,arrows.meta,automata,shapes}
\usepackage[all,cmtip]{xy}
\SilentMatrices

\pgfdeclarelayer{nodelayer}
\pgfdeclarelayer{edgelayer}
\pgfsetlayers{edgelayer,nodelayer,main}

\tikzstyle{none}=[]
\tikzstyle{new style 0}=[fill=white, draw=black, shape=circle]

\tikzstyle{new edge style 0}=[<-]
\tikzstyle{new edge style 1}=[->]
\tikzstyle{new edge style 2}=[-, dashed]
\tikzstyle{new edge style 3}=[-, thick]

\usepackage{xr-hyper}
\def\ecref{\cref*}

\newif\ifpdf \pdffalse \ifx\pdfoutput\undefined\else\ifx\pdfoutput\relax\else\ifnum\pdfoutput<1 \else\pdftrue\fi\fi\fi
\ifpdf
\usepackage[pdfencoding=auto,psdextra,backref=section,linktoc=all]{hyperref}

\else
\usepackage[hypertex,backref=section,linktoc=all]{hyperref}
\fi

\usepackage[capitalize]{cleveref}

\crefformat{equation}{(#2#1#3)}
\crefname{figure}{Figure}{Figures}

\catcode`@=11
\let\over\@@over
\let\atop\@@atop
\let\above\@@above
\let\overwithdelims\@@overwithdelims
\let\atopwithdelims\@@atopwithdelims
\let\abovewithdelims\@@abovewithdelims
\def\eqalign#1{\null\,\vcenter{\openup\jot\m@th
  \ialign{\strut\hfil$\displaystyle{##}$&$\displaystyle{{}##}$\hfil
      \crcr#1\crcr}}\,}
\newskip\xcentering
\def\eqalignno#1{\displ@y \tabskip\xcentering
  \halign to\displaywidth{\hfil$\@lign\displaystyle{##}$\tabskip\z@skip
    &$\@lign\displaystyle{{}##}$\hfil\tabskip\xcentering
    &\llap{$\@lign##$}\tabskip\z@skip\crcr
    #1\crcr}}
\def\eqlabel#1{\refstepcounter{equation}\label{#1}\ifmmode\ifinner\else\eqno\fi\fi\hbox{\@eqnnum}} 

\catcode`@=12

\makeatletter
\@addtoreset{equation}{section}
\@addtoreset{equation}{subsection}
\let\proof@qed\displaymath@qed
\makeatother

\theoremstyle{definition}
\newtheorem{theorem}[equation]{Theorem}
\newtheorem{definition}[equation]{Definition}
\newtheorem{lemma}[equation]{Lemma}
\newtheorem{corollary}[equation]{Corollary}
\newtheorem{proposition}[equation]{Proposition}

\newtheorem{construction}[equation]{Construction}
\newtheorem{remark}[equation]{Remark}

\newtheorem{example}[equation]{Example}
\newtheorem{notation}[equation]{Notation}

\numberwithin{equation}{subsection}

\newunicodechar{α}{\alpha} \newunicodechar{β}{\beta} \newunicodechar{γ}{\gamma} \newunicodechar{δ}{\delta} \newunicodechar{ϵ}{\epsilon} \newunicodechar{ζ}{\zeta} \newunicodechar{η}{\eta} \newunicodechar{θ}{\theta} \newunicodechar{ι}{\iota} \newunicodechar{κ}{\kappa} \newunicodechar{λ}{\lambda} \newunicodechar{μ}{\mu} \newunicodechar{ν}{\nu} \newunicodechar{ξ}{\xi} \newunicodechar{ο}{o} \newunicodechar{π}{\pi} \newunicodechar{ρ}{\rho} \newunicodechar{σ}{\sigma} \newunicodechar{τ}{\tau} \newunicodechar{υ}{\upsilon} \newunicodechar{ϕ}{\phi} \newunicodechar{χ}{\chi} \newunicodechar{ψ}{\psi} \newunicodechar{ω}{\omega} \newunicodechar{ε}{\varepsilon} \newunicodechar{ϑ}{\vartheta} \newunicodechar{ϖ}{\varpi} \newunicodechar{ϱ}{\varrho} \newunicodechar{ς}{\varsigma} \newunicodechar{φ}{\varphi} \newunicodechar{Γ}{\Gamma} \newunicodechar{Δ}{\Delta} \newunicodechar{Θ}{\Theta} \newunicodechar{Λ}{\Lambda} \newunicodechar{Ξ}{\Xi} \newunicodechar{Π}{\Pi} \newunicodechar{Σ}{\Sigma} \newunicodechar{Υ}{\Upsilon} \newunicodechar{Φ}{\Phi} \newunicodechar{Ψ}{\Psi} \newunicodechar{Ω}{\Omega} \newunicodechar{ℵ}{\aleph} \newunicodechar{ℏ}{\hbar} \newunicodechar{ℓ}{\ell} \newunicodechar{℘}{\wp} \newunicodechar{ℜ}{\Re} \newunicodechar{ℑ}{\Im} \newunicodechar{∂}{\partial} \newunicodechar{∞}{\infty} \newunicodechar{∅}{\emptyset} \newunicodechar{∇}{\nabla} \newunicodechar{√}{\surd} \newunicodechar{⊤}{\top} \newunicodechar{⊥}{\bot} \newunicodechar{∠}{\angle} \newunicodechar{△}{\triangle} \newunicodechar{∀}{\forall} \newunicodechar{∃}{\exists} \newunicodechar{♭}{\flat} \newunicodechar{♮}{\natural} \newunicodechar{♯}{\sharp} \newunicodechar{♣}{\clubsuit} \newunicodechar{♢}{\diamondsuit} \newunicodechar{♡}{\heartsuit} \newunicodechar{♠}{\spadesuit} \newunicodechar{∐}{\coprod} \newunicodechar{⋁}{\bigvee} \newunicodechar{⋀}{\bigwedge} \newunicodechar{⨄}{\biguplus} \newunicodechar{⋂}{\bigcap} \newunicodechar{⋃}{\bigcup} \newunicodechar{∫}{\int} \newunicodechar{∏}{\prod} \newunicodechar{∑}{\sum} \newunicodechar{⨂}{\bigotimes} \newunicodechar{⨁}{\bigoplus} \newunicodechar{⨀}{\bigodot} \newunicodechar{⨆}{\bigsqcup} \newunicodechar{◁}{\triangleleft} \newunicodechar{▷}{\triangleright} \newunicodechar{▽}{\bigtriangledown} \newunicodechar{∧}{\wedge} \newunicodechar{∨}{\vee} \newunicodechar{∩}{\cap} \newunicodechar{∪}{\cup} \newunicodechar{⊓}{\sqcap} \newunicodechar{⊔}{\sqcup} \newunicodechar{⊎}{\uplus} \newunicodechar{⨿}{\amalg} \newunicodechar{⋄}{\diamond} \newunicodechar{∙}{\bullet} \newunicodechar{≀}{\wr} \newunicodechar{⊙}{\odot} \newunicodechar{⊘}{\oslash} \newunicodechar{⊗}{\otimes} \newunicodechar{⊖}{\ominus} \newunicodechar{⊕}{\oplus} \newunicodechar{∓}{\mp} \newunicodechar{∘}{\circ} \newunicodechar{○}{\Orb} \newunicodechar{∖}{\setminus} \newunicodechar{⋅}{\cdot} \newunicodechar{∗}{\ast} \newunicodechar{⨯}{\times} \newunicodechar{⋆}{\star} \newunicodechar{∝}{\propto} \newunicodechar{⊑}{\sqsubseteq} \newunicodechar{⊒}{\sqsupseteq} \newunicodechar{∥}{\parallel} \newunicodechar{∣}{\divides} \newunicodechar{⊣}{\dashv} \newunicodechar{⊢}{\vdash} \newunicodechar{↗}{\nearrow} \newunicodechar{↘}{\searrow} \newunicodechar{↖}{\nwarrow} \newunicodechar{↙}{\swarrow} \newunicodechar{⇔}{\Leftrightarrow} \newunicodechar{⇐}{\Leftarrow} \newunicodechar{⇒}{\Rightarrow} \newunicodechar{≠}{\neq} \newunicodechar{≤}{\leq} \newunicodechar{≥}{\geq} \newunicodechar{≻}{\succ} \newunicodechar{≺}{\prec} \newunicodechar{≈}{\approx} \newunicodechar{≽}{\succeq} \newunicodechar{≼}{\preceq} \newunicodechar{⊃}{\supset} \newunicodechar{⊂}{\subset} \newunicodechar{⊇}{\supseteq} \newunicodechar{⊆}{\subseteq} \newunicodechar{∈}{\in} \newunicodechar{∋}{\ni} \newunicodechar{≫}{\gg} \newunicodechar{≪}{\ll} \newunicodechar{↔}{\leftrightarrow} \newunicodechar{↦}{\mapsto} \newunicodechar{∼}{\sim} \newunicodechar{≃}{\simeq} \newunicodechar{≡}{\equiv} \newunicodechar{≍}{\asymp} \newunicodechar{⌣}{\smile} \newunicodechar{⌢}{\frown} \newunicodechar{↼}{\leftharpoonup} \newunicodechar{↽}{\leftharpoondown} \newunicodechar{⇀}{\rightharpoonup} \newunicodechar{⇁}{\rightharpoondown} \newunicodechar{↪}{\hookrightarrow} \newunicodechar{↩}{\hookleftarrow} \newunicodechar{⋈}{\bowtie} \newunicodechar{⊨}{\models} \newunicodechar{⟹}{\Longrightarrow} \newunicodechar{⟶}{\longrightarrow} \newunicodechar{⟵}{\longleftarrow} \newunicodechar{⟸}{\Longleftarrow} \newunicodechar{⟼}{\longmapsto} \newunicodechar{⟷}{\longleftrightarrow} \newunicodechar{⟺}{\Longleftrightarrow} \newunicodechar{⋯}{\cdots} \newunicodechar{⋮}{\vdots} \newunicodechar{⋱}{\ddots} \newunicodechar{↕}{\updownarrow} \newunicodechar{⇑}{\Uparrow} \newunicodechar{⇓}{\Downarrow} \newunicodechar{⇕}{\Updownarrow} \newunicodechar{⌉}{\rceil} \newunicodechar{⌈}{\lceil} \newunicodechar{⌋}{\rfloor} \newunicodechar{⌊}{\lfloor} \newunicodechar{≅}{\cong} \newunicodechar{∉}{\notin} \newunicodechar{⇌}{\rightleftharpoons} \newunicodechar{≐}{\doteq} \newunicodechar{∄}{\not\exists} \newunicodechar{∌}{\not\ni} \newunicodechar{∔}{\dot+} \newunicodechar{∕}{/} \newunicodechar{∤}{\not|} \newunicodechar{∦}{\not\|} \newunicodechar{∬}{\int\!\!\!\int} \newunicodechar{∭}{\int\!\!\!\int\!\!\!\int} \newunicodechar{∮}{\oint} \newunicodechar{∸}{\dot-} \newunicodechar{≁}{\not\sim} \newunicodechar{≄}{\not\simeq} \newunicodechar{≆}{\not\cong} \newunicodechar{≇}{\not\cong} \newunicodechar{≉}{\not\approx} \newunicodechar{≔}{:=} \newunicodechar{≕}{=:} \newunicodechar{≢}{\not\equiv} \newunicodechar{≭}{\not\asump} \newunicodechar{≮}{\not<} \newunicodechar{≯}{\not<} \newunicodechar{≰}{\not\le} \newunicodechar{≱}{\not\ge} \newunicodechar{⊀}{\not\prec} \newunicodechar{⊁}{\not\succ} \newunicodechar{⊄}{\not\subset} \newunicodechar{⊅}{\not\supset} \newunicodechar{⊈}{\not\subseteq} \newunicodechar{⊉}{\not\supseteq} \newunicodechar{⊦}{\vdash} \newunicodechar{⊧}{\models} \newunicodechar{⊬}{\not\vdash} \newunicodechar{⊭}{\not\models} \newunicodechar{⊲}{\triangleleft} \newunicodechar{⊳}{\triangleright} \newunicodechar{⋠}{\not\preceq} \newunicodechar{⋡}{\not\succeq} \newunicodechar{⋤}{\not\sqsubseteq} \newunicodechar{⋥}{\not\sqsupseteq} \newunicodechar{⋪}{\not\triangleleft} \newunicodechar{⋫}{\not\triangleright} \newunicodechar{◻}{\square} 

\newunicodechar{↓}{\downarrow}
\newunicodechar{→}{\to}
\newunicodechar{←}{\gets}
\newunicodechar{⟨}{\langle}
\newunicodechar{⟩}{\rangle}
\newunicodechar{…}{\ldots}
\newunicodechar{∞}{\ifmmode\infty\else$\infty$\fi}
\newunicodechar{≔}{\colon=}

\newunicodechar{⊠}{\pp}
\mathchardef\pp"2403 
\mathchardef\ppdom"2400 

\def\Cc{{\mathcal C}}

\newcommand{\into}{\hookrightarrow}
\newcommand{\dual}{{\rm Dual}}

\def\hom{\mathop{\rm Hom}\nolimits}
\def\map{\mathop{\rm Map}\nolimits}

\def\GL{{\rm GL}}
\def\O{{\rm O}}

\newcommand{\RR}{{\bf R}}
\newcommand{\NN}{{\bf N}}
\newcommand{\ZZ}{{\bf Z}}

\newcommand{\CC}{{\bf C}}

\def\inj{{\sf inj}}

\def\glob{{\sf glob}}

\def\uple{{\sf uple}}

\def\tdeloop{{\rm B}}

\def\MT{{\sf MT}}
\def\MTcla{{\rm MT}}
\def\cK{{\cal K}}

\newcommand{\sset}{{s\mathscr{S}\mathrm{et}}}
\newcommand{\set}{{\mathscr{S}\mathrm{et}}}

\newcommand{\PSh}{{\mathscr{PS}\mathrm{h}}}
\def\sPSh{\PSh_\Delta}

\newcommand{\Sp}{{\sf Sp}}

\newcommand{\cartsp}{{\sf Cart}}
\newcommand{\FEmbCart}{{\sf FEmbCart}}

\def\FEmb{{\sf FEmb}}
\def\Emb{{\sf Emb}}

\def\id{{\rm id}}

\def\Fun{\mathop{\rm Fun}\nolimits}
\def\Funmon{\Fun^\otimes}

\def\Vect{{\rm Vect}}
\def\Bord{{\rm Bord}}

\def\Cat{{\rm Cat}}
\def\FFT{{\sf FFT}}
\def\op{{\rm op}}

\def\csp{{\sf B_\smallint}} 
 
\def\deloop{{\bf B}}
\def\fdeloop{{\sf B}}
\def\smcat{{\rm C^∞Cat}^\otimes}
\def\smcatglob{{\rm C^∞Cat}^{\otimes,\glob}}
\def\smcatuple{{\rm C^∞Cat}^{\otimes,\uple}}
\def\catdual{{\rm Cat}^{\otimes,\vee}}
\def\smcatdual{{\rm C^{\infty}Cat}^{\otimes,\vee}}
\def\smcatinv{{\rm C^{\infty}Cat}^{\otimes,\times}}
\def\sm{{\rm C}^∞}
\def\hocolim{\mathop{\rm hocolim}}
\def\holim{\mathop{\rm holim}}

\def\sing{\mathord{\rm sing}}
\def\hq{{/\!/}}

\def\core{\mathop{\rm core}}

\def\Ho{\mathop{\sf Ho}\nolimits}
\def\lconst{{\sf \check Cech,flc}}
\def\loc{{\sf \check Cech}}
\def\Sing{\mathop{\rm Sing}\nolimits}
\def\unit{{\rm unit}}

\def\Cnec{\mathfrak{C}^{\rm nec}}
\def\Cut{{\sf Cut}}
\def\DiffCut{{\sf DiffCut}}

\def\unitable{{\rm unit}}
\def\Ht{\widetilde{H}}
\def\Ot{\widetilde{O}}

\def\Hb{\overline{H}}
\def\Ob{\overline{O}}

\def\ltoarr#1{\mathop{\count0=#1 \loop\ifnum\count0>0 \smash-\mkern-7mu \advance\count0 -1 \repeat \mathord\rightarrow}\limits} 
\def\longto#1^#2{\mathrel{\ltoarr{#1}^{#2}}} 
\def\lgetsarr#1{\mathop{\mathord\leftarrow \count0=#1 \loop\ifnum\count0>0 \mkern-7mu\smash-\advance\count0 -1 \repeat}\limits} 
\def\longgets#1^#2{\mathrel{\lgetsarr{#1}\limits^{#2}}} 

\def\gmatrix#1#2{\null\,\vcenter{\normalbaselines
        \ialign{#1\crcr
                \mathstrut\crcr\noalign{\kern-\baselineskip}
                #2\crcr\mathstrut\crcr\noalign{\kern-\baselineskip}}}\,}

\def\sqmatrix{\gmatrix{\hfil$##$&\enspace\hfil$##$\hfil\enspace&$##$\hfil}}
\def\cdbl{\def\normalbaselines{\baselineskip20pt \lineskip3pt \lineskiplimit3pt }}

\def\sqcd{\cdbl\let\vagap\;\sqmatrix}

\newcount\arrowsize \arrowsize3

     \def\vagap{}

\def\Sd{\mathop{\rm Sd}\nolimits}

\def\BBord{\mathfrak{Bord}}
\def\ldf{\mathbb{L}}
\def\rdf{\mathbb{R}}

\def\tboxtimes{\mathbin{\tilde\boxtimes}}
\def\Map{{\cal M}{\rm ap}}

\mathcode`\:="603A 
\mathchardef\colon="303A 

\title{The geometric cobordism hypothesis}
\author{Daniel Grady \\ \href{http://www.gradydaniel.com/}{gradydaniel.com} \and Dmitri Pavlov \\ \href{https://dmitripavlov.org/}{dmitripavlov.org}} 
\date{Department of Mathematics and Statistics, Texas Tech University}

\begin{document}


\maketitle

\begin{abstract}
We prove a generalization of the cobordism hypothesis of Baez–Dolan and Hopkins–Lurie
for bordisms with arbitrary geometric structures, such as Riemannian or Lorentzian metrics, complex and symplectic structures,
smooth maps to a fixed target manifold,
principal bundles with connections, or geometric string structures.
Our methods rely on the locality property for fully extended functorial field theories established in arXiv:2011.01208,
reducing the problem to the special case of geometrically framed bordism categories.
As an application, we upgrade the classification of invertible fully extended topological field theories
by Bökstedt–Madsen and Schommer-Pries to nontopological field theories,
generalizing the work of Galatius–Madsen–Tillmann–Weiss to arbitrary geometric structures.
\end{abstract}

\tableofcontents

\section{Introduction}

The cobordism hypothesis is a beautiful statement about the nature of topological field theories.
It can be formulated by saying that the $d$-dimensional fully extended framed bordism category is the free symmetric monoidal $(\infty,d)$-category with duals on one object.
The intuition behind the cobordism hypothesis comes from the observation of Baez--Dolan \cite{BaezDolan} that it should be possible to decompose any $d$-dimensional smooth manifold by cutting it along higher codimensional submanifolds and the resulting cells in the decomposition should have a simple algebraic interpretation: either they are identity morphisms or they can be identified with the unit, counit, or triangle identity, implementing a duality in higher dimensions. 

Despite its conceptually clear meaning, a detailed proof of the cobordism hypothesis has remained somewhat elusive.
The seminal paper of Lurie \cite{Lurie.TFT} outlined a proof of a version of the topological cobordism hypothesis, however a complete proof has yet to emerge in the literature.
Part of the reason for this difficulty appears to be that precise formulations of both the algebraic and geometric structures involved in the statement have only recently appeared in the literature.
For example, a precise formulation of the $(\infty,d)$-category of bordisms was not available
until the work of Lurie \cite{Lurie.TFT} and Calaque--Scheimbauer \cite{CalaqueScheimbauer}. 

It is natural to ask to what extent the cobordism hypothesis holds in the geometric setting, i.e., for nontopological field theories.
Indeed, such a generalization is crucial in physics applications, where one typically does not expect realistic quantum field theories to be topological in nature.
However, given the state of the topological cobordism hypothesis, an analogue of the cobordism hypothesis in the geometric case presents a number of difficulties.
In particular, any proof of such a geometric cobordism hypothesis should provide a proof of the topological cobordism hypothesis,
since topological structures on cobordisms should be a special case of more general geometric structures. 

In Grady--Pavlov \cite{GradyPavlov}, we introduced a smooth variant $\BBord_d^{\cal S}$ of the symmetric monoidal $(\infty,d)$-category of bordisms.
This bordism category satisfies the following properties:
\begin{itemize}
\setlength\itemsep{0em}
\item It is extended: we have noninvertible $k$-morphisms, given by $k$-dimensional bordisms for $0\leq k\leq d$,
and invertible $k$-morphisms for $k>d$, given by diffeomorphisms as well as isotopies that move the source and target within a given bordism.
\item It is smooth: we have a notion of a smooth family of bordisms, which a functorial field theory is required to map to a smooth family of values.
\item It incorporates geometric structures on bordisms, including higher geometric structures with (higher) gauge automorphisms such as principal bundles with connection and bundle $n$-gerbes with connection.
\item It incorporates $d$-thin homotopies between geometric structures,
where a $d$-thin homotopy is a homotopy implemented using a smooth map of rank at most~$d$.
\end{itemize}
Geometric structures are encoded by simplicial presheaves ${\cal S}$ on the site $\FEmb_d$  (\cref{fembdef}).
The objects of $\FEmb_d$ are submersions with $d$-dimensional fibers and morphisms are fiberwise open embeddings.
Covering families are given by open covers on total spaces. Thus, geometric structures can be restricted along open embeddings, glued along open covers, and allow for smooth deformations parametrized by any~$\RR^n$.
In particular, we can encode the following structures in this formalism:
\begin{itemize}
\setlength\itemsep{0em}
\item Smooth maps to some fixed target manifold.
\item Riemannian and pseudo-Riemannian metrics, possibly with restrictions on Ricci or sectional curvature.
\item Conformal, complex, symplectic, contact, K\"ahler structures.
\item Principal $G$-bundles and vector bundles with connection.
\item Bundle $n$-gerbes with connection.
\item Foliations, possibly equipped with additional structures such as transversal metrics.
\item Geometric spin-c, string, fivebrane, and ninebrane structures.
\item Topological structures, such as orientation, spin, string, framing, etc. 
\end{itemize}

We proved the following locality property of corresponding field theories in Grady--Pavlov \cite[\ecref{EL-mainthm}]{GradyPavlov}.

\begin{theorem}\label{local}
Fix $d\geq 0$.
Let $\mathscr{C}$ be a smooth symmetric monoidal $(\infty,d)$-category (\cref{smcatglob}).
The smooth symmetric monoidal $(\infty,d)$-category of symmetric monoidal functors $\Funmon(\BBord_d^{{\cal S}},\mathscr{C})$
satisfies descent with respect to the target ${\cal S}\in \sPSh(\FEmb_d)$.
That is to say, the functor
$$\FFT_{d,\mathscr{C}}:\sPSh(\FEmb_d)^\op→\smcat_{∞,d},\qquad {\cal S}↦\FFT_{d,\mathscr{C}}({\cal S})≔\Funmon(\BBord_d^{{\cal S}},\mathscr{C})$$
is an ∞-sheaf, i.e., it preserves homotopy limits.
\end{theorem}

Using \cref{local}, we prove the following generalization of the cobordism hypothesis
(see \cref{mainthm} for a detailed statement and proof).
Informally, \cref{intro.mainthm} states that the smooth symmetric monoidal $(∞,d)$-category with duals
$\BBord_d^{\RR^d\times U\to U}$ is freely generated by a single $U$-family of objects (which are given by points, interpreted as 0-bordisms).

\begin{theorem}[Geometric framed cobordism hypothesis]\label{intro.mainthm}
Fix $d\geq 0$ and $U\in\cartsp$, i.e., $U≅\RR^n$ (see \cref{smoothcats}).
Let $\mathscr{C}$ be a smooth symmetric monoidal $(\infty,d)$-category with duals (\cref{smcatdual}).
We have an equivalence of smooth symmetric monoidal $(\infty,d)$-categories
$$\Funmon(\BBord_d^{\RR^d\times U\to U},\mathscr{C})\simeq \Map(U,\mathscr{C}^{\times}),$$
where $\mathscr{C}^{\times}$ denotes the invertible part (or the core) of $\mathscr{C}$ (\cref{coreadj})
and $\Map(U,-)$ denotes the powering over simplicial presheaves on $\cartsp$ (\cref{power.cat}). 
In particular, the left side is a smooth symmetric monoidal $∞$-groupoid.
\end{theorem}

We remark that we can ignore the sheaf structure over $\cartsp$ on both sides of the weak equivalence, which yields $\Map(U,\mathscr{C}^⨯)(\RR^0)\simeq \mathscr{C}^⨯(U)$.
The geometric structure appearing in \cref{intro.mainthm} is given by fiberwise open embeddings into $\RR^d$.
Concretely, a $k$-morphism in the bordism category $\BBord_d^{\RR^d\times U\to U}$ is a germ of a $k$-dimensional submanifold of a $d$-dimensional manifold.
The geometric structure encodes an immersion of this germ into $\RR^d$.

In the general case, the smooth symmetric monoidal $\infty$-groupoid $\mathscr{C}^{\times}$
must be enhanced to a presheaf of smooth symmetric monoidal $\infty$-groupoids $\mathscr{C}^⨯_d$ on the site $\FEmb_d$, using \cref{intro.mainthm}.
This is similar to the topological case, where $\mathscr{C}^{\times}$ admits an action of the group $\O(d)$ through the topological framed cobordism hypothesis.
The categories $\cartsp$ and $Γ$ can be omitted from the definition of $\mathscr{C}^⨯_d$ below if the corresponding structures are not needed,
in which case $\mathscr{C}^⨯_d$ can be taken to be simply a presheaf of simplicial sets.

\begin{definition}\label{refinement}
(See \cref{extcx} for a precise formulation.)
Fix $d\geq 0$ and a smooth symmetric monoidal $(∞,d)$-category with duals $\mathscr{C}\in \smcatdual_{\infty,d}$ (\cref{smcatdual}).
We define a presheaf of smooth symmetric monoidal $∞$-groupoids
$$\mathscr{C}^{\times}_d:\FEmb_d^\op\to \sPSh(\cartsp⨯Γ), \qquad\mathscr{C}^{\times}_d(\RR^d\times U\to U)≔\Funmon(\BBord_d^{\RR^d\times U\to U},\mathscr{C})^⨯.$$
\end{definition}

We remark that $\mathscr{C}^{\times}_d(\RR^d\times U\to U)\simeq \Map(U,\mathscr{C}^{\times})$ by \cref{intro.mainthm},
but this weak equivalence is {\it not\/} natural because $\mathscr{C}^{\times}$ is not functorial
with respect to morphisms in $\FEmb_d$.

We can now formulate the geometric cobordism hypothesis for arbitrary geometric structures (see \cref{mainthm.geometric} for a detailed statement and proof).

\begin{theorem}\label{intro.mainthm.geometric}
Fix $d\geq 0$,
$\mathscr{C}\in \smcatdual_{\infty,d}$ (\cref{smcatdual}),
and ${\cal S}\in \sPSh(\FEmb_d)$.
We have a natural equivalence of smooth symmetric monoidal $(\infty,d)$-categories
$$\Funmon(\BBord_d^{{\cal S}},\mathscr{C})\simeq \Map_{\FEmb_d}({\cal S},\mathscr{C}^{\times}_d).$$
The left side computes the smooth symmetric monoidal $(∞,d)$-category of functors.
The right side
is a smooth symmetric monoidal $∞$-groupoid
given by the derived mapping object of presheaves on $\FEmb_d$ (\cref{power.end.cat}).
In particular, the smooth symmetric monoidal $(∞,d)$-category on the left side is a smooth symmetric monoidal $∞$-groupoid.
\end{theorem}


We remark that forcing the smooth symmetric monoidal $(∞,d)$-category~$\mathscr{C}$ to have duals
does not create issues for practical applications.
In particular, quantum mechanics on infinite-dimensional Hilbert spaces
can be easily encoded by our formalism, even though
all dualizable Hilbert spaces are finite-dimensional.
Indeed,
any category~$C$ admits a faithful functor $C→L(C)$,
where $L(C)$ is the free symmetric monoidal category with duals on~$C$
and the functor $C→L(C)$ is the unit of the adjunction.
Applying this construction to the category of Hilbert spaces and bounded linear maps
produces a target category that has duals and is capable of encoding quantum mechanics on infinite-dimensional Hilbert spaces.
We also remark that our presentation of duals as a left Bousfield localization
yields a localization functor~$L$ that computes the free smooth symmetric monoidal $(\infty,d)$-category with duals
on a smooth $(\infty,d)$-category, which allows us to perform this trick in general.

In the special case of topological structures, we also obtain a proof (\cref{top.cob.hyp,lurie.formulation}) of the topological cobordism hypothesis of Baez--Dolan in the formulation of Hopkins–Lurie.
In contrast to Lurie, we crucially rely on \cref{local}, which allow us to reduce to the geometric framed case.
In the geometric framed case, the strategy of our proof of \cref{mainthm} resembles Lurie \cite[Proof of Theorem 2.4.6]{Lurie.TFT}. 
We also provide proofs of generalizations of Lurie \cite[Claim 3.4.12 and Claim 3.4.17]{Lurie.TFT} in \cref{bkequivalence},
reusing Lurie \cite[Proposition 3.4.20]{Lurie.TFT} in the proof.
Our methods are different in that we do not use framed generalized Morse functions directly.
Instead, we rely on the cutting and gluing lemmas established in Grady--Pavlov \cite{GradyPavlov}, which provide a convenient framework to access tools from Morse theory.

The right sides of \cref{refinement} and \cref{intro.mainthm.geometric} can often be computed in practice.
In \cref{refinement.deligne}, we consider the example (joint work of Stephan Stolz, Peter Teichner, and the second author)
of the target $(∞,d)$-category $\deloop^d A$, where $A$ is an abelian Lie group.
This category has a single $k$-morphism for $0\le k<d$ and its $d$-morphisms form the abelian Lie group~$A$.
As it turns out, $(\deloop^d A)^\times_d$ can be computed as the fiberwise Deligne complex valued in~$A$
and $\map({\cal S},\mathscr{C}^{\times}_d)$ is space of $A$-banded bundle $(d-1)$-gerbes with connections over~${\cal S}$.
This is a rigorous expression of the idea of higher parallel transport for bundle $(d-1)$-gerbes with connection
along $k$-manifolds with corners of arbitrary codimension, where $0\le k\le d$.

The abstract tool that allows us to compute $\mathscr{C}^⨯_d$ for this and many other choices of~$\mathscr{C}$ is the following simple consequence
of \cref{intro.mainthm},
which allows one to easily identify the action of $\O(d)$ (or rather, the equivalent homotopical information)
on the (smooth) $∞$-groupoid of fully dualizable objects in a (smooth) symmetric monoidal $(∞,d)$-category.
As before, $\cartsp$ and $Γ$ can be omitted from the statement below if the corresponding structures are not relevant for the problem at hand.

\begin{corollary}
\label{identify.refinement}
Fix $d\geq 0$ and a fibrant object $\mathscr{C}\in \smcatdual_{\infty,d}$ (\cref{smcatdual}).
Recall the simplicial presheaf $\mathscr{C}^⨯_d$ from \cref{refinement}.
Suppose 
$$F:\FEmb_d^\op\times \cartsp^\op\times \Gamma^\op\to\sset$$ 
is a simplicial presheaf that satisfies the descent condition with respect to $\FEmb_d$ and $f:F→\mathscr{C}^⨯_d$ is a morphism of simplicial presheaves.
Then $f$ is a local weak equivalence if and only if for any $(T=\RR^d⨯U→U)∈\FEmb_d$, the composition of induced maps
$$F(T→U)→\mathscr{C}^⨯_d(T→U)=\Funmon(\BBord_d^{T\to U},\mathscr{C})→\Map(U,\mathscr{C}^⨯)$$
is a local weak equivalence of simplicial presheaves on $\cartsp⨯Γ$.
\end{corollary}


We emphasize that in order to apply \cref{identify.refinement}
we do not ever need to compute spaces of field theories:
all that is needed is a guess for what such a space should be (organized into a simplicial presheaf~$F$),
together with a map~$f$ to actual field theories.
The geometric framed cobordism hypothesis (\cref{identify.refinement}) then allows us to identify $F(T→U)$ with $\mathscr{C}^⨯_d(T→U)$
by looking at $\Map(U,\mathscr{C}^⨯)$ only, which is often much easier to do in practice, since the simplicial presheaf
$\Map(U,\mathscr{C}^⨯)\in \sPSh(\cartsp\times \Gamma)$ makes no use of field theories whatsoever.

We also remark that
although fiberwise homotopy locally constant simplicial presheaves on $\FEmb_d$ are Quillen equivalent to presheaves on $\cartsp$ valued in
spaces equipped with an action of the group~$\O(d)$,
the equivalence is by no means trivial (see \cref{fiberwiseshape} and its dependencies)
and the model that uses simplicial presheaves on $\FEmb_d$
is particularly well-suited for the purposes of \cref{identify.refinement},
enabling applications such as \cref{refinement.deligne}.
For another example (principal $G$-bundles) see \cref{principal.bundles}.

Specializing to the case of invertible field theories, we obtain the following corollary,
which generalizes the work of Galatius--Madsen--Tillmann--Weiss \cite{GMTW},
Bökstedt--Madsen \cite{BokstedtMadsen},
and (in the extended case only) Schommer-Pries \cite{SchommerPries.ITFT}.
See \cref{invertible.tft,madsensphere,madsengstr} for details and proofs.

\begin{theorem}
Fix $d\geq 0$ and ${\cal S}\in \sPSh(\FEmb_d)$.
Set
$$\MT({\cal S}) ≔ \ldf\cK\BBord_d^{{\cal S}},$$
where $\ldf\cK$ (\cref{def.K.functor}) inverts all $k$-morphisms ($0<k≤d$) and objects in a smooth symmetric monoidal $(∞,d)$-category.
If
$\mathscr{C}\in \smcat_{\infty,d}$ is a fibrant object
in the model structure of \cref{sheaf.spectra.models},
i.e., a smooth group-like symmetric monoidal ∞-groupoid,
then
we have a natural weak equivalence
$$\ldf\cK\Funmon(\BBord_d^{\cal S},\mathscr{C})\simeq \rdf\hom(\MT({\cal S}),\ldf\cK(\mathscr{C})),$$
where $\hom(-,-)$ denotes the (derived) internal hom in $\PSh(\cartsp,\Sp_{\geq 0})_\loc$ (see \cref{sheaf.spectra.models}).

Furthermore, for any Lie group~$G$ with a representation $\rho:G→\GL(d)$ we have a weak equivalence $$\ldf\csp\MT((\RR^d→\RR^0)\hq \sm(-, G))≃Σ^d \MTcla G ,$$
where $\MTcla$ denotes the classical Madsen–Tillmann spectrum
and $\csp$ is the shape functor (\cref{shape.functor}).
Thus, $\MT$ is a geometric refinement of the Madsen--Tillmann spectrum (see \cref{madsengstr}).
\end{theorem}

We conclude by listing some key insights that enable the proofs of \cref{intro.mainthm} and \cref{intro.mainthm.geometric}
and distinguish our approach from the existing approaches to the cobordism hypothesis.
\begin{itemize}
\item We invoke the locality property (\cref{local})
to establish an Quillen adjunction $\BBord_d^{(-)}⊣(-)^⨯_d$, at the level of the Čech-local model structure on simplicial presheaves.  
This allows us to compute field theories $\FFT_{d,{\cal S},\mathscr{C}}=\Funmon(\BBord_d^{\cal S},\mathscr{C})$
as the derived mapping object $\Map_{\FEmb_d}({\cal S},\mathscr{C}^⨯_d)$.
Hence, we reduce the problem to computing the mapping object, eliminating $(∞,d)$-categories from the picture.
This contrasts with the existing approaches, e.g., Lurie \cite[Remark~2.4.20]{Lurie.TFT} deduces the topological case of the locality principle
from the cobordism hypothesis.
\item
We exploit the fact that in the geometric framed case,
multisimplices in $\BBord_d$ admit a simple geometric description:
bordisms in the $(∞,d)$-category of bordisms $\BBord_d^{\RR^d⨯U→U}$ are embedded fiberwise into $\RR^d\times U$ as an open subset.
\item We encode $d$-thin homotopies (and higher homotopies) for geometric structures on bordisms, corresponding to the “$∞$” in $(∞,d)$-category.
Here a $d$-thin homotopy is a homotopy given by a smooth map of rank at most~$d$.
This forces functorial field theories to be invariant under $d$-thin homotopies.
In particular, bordisms given by $d$-dimensional cubes become invertible,
corresponding to the fact that parallel transport maps in differential geometry and time evolution operators in physics are invertible.
Additionally, higher holonomy for bundle $(d-1)$-gerbes with connection is invariant under $d$-thin homotopies.
\item We encode the action of the orthogonal group on fully dualizable objects
using the structure of a presheaf on the site $\FEmb_d$ of $d$-dimensional manifolds and open embeddings.
This formulation immediately connects to the bordism category
and allows us to exploit the homotopy cocontinuity of $\BBord_d$ to reduce to the case of geometrically framed bordisms. 
\item The technical machinery of Grady–Pavlov \cite{GradyPavlov}
is used to establish the filtrations and pushout squares for the geometric framed bordism category in \cref{codescent}.
This provides an easy way to cut bordisms into elementary pieces.
\end{itemize}

\subsection{Previous work on functorial field theories}

The original definition of functorial field theories by Graeme Segal \cite{Segal.CFT.note, Segal.CFT} was already nontopological:
the 2-dimensional bordisms were equipped with a conformal structure.
See also an early survey by Gawędzki \cite{Gawedzki}. 

Conformal field theories proved to be very difficult to construct (but see Pickrell \cite{Pickrell} for a related example of a nonextended 2-dimensional field theory
in the setting of Riemannian metrics and Runkel–Szegedy \cite{RunkelSzegedy} for an example using volume forms on 2-manifolds),
so much of the subsequent work was focused on the case of topological field theories,
with the early works of Witten \cite{Witten.TQFT, Witten.Jones}, Kontsevich \cite{Kontsevich}, Atiyah \cite{Atiyah} followed by many others.
In particular, extended functorial field theories were proposed by Lawrence \cite{Lawrence} and Freed \cite{Freed.Ext} in the topological context.
We refer the reader to the survey of Baez–Lauda \cite{BaezLauda} for further details and discussion of many other important papers. 

Field theories where bordisms are equipped with a map to some fixed space
were considered by Segal \cite{Segal.Elliptic}, Freed–Quinn \cite{FreedQuinn}, Turaev \cite{Turaev}. Smooth functorial field theories were defined by Stolz--Teichner \cite{StolzTeichner.Elliptic, StolzTeichner.SUSY}.
Such theories allow for much more general geometric structures.
Berwick-Evans and the second author \cite{BerwickEvansPavlov} defined smooth 1-dimensional
topological field theories valued in vector spaces and proved that they are equivalent to vector bundles with connection.
Ludewig–Stoffel \cite{LudewigStoffel} proved a similar statement with bordism categories
having germs of paths for objects and the target category being sheaves of vector spaces with certain properties.

In a forthcoming work, Stolz, Teichner, and the second author \cite{PavlovStolzTeichner}
compute (using the results of this paper) the simplicial set of $d$-dimensional smooth fully extended functorial field theories
with target $\deloop^d A$, where $A$ denotes the representable presheaf of an abelian Lie group~$A$.
The resulting simplicial set is weakly equivalent to the simplicial set of $A$-banded bundle $(d-1)$-gerbes with connection. 

There is surprisingly little literature on nontopological extended field theories.
The first appearance of a smooth symmetric monoidal $(\infty,d)$-category of bordisms seems to be Grady--Pavlov \cite{GradyPavlov},
where we prove the locality property for extended field theories.
This bordism category was inspired by the 1-category introduced by Stolz--Teichner \cite{StolzTeichner.Elliptic, StolzTeichner.SUSY}
and allows for a very general class of geometric structures on bordisms. 

\subsection{Previous work on the topological cobordism hypothesis}

Considerable work has been done on various special cases of the {\it topological\/} cobordism hypothesis,
corresponding to the case when the space of geometric structures on a $d$-manifold is invariant under homotopy equivalences.
This allows for structures like orientations, spin and string structures, but excludes
structures with geometric data such as connections, differential forms, or Riemannian metrics.

The topological cobordism hypothesis was formulated by Baez–Dolan \cite[\S7]{BaezDolan},
where the framed $n$-dimensional topological case is stated as “Extended TQFT Hypothesis, Part~I”.
(The term “cobordism hypothesis” also appears in this paper, albeit only once.)
Later, Costello \cite{Costello} and Lurie \cite{Lurie.TFT} (in collaboration with Hopkins) refined its statement, replacing $n$-categories
with $(∞,n)$-categories, and allowing for bordisms with maps to arbitrary targets.
Baez–Lauda \cite{BaezLauda} survey a lot of the preceding history of functorial field theory.
See also the surveys by Bergner \cite{Bergner} and Freed \cite{Freed.CH}.

Costello \cite{Costello} proved a variant of the 2-dimensional topological cobordism hypothesis
(with the $(∞,1)$-category of chain complexes as a target),
where 2-bordisms cannot have connected components with empty source.
Schommer-Pries \cite{SchommerPries.Thesis} proved the 2-dimensional topological cobordism hypothesis
in the case when the target is an ordinary bicategory (i.e., a 2-truncated $(∞,2)$-category)
and the geometric structure is either trivial or given by orientation.
Lurie \cite{Lurie.TFT} gave an informal account of some ideas relating to the classification of topological field theories,
outlining a proof of a version of the topological cobordism hypothesis in \S\S3.1–3.5.
Eliashberg–Mishachev \cite{EliashbergMishachev} published a result that supersedes \S3.5.
On the other hand, a formal account of \S\S3.1–3.4 has not yet appeared.

Harpaz \cite{Harpaz} considered the 1-dimensional case (with an arbitrary symmetric monoidal $(∞,1)$-category as a target).
Ayala–Francis \cite{AyalaFrancis} gave a conditional proof of the framed topological cobordism (and tangle) hypothesis,
relying on a conjecture about factorization homology,
which is to be proved in a forthcoming series of papers by Ayala–Francis.
A forthcoming paper by Schommer-Pries \cite{SchommerPries.RTH} proves the topological relative tangle hypothesis, generalizing the topological cobordism hypothesis.

In the special case when the symmetric monoidal $(∞,d)$-category of values has only invertible objects and invertible $k$-morphisms for all $k>0$ (i.e., is a Picard $∞$-groupoid),
the topological cobordism hypothesis reduces to showing that the homotopy type of the fully extended bordism category can be computed as the corresponding Madsen–Tillmann spectrum.
This was proved by
Bökstedt–Madsen \cite{BokstedtMadsen} in the case of multiple $(∞,d)$-categories
and Schommer-Pries \cite{SchommerPries.ITFT} in the case of globular $(∞,d)$-categories,
generalizing the work of Galatius--Madsen--Tillmann--Weiss \cite{GMTW} in the nonextended case.

\subsection{Acknowledgments}

We thank Peter Teichner and Stephan Stolz
for launching the program that eventually led to this paper
and providing the second author with a rich collection of insights
without which this work would not be possible,
as well as creating the environment through which he could absorb these insights.
The second author also thanks Arthur Bartels, Chenchang Zhu, and Ulrich Bunke for supporting his work.

We thank Hisham Sati for providing the first author with a wealth of knowledge connecting functorial field theories and generalized cohomology theories with mathematical physics.
The first author also thanks Domenico Fiorenza for useful discussions.

We thank Chris Schommer-Pries for useful discussions about the cobordism hypothesis.
We thank Daniel Brügmann for a careful reading of this paper, including the more technical parts, and for suggesting many improvements.
We thank Yonatan Harpaz for a lot of feedback on an earlier draft of this paper, in particular on \cref{handles.duals}.

This work would not exist without the $n$Lab.
We express our tremendous gratitude to its authors
and personally to Urs Schreiber
for selflessly contributing vast amounts of high-quality material.

\section{Smooth $(\infty,d)$-categories with duals}

In this section, we introduce smooth $(\infty,d)$-categories with duals. We will closely follow Grady--Pavlov \cite{GradyPavlov} for the presentation of smooth $(\infty,d)$-categories (without duals). Much of the material in \cref{smoothcats} was taken directly from Grady--Pavlov \cite{GradyPavlov} and we only include it here to keep the present work somewhat self-contained. 

\subsection{Smooth $(\infty,d)$-categories}\label{smoothcats}

In Grady–Pavlov \cite{GradyPavlov}, we presented the $(\infty,1)$-category of smooth symmetric monoidal $(\infty,d)$-categories
as a left Bousfield localization of a simplicial presheaf category, equipped with the injective model structure.
In this section, we review this construction.
We refer the reader to Grady–Pavlov \cite[\ecref{EL-symminfn}]{GradyPavlov} for further details and exposition. 

Let $\Delta^{\times d}$ denote the $d$-fold product of the simplex category.
Objects are called \emph{multisimplices} and are denoted by
$${\bf m}=([m_1],[m_2],\ldots,[m_d])\in \Delta^{\times d}.$$ 
In particular, ${\bf0}$ denotes the multisimplex $([0],\ldots,[0])$.
Let $\Gamma$ denote the opposite category of pointed finite sets introduced by Segal.
Objects are pointed finite sets of the form $\{*,1,\ldots, n\}$, where $n\in \NN$.
Let $\cartsp$ be the category with objects open subsets of $\RR^d$ that are diffeomorphic to $\RR^d$.
Morphisms are smooth maps between such open subsets. 

We consider the category $$\sPSh(\cartsp⨯Γ⨯\Delta^{\times d})=\Fun((\cartsp⨯Γ⨯\Delta^{\times d})^\op,\sset)\eqlabel{bigspshcat}$$ of simplicial presheaves on the threefold product $\cartsp⨯Γ⨯\Delta^{\times d}$.
\begin{definition}\label{smcatuple}
The category \cref{bigspshcat} admits a model structure by Grady–Pavlov \cite[\ecref{EL-multiple.model.structure}]{GradyPavlov}, given by performing a left Bousfield localization of the injective model structure. We write this model category as 
$$\smcatuple_{∞,d}≔\sPSh(\cartsp⨯Γ⨯\Delta^{\times d})_\uple$$
and refer to it as the {\it multiple injective model structure}.
\end{definition}

The fibrant objects in this model structure are injectively fibrant objects that are $d$-fold complete Segal spaces (in every direction),
satisfy Segal's special $\Gamma$-condition,
and satisfy homotopy descent with respect to good open covers of cartesian spaces.

\begin{remark}
The fibrant objects in the above model category are the higher categorical analogue of double categories, where different composition directions are put on equal footing.
In order to have the correct generalization of a bicategory, fibrant objects must satisfy an additional condition, called \emph{globularity}.
The fibrant objects satisfying this condition have a preferred ordering of composition directions,
which serves as a higher categorical analogue of vertical and horizontal composition in a bicategory.
\end{remark}

We encode the \emph{globular condition} through a further left Bousfield localization.

\begin{definition}\label{smcatglob}
The model category of \cref{smcatuple} admits a further left Bousfield localization, which incorporates globularity (Grady–Pavlov \cite[\ecref{EL-globular.model.structure}]{GradyPavlov}).
We write this model category as 
$$\smcatglob_{∞,d}=\smcat_{∞,d}≔\sPSh(\cartsp⨯Γ⨯\Delta^{\times d})_{\glob}$$
and refer to it as the {\it globular injective model structure}. 
A \emph{smooth symmetric monoidal $(∞,d)$-category} is defined as a fibrant object in $\smcatglob_{∞,d}$.
\end{definition}

Since this condition is less well known than the Segal condition and completeness,
we recall
the local objects in the resulting localization
(details can be found in Grady--Pavlov \cite{GradyPavlov}).
The local objects in this left Bousfield localization are multisimplicial objects~$X$ such that $X_{0}$,
which we interpret as an object in $(d-1)$-fold simplicial objects, is homotopy constant and $X_k$ is a local object in $(d-1)$-fold simplicial spaces for all $k∈Δ$.
For example, for $d=2$, the locality condition boils down to forcing the degeneracy maps $X_{0,0}\to X_{0,b}$ to be weak equivalences.
This makes all vertical morphisms invertible.

\subsection{Functors and the core}

We begin with some notation, which will be used throughout the paper.

\begin{notation}\label{partials} 
Fix $R=L⨯R/L$ a product of two categories and $X\in \sPSh(R)$. 
For any $l\in L$, we can evaluate $X$ on $l$ to obtain a simplicial presheaf on the remaining factor~$R/L$.
We denote the corresponding partial evaluation of $X$ on an object $l\in L$ as $X(l)\in \sPSh(R/L)$.
\end{notation}

A frequently used example will be the following. Let $S\subset \{1,\ldots,d\}$ be a subset and let ${\bf m}\in \Delta^{S}$ be a multisimplex, i.e., a functor ${\bf m}:S\to \Delta$, where $S$ is regarded as a discrete category. Then for a simplicial presheaf $X\in \sPSh(\cartsp\times \Gamma\times \Delta^{\times d})$, the partial evaluation of $X$ on ${\bf m}\in \Delta^{S}$ yields an object
$$X({\bf m})\in \sPSh(\cartsp\times \Gamma\times \Delta^{\{1,\ldots,d\}\setminus S}).$$

To avoid confusion, we will always make the category explicit when applying partial evaluation. For example, if $X\in \sPSh(\cartsp\times \Gamma\times \Delta^{\times d})$ and we fix ${\bf m}\in \Delta^{\times i}$, the notation $X({\bf m})$ should be understood as a simplicial presheaf on the remaining variables in $\cartsp\times \Gamma\times \Delta^{\times (d-i)}$.

\begin{definition}\label{tensoringmultisimp}
Let $S\subset \{1,\ldots,d\}$ and let ${\bf m}\in \Delta^S$ be a multisimplex. Denote by 
$$j_{\Delta^{S}}:\Delta^{ S}\into \PSh(\Delta^{S})\into \sPSh(\Delta^{S})\to \sPSh(\Delta^{\times d})$$
composition of the Yoneda embedding with restriction along the projection $\Delta^{\times d}\to \Delta^S$. The partial evaluation functor (see \cref{partials})
$$\sPSh(\Delta^{\times d})\to \sPSh(\Delta^{\{1,\ldots,d\}\setminus S}), \qquad X\mapsto X({\bf m})$$
admits a left adjoint, given by externally tensoring $j_{\Delta^S}({\bf m})$ with an object in $\sPSh(\Delta^{\{1,\ldots,d\}\setminus S})$. We denote this tensoring by
$$Y\mapsto j_{\Delta^{\times S}}({\bf m}_S)\boxtimes Y.$$ 
Explicitly, the value of $j_{\Delta^{ S}}({\bf m})\boxtimes Y$ at ${\bf n}\in \Delta^{\times d}$ is given by 
$$(j_{\Delta^{S}}({\bf m})\boxtimes Y)({\bf n})=\coprod_{\hom({\bf n}_{ S},{\bf m})} Y({\bf n}_{\{1,\ldots,d\}\setminus S}),$$
where ${\bf n}_{S}\in \Delta^{S}$ denotes the multisimplex given by throwing away all simplices that are not indexed by an element of $S$. This functor is the left Kan extension of the functor $\Delta^{\{1,\ldots,d\}\setminus S}\to \Delta^{\times d}$ that inserts $[m_s]={\bf m}(s)$ into the $s$ slot.
\end{definition}

In the statement of \cref{mainthm}, we will need both the internal hom in $\smcatdual_{\infty,d}$ and the powering over $\sPSh(\cartsp)$.
To emphasize the categorical and symmetric monoidal structure, we use the following notation for the internal hom in $\smcatdual_{\infty,d}$.

\begin{definition}
\label{ihom.cat}
(Grady–Pavlov \cite[\ecref{EL-gamma.smash.product}]{GradyPavlov}.)
Fix $d\geq 0$.
We denote by $\Funmon(-,-)$ the internal hom in $\sPSh(\cartsp⨯Γ⨯Δ^{⨯d})$,
which is adjoint to the symmetric monoidal structure on $\sPSh(\cartsp⨯Γ⨯Δ^{⨯d})$
given by the Day convolution with respect to the smash product on~$Γ$
and cartesian product on $\cartsp⨯Δ^{⨯d}$.
\end{definition}

\begin{remark}
The multiple model structure on $\sPSh(\cartsp⨯Γ⨯Δ^{⨯d})$ (Grady–Pavlov \cite[\ecref{EL-multiple.model.structure}]{GradyPavlov})
is a monoidal model structure (Grady–Pavlov \cite[\ecref{EL-functor.categories}]{GradyPavlov}).
This means that if $\mathscr{C}$ is a fibrant object in this model structure,
then $\Funmon(X,\mathscr{C})$ computes the correct derived internal hom.
Throughout the paper, $\Funmon(-,-)$ will always be derived in the multiple model structure.
In fact, the second argument will always be fibrant, which eliminates the need to derive $\Funmon$.

The globular model structure on $\sPSh(\cartsp⨯Γ⨯Δ^{⨯d})$ (Grady–Pavlov \cite[\ecref{EL-globular.model.structure}]{GradyPavlov})
is obtained by further localizing the multiple model structure $\smcatuple_{∞,d}$.
The resulting model structure is not a monoidal model structure (Grady–Pavlov \cite[\ecref{EL-globular.not.monoidal}]{GradyPavlov}).

However, in our case, computing the derived internal homs $\Funmon(X,\mathscr{C})$
in the multiple model structure
results (\cref{mainthm.geometric}) in simplicial presheaves that are constant in the direction of $Δ^{⨯d}$,
i.e., the relevant multiple $(∞,d)$-categories are $∞$-groupoids.
Such derived internal homs in the multiple model structure
also compute derived internal homs in the globular model structure. 
\end{remark}

We now turn to the various powerings that will be used. 

\begin{definition}\label{power.cat}
Fix some symmetric monoidal categories $L$ and $R/L$ and consider the category $R=L⨯R/L$.
Equip the categories of simplicial presheaves $\sPSh(L)$ and $\sPSh(R)$ with the Day convolution closed monoidal structure.
We define the powering of $Y\in \sPSh(R)$ by $X\in \sPSh(L)$ as 
$$\Map(X,Y)=\hom(p^*X,Y)\in \sPSh(R),$$
where $p:R=L⨯R/L→L$ is the projection functor.
In particular, if $R=L$, then $\Map(X,Y)$ is the internal hom from~$X$ to~$Y$.
For objects $X,Y\in \sPSh(R)$, we reserve the notation 
$$\map(X,Y)\in \sset,$$
with roman font instead of calligraphic, for the usual simplicial enrichment of simplicial presheaves over simplicial sets.  
\end{definition}

In particular, \cref{power.cat} is applicable when $L$ is a factor of $\cartsp⨯Γ⨯Δ^{⨯d}$, using the monoidal structure of \cref{ihom.cat}.
When $\sPSh(R)$ is equipped with the multiple injective model structure, obtained by localizing the injective model structure at those morphisms
in Grady–Pavlov \cite[\ecref{EL-multiple.model.structure}]{GradyPavlov}
that are relevant for the factors present in~$R$, the functor $p^*$ is trivially left Quillen.
Hence, the powering is a Quillen bifunctor.

We prove the main theorem \cref{mainthm.geometric} for the full enrichment in smooth symmetric monoidal $(\infty,d)$-categories.
As a consequence, the mapping object on the right side of the equivalence must be suitably defined.
Having in mind $R=\FEmb_d⨯\cartsp⨯Γ$, $M=\FEmb_d$, $L/M=*$ (as used in \cref{mainthm.geometric}), we make the following definition.

\begin{definition}\label{power.end.cat}
Fix some symmetric monoidal categories $M$, $L/M$, $R/L$, and consider the categories $L=M⨯L/M$ and $R=L⨯R/L=M⨯R/M$, where $R/M=L/M⨯R/L$.
Let $Y\in \sPSh(R)$ and $X\in \sPSh(L)$.
We define the mapping object 
$$\Map_M(X,Y)≔\int_{m\in M}\Map(X(m),Y(m))\in \sPSh(R/M),$$
where $X(m)$ denotes the partial evaluation of \cref{partials}
and $\Map$ denotes the powering of \cref{power.cat}.
For $M=\ast$, we have $\Map_M(X,Y)=\Map(X,Y)$, while for $M=L=R$, we have $\Map_M(X,Y)=\map(X,Y)$. 
\end{definition}

We will also use following functors, which promote an $(\infty,d-1)$ category to an $(\infty,d)$-category. 

\begin{definition}
\label{evconst}
Let $S\subset \{1,\ldots,d\}$ and let $i:\Delta^{\{1,\ldots,d\}\setminus S}\to \Delta^{\times d}$
be the functor that inserts $[0]$ in the simplicial directions indexed by $S$.
Morally, $S$ indexes the directions in which we want to discard noninvertible morphisms.
Restriction along $i$ induces the corresponding functor 
$${\rm ev}_{S}:\sPSh(\cartsp⨯Γ⨯\Delta^{\times d})\to \sPSh(\cartsp⨯Γ⨯\Delta^{\{1,\ldots,d\}\setminus S}),$$ 
which evaluates at ${\bf 0}\in \Delta^{S}$.
The left adjoint ${\rm c}_S\dashv {\rm ev}_{S}$ sends an object $X$ to the simplicial presheaf that is constant on $\Delta^{S}$.
\end{definition}

\begin{lemma}\label{evconstadj}
The adjunction ${\rm c}_S\dashv {\rm ev}_S$ descends to a Quillen adjunction at the level of local injective model structures.
Furthermore ${\rm c}_S$ is itself a right Quillen functor when considered with the multiple projective model structures.
For the globular projective model structure, ${\rm c}_S$ is right Quillen when $S$ is of the form $S=\{i,\dots, d\}$, for some $1\leq i \leq d$. 
\end{lemma}

\begin{proof}
The adjunction is clearly Quillen at the level of the injective model structures.
Since both model categories are left proper, we need only check that the right adjoint ${\rm ev}_S$  preserves locally fibrant objects
(see, for example, Lurie \cite[Corollary A.3.7.2]{Lurie.HTT}).
But this is clear, since local objects in the domain of ${\rm ev}_S$ are levelwise (for each ${\bf m}\in \Delta^{\{1,\ldots, d\}\setminus S}$) local objects on the right.

Finally, ${\rm c}_S$ is a right Quillen functor for local projective model structures:
for the projective model structures this is true by definition,
and for the local projective model structures,
it suffices to show that ${\rm c}_S$ preserves local objects. This is clearly true in the multiple model structure. For the globular model structure, ${\rm c}_S$ preserves local objects when $S$ is of the form $S=\{i,\dots, d\}\subset \{1,\ldots,d\}$.
\end{proof}

\begin{definition}
\label{coreadj}
Given a smooth symmetric monoidal $(∞,d)$-category~$\mathscr{C}$,
we denote by $\mathscr{C}^⨯$ the value on~$\mathscr{C}$ of the right derived functor of
$${\rm ev}_{\{1,…,d\}}:\sPSh(\cartsp⨯Γ⨯Δ^{⨯d})\to \sPSh(\cartsp⨯Γ).$$
\end{definition}

Thus, $\mathscr{C}^⨯$ denotes the smooth symmetric monoidal $∞$-groupoid of invertible morphisms in~$\mathscr{C}$, i.e., the core of~$\mathscr{C}$.

\subsection{Duals}

In \cref{smoothcats}, we constructed a model category in which the fibrant objects are smooth symmetric monoidal $(\infty,d)$-categories. 
However, in order to prove (or even state) the cobordism hypothesis, we must be able to say what a $(\infty,d)$-category \emph{with duals} is. 
In this section, we introduce yet a further localized model structure on $\sPSh(\cartsp⨯Γ⨯\Delta^{\times d})$ that encodes duals. 
The fibrant objects in this model structure will be $(\infty,d)$-categories with duals.  

Let ${\sf C}$ be a bicategory. We can think of 1-morphisms in ${\sf C}$ as functors and 2-morphisms in ${\sf C}$ as natural transformations. Hence, it makes sense to talk about an adjunction in ${\sf C}$.  Given two 1-morphisms $f:x\to y$ and $g:y\to x$, and a 2-morphism $u:\id_{x}\to gf$, we call $u$ the \emph{unit of an adjunction} if there is $v:fg\to \id_y$ such that the triangle identities 
$$(f\cong f\id_x\to fgf\to \id_y f\cong f)=\id_f \qquad \text{and} \qquad (g\cong \id_xg\to gfg\to g\id_y\cong g)=\id_g$$
hold in ${\sf C}$.

Given a $d$-fold complete Segal space $X$, we can form the homotopy 2-category $\Ho_2 X$ as follows.
The set of objects of $\Ho_2 X$ is the set of vertices in $X({\bf 0})$, where ${\bf 0}\in \Delta^{\times d}$ is the zero multisimplex.
Consider the subset $S=\{1,3,\ldots,d\}\subset \{1,2,\ldots, d\}$ and let ${\bf 1}_S\in \Delta^S$ be the multisimplex defined by $[1_s]=[0]$ if $s\in S, s\neq 1$ and $[1_1]=[1]$.
Similarly, let ${\bf 0}_S\in \Delta^{S}$ be the multisimplex defines by $[0_s]=[0]$ for all $s\in S$. 
For a pair of objects $x,y\in X({\bf 0})$, the hom categories of the homotopy 2-category are defined as the homotopy category of the 1-fold Segal space given by the pullback:
$$\Ho_2 X(x,y)≔\Ho\left(\{x\}\times_{X({\bf 0}_S)}X({\bf 1}_S)\times_{X({\bf 0}_S)}\times \{y\}\right).$$
This pullback is a homotopy pullback because $X$ is a fibrant object.

Following Lurie \cite{Lurie.TFT}, we make the following definition.
\begin{definition}\label{adjunction}
Let $k\geq 2$. Let $\mathscr{C}$ be an $(\infty,k)$-category and let $\Ho_2\mathscr{C}$ denote its homotopy 2-category. We say that $\mathscr{C}$ admits \emph{adjoints for 1-morphisms} if for all 1-morphisms $f:x\to y$ in $\Ho_2\mathscr{C}$, there is $g:y\to x$ and $u:\id_x\to gf$ such that $u$ is the unit of an adjunction. 
\end{definition}

\begin{definition}
Let $\mathscr{C}\in \smcatglob_{\infty,d}$ (\cref{smcatglob}) be a fibrant object.
We say that $\mathscr{C}$ has adjoints for $k$-morphisms if for every fixed $U\in \cartsp$, for all ${\bf m}\in \Delta^{\times k}$, the $(\infty,d-k)$-category $\mathscr{C}(U,\langle 1\rangle,{\bf m})$ has adjoints for 1-morphisms.
\end{definition}

\begin{definition}
Let $\mathscr{C}\in \smcatglob_{\infty,d}$ (\cref{smcatglob}) be a fibrant object.
We say that $\mathscr{C}$ has duals for objects if the monoidal 1-category $\Ho\mathscr{C}$ has duals for objects. 
\end{definition}

\cref{adjunction} may appear to be incomplete in that we do not consider homotopy coherent adjunctions. However, it turns out that any adjunction in the homotopy 2-category can be lifted to a homotopy coherent adjunction by Riehl–Verity \cite[Theorem 4.3.9]{RiehlVerity}.
Moreover the space of such lifts extending the left adjoint in the adjunction is contractible by Riehl–Verity \cite[Theorem 4.4.18]{RiehlVerity}.

Following Riehl–Verity \cite{RiehlVerity}, we write $\underline{\rm Adj}$ for the simplicial category given by taking nerves of hom-categories of the free adjunction 2-category considered by Schanuel–Street \cite{SchanuelStreet}.
An explicit presentation of this 2-category is constructed in Riehl–Verity \cite[\S3.1]{RiehlVerity}, as the 2-category with two objects, $+$ and $-$, and categories of morphisms are given by \emph{strictly undulating squiggles}.
\begin{center}
\begin{tikzpicture}
\draw[dashed] (0,0) -- (3.5,0);
\draw[dashed] (0,-.5) -- (3.5,-.5);
\draw[dashed] (0,-1) -- (3.5,-1);
\draw[dashed] (0,-1.5) -- (3.5,-1.5);
\draw[dashed] (0,-2) -- (3.5,-2);
\draw[dashed] (0,-2.5) -- (3.5,-2.5);
\draw[dashed] (0,-3) -- (3.5,-3);

\node at (-.3,.25) {$-$};
\node at (-.3,-.25) {$1$};
\node at (-.3,-.75) {$2$};
\node at (-.3,-1.25) {$3$};
\node at (-.3,-1.75) {$4$};
\node at (-.3,-2.25) {$5$};
\node at (-.3,-2.75) {$6$};
\node at (-.3,-3.25) {$+$};

\node at (3.8,0) {$0$};
\node at (3.8,-.5) {$1$};
\node at (3.8,-1) {$2$};
\node at (3.8,-1.5) {$3$};
\node at (3.8,-2) {$4$};
\node at (3.8,-2.5) {$5$};
\node at (3.8,-3) {$6$};

\node at (-.1,-3.75) {$+$};
\node at (.3,-3.75) {$3$};
\node at (.7,-3.75) {$6$};
\node at (1.1,-3.75) {$2$};
\node at (1.5,-3.75) {$3$};
\node at (1.9,-3.75) {$1$};
\node at (2.3,-3.75) {$+$};
\node at (2.7,-3.75) {$4$};
\node at (3.1,-3.75) {$5$};
\node at (3.5,-3.75) {$-$};

\begin{scope}[xshift=-5pt]
\draw[line width=.5mm] (.3,-3.25) -- (.3,-1.25);
\draw[line width=.5mm] (.7,-2.75) -- (.7,-1.25);
\draw[line width=.5mm] (1.1,-2.75) -- (1.1,-.75);
\draw[line width=.5mm] (1.5,-1.25) -- (1.5,-.75);
\draw[line width=.5mm] (1.9,-1.25) -- (1.9,-.25);
\draw[line width=.5mm] (2.3,-3.25) -- (2.3,-.25);
\draw[line width=.5mm] (2.7,-3.25) -- (2.7,-1.75);
\draw[line width=.5mm] (3.1,-2.25) -- (3.1,-1.75);
\draw[line width=.5mm] (3.5,-2.25) -- (3.5,.25);

\draw[line width=.5mm] (.3,-1.25) to[bend left= 80] (.7,-1.25);
\draw[line width=.5mm] (.7,-2.75) to[bend right= 80] (1.1,-2.75);
\draw[line width=.5mm] (1.1,-.75) to[bend left= 80] (1.5,-.75);
\draw[line width=.5mm] (1.5,-1.25) to[bend right= 80] (1.9,-1.25);
\draw[line width=.5mm] (1.9,-.25) to[bend left= 80] (2.3,-.25);
\draw[line width=.5mm] (2.7,-1.75) to[bend left= 80] (3.1,-1.75);
\draw[line width=.5mm] (3.1,-2.25) to[bend right= 80] (3.5,-2.25);

\draw[line width=.5mm, dashed] (2.3,-3.25) to[bend right= 80] (2.7,-3.25);
\draw[line width=.5mm, dashed] (.1,-3.4) to[bend right= 40] (.3,-3.25);
\draw[line width=.5mm, dashed] (3.5,.25) to[bend left= 40] (3.7,.4);
\end{scope}
\end{tikzpicture}
\end{center}
Taking the nerve of the hom categories of $\underline{\rm Adj}$ gives a simplicial computad, i.e., a cofibrant object in the Dwyer--Kan model structure on simplicial categories.
Simplicial computads are characterized by the property that every $n$-arrow
can be expressed uniquely as a composition of elementary arrows, called atomic $n$-arrows, which are closed under degeneracy maps. The 2-category $\underline{\rm Adj}$ has an atomic $0$-arrow $f\in \underline{\rm Adj}(+,-)$, which we call the \emph{left adjoint}. It also has an atomic $0$-arrow $g\in \underline{\rm Adj}(-,+)$, called the \emph{right adjoint}, and two atomic $1$-arrows $\epsilon\in \underline{\rm Adj}(-,-)$ and $\eta\in \underline{\rm Adj}(+,+)$, called the \emph{counit and unit} respectively. Finally, it has atomic 2-arrows $\alpha$ and $\beta$, which implement the \emph{triangle identities}. 

\begin{definition}
\label{subcomputads}
We write $\underline{\epsilon}\into \underline{\rm Adj}$ for the simplicial subcomputad generated by the counit $\epsilon$, $f$ and $g$.
Similarly, write $\underline{f}\into \underline{\rm Adj}$ for the simplicial subcomputad generated by the left adjoint $f$.
We write $\underline{\eta}\into \underline{\rm Adj}$ for the simplicial subcomputad generated by the unit $\eta$, $f$ and $g$.
\end{definition}  

\begin{remark}\label{nerve2cat}
Recall that there is a canonical functor 
$${\cal N}:\sset_{J}\text{-}{\sf Cat}\to \PSh(\Delta^{\times 2})\into \sPSh(\Delta^{\times 2})$$
from the category of simplicial categories, enriched in the Joyal model structure, to bisimplicial sets. The functor ${\cal N}$ takes a small simplicial category, regards it as a simplicial object in categories, and takes the nerve levelwise. The second inclusion regards a bisimplicial set as a bisimplicial space. The composition ${\cal N}$ is a right Quillen equivalence when $\sPSh(\Delta^{\times 2})$ is equipped with the globular Barwick model structure.
\end{remark}

Since we are working with the Segal space formalism, we will need to regard $\underline{\rm Adj}$, $\underline{\epsilon}$, $\underline{\eta}$ and $\underline{f}$ as $d$-fold Segal spaces. Moreover, we will need place these categories in various multisimplicial degrees in order to encode adjoints for higher morphisms. We now use the tensoring in \cref{tensoringmultisimp} to define variants of these simplicial categories that are simplicial presheaves on $\Delta^{\times d}$. 

\begin{definition}\label{dcatadjoints}
Fix $d\geq 0$, $1\leq k\leq d-1$, and ${\bf m}\in \Delta^{\times (k-1)}$.
Let $p_{2}: \Delta^{\times (d-k+1)}\to \Delta^{\times 2}$ denote the projection onto the first $2$ factors of $\Delta^{\times (d-k+1)}$. Consider the composite functor
$$j_{\Delta^{\times (k-1)}}({\bf m})\boxtimes  p_2^*{\cal N}:\sset_{J}\text{-}{\sf Cat}\to  \sPSh(\Delta^{\times 2})\into \sPSh(\Delta^{\times(d-k+1)})\to \sPSh(\Delta^{\times d} ) ,$$
where ${\cal N}$ is the nerve functor (\cref{nerve2cat}) and the tensoring is the left adjoint to partial evaluation described in \cref{tensoringmultisimp}.
We apply 
$j_{\Delta^{\times (k-1)}}({\bf m})\boxtimes  p_2^*{\cal N}$
to the simplicial categories $\underline{\rm Adj}$, $\underline{\eta}$, $\underline{\epsilon}$ and $\underline{f}$.
The resulting objects encode adjoints for $k$-morphisms.
We denote the resulting $(\infty,d)$-categories as  
$${\rm Adj}_{{\bf m}}, \eta_{ \bf m} ,\epsilon_{\bf m}, f_{\bf m},$$ 
respectively.
\end{definition}

\begin{definition}[Localizing maps for the model structure with duals for morphisms]\label{dualmorph}
Fix $d\geq 0$, $1\leq k\leq d-1$ and ${\bf m}\in \Delta^{\times k-1}$.
\begin{enumerate}
\item[(vi)] We consider the morphisms 
$$f_{\bf m}\to {\rm Adj}_{\bf m},$$
obtained by applying the functor $j_{\Delta^{\times (k-1)}}({\bf m})\boxtimes  p_2^*{\cal N}$ in \cref{dcatadjoints} to the morphism of simplicial categories $\underline{f}\into \underline{\rm Adj}$. 
\end{enumerate}
More concretely, the functor first converts the inclusion $\underline{f}\into \underline{\rm Adj}$ to a morphism of 2-fold Segal spaces, via the nerve. It then converts it to a $(d-k+1)$-fold Segal space by taking identity maps in the last $(d-k-1)$ simplicial directions. Finally, it converts it to a $d$-fold Segal space by tensoring with ${\bf m}\in \Delta^{\times (k-1)}$. 
\end{definition}

\begin{remark}
We will often consider maps from $f_{\bf m}$, $\eta_{\bf m}$ and ${\rm Adj}_{\bf m}$ to some smooth symmetric monoidal $(\infty,d)$-category $\mathscr{C}$. In practice, we want to remember the smooth and symmetric monoidal structure in the mapping object. Hence, we will work almost exclusively with the powering of $\sPSh(\cartsp\times \Gamma\times \Delta^{\times d})$ over $\sPSh(\Delta^{\times d})$, as opposed to just the $\infty$-groupoid of maps (\cref{power.cat}).
\end{remark}
 
The following lemma will be helpful in translating between simplicial categories enriched in the Joyal model structure and the Segal space formalism. 

\begin{lemma}\label{simpltosegal}
Given $d≥2$, $1\leq k\leq d-1$, $V\in \cartsp$, $\langle \ell\rangle\in \Gamma$, ${\bf m}\in \Delta^{\times k-1}$, consider the bisimplicial object in simplicial sets
$$\mathscr{D}≔\mathscr{C}(V,\langle\ell\rangle,{\bf m}\boxtimes(-,-)\boxtimes{\bf0}^{d-k-1}).$$

There is a simplicial category ${\sf C}$ such that ${\cal N}({\sf C})\simeq \mathscr{D}$,
where ${\cal N}$ is the nerve functor (\cref{nerve2cat}).
Moreover, the nerve functor induces an equivalence on derived mapping spaces
\begin{equation}\label{simpltosegalmap}
{\cal N}:\map(\underline{\rm Adj}, {\sf C})\overset{\simeq}{\to}\map({\cal N}(\underline{\rm Adj}),\mathscr{D})\simeq \map({\rm Adj}_{\bf m},\mathscr{C}(V,\langle \ell\rangle)).
\end{equation}
\end{lemma}
\begin{proof}
Since $\mathscr{C}$ is assumed to be a local object, the bisimplicial space $\mathscr{D}$ is a 2-fold complete Segal space that satisfies the globular condition. 
It follows from Joyal--Tierney \cite[Theorem 4.11]{JoyalTierney} that $\mathscr{D}$ is equivalent (in the Barwick model structure) to the 2-fold simplicial set $\mathscr{D}^0$ given by taking its zeroth row.  
By globularity, $\mathscr{D}^0$ is equivalent to the nerve (see \cref{nerve2cat}) of the simplicial category ${\sf C}$.

One subtle point is that the model category $\sset_{J}\text{-}\Cat$ is not a simplicial model category because the Joyal model structure is not simplicial.
However, the derived mapping space can be computed by taking the simplicial subset on invertible morphisms in the simplicial enrichment. 
By Riehl–Verity \cite[Lemma 4.4.6]{RiehlVerity}, the simplicial set $\map(\underline{\rm Adj},{\sf C})$ is in fact a Kan complex, hence the mapping space on the left in \cref{simpltosegalmap} is already derived. 

With these observations, the first equivalence in \cref{simpltosegalmap} follows immediately from the fact that the nerve is a right Quillen equivalence (see \cref{nerve2cat}). Since $\mathscr{C}$ is assumed to be fibrant, the second mapping space in \cref{simpltosegalmap} is already derived. Hence the second equivalence is the canonical isomorphism induced by the adjunction described in the paragraph following  \cref{nerve2cat}.
\end{proof}

We also need to treat dual objects in a similar manner (morally corresponding to $k=0$ above),
except that we now have to use $Γ$ instead of the first factor~$Δ$ in~$Δ^{⨯d}$.

\begin{definition}[Localizing maps for the model structure with duals for objects]\label{dualobj}
We consider the following maps, which will encode duals for objects. 
\begin{enumerate} 
\item[(vii)]
Take the free symmetric monoidal category~$C$ with duals on a single object~$*$.
Segal's machine converts $C$ to a $Γ$-object in small categories.
We then further convert this $Γ$-object in small categories to a $Γ$-object in simplicial sets by taking the nerves,
promote it to a simplicial presheaf on $Γ⨯Δ$,
and then pull it back to a simplicial presheaf on $Γ⨯Δ^{⨯d}$.
Consider the subobject $f_⊗⊂\dual_⊗$ generated by the image of the object~$*$ inside~$\dual_⊗$.
The inclusion 
$$f_⊗→\dual_⊗$$ 
is the desired morphism.
\end{enumerate}
\end{definition}

\begin{definition}\label{smcatdual}
We define the model category $\smcatdual_{∞,d}$ as the left Bousfield localization of $\smcatglob_{∞,d}$ (\cref{smcatglob}) at the morphisms given by
externally tensoring (via $\boxtimes$) representables of the form $j(V,\langle \ell\rangle)$, with $V\in \cartsp$ and $\langle \ell\rangle\in \Gamma$, with the morphisms (vi) in \cref{dualmorph} (for arbitrary ${\bf m}$),
and externally tensoring representables of the form $j(V)$ with the morphisms (vii) in \cref{dualobj}.
A \emph{smooth symmetric monoidal $(∞,d)$-category with duals} is defined as a fibrant object in $\smcatdual_{∞,d}$.
\end{definition}

\begin{proposition}
Let $\mathscr{C}$ be a fibrant object in $\smcatdual_{∞,d}$.
Then for all $(V,\langle \ell\rangle)\in \cartsp\times \Gamma$, the $(\infty,d)$-category $\mathscr{C}(V,\langle \ell\rangle)$ admits duals for all $k$-morphisms, with $1\leq k<d$. 
Also, for all $V∈\cartsp$, the symmetric monoidal $(∞,d)$-category $\mathscr{C}(V)$ admits duals for objects.
\end{proposition}

\begin{proof}
We use the notation of \cref{simpltosegal}.
Let $\mathscr{C}$ be a fibrant object.
By \cref{simpltosegal}, $\mathscr{D}$ is equivalent to the nerve of a simplicial category ${\sf C}$ that is fibrant in $\sset_J\text{-}\Cat$. 

Let $\underline{f}\into \underline{\rm Adj}$ denote the inclusion of the simplicial computad generated by the left adjoint into $\underline{\rm Adj}$.
By Riehl–Verity \cite[Proposition 4.4.17]{RiehlVerity} there is an induced Kan fibration between Kan complexes 
$$\map(\underline{\rm Adj},{\sf C})\to {\rm leftadj}({\sf C})\subset \map(\underline{f},{\sf C})\simeq \mathscr{D}_{1,0},$$
where ${\rm leftadj}({\sf C})$ is the simplicial subset with 0-simplices given by those morphisms whose image in the homotopy 2-category admits a right adjoint,
1-simplices are isomorphisms, and for $n\geq 2$,
an $n$-simplex in $\map(\underline{f},{\sf C})$ is in ${\rm leftadj}({\sf C})$ precisely when its vertices and 1-simplices are in ${\rm leftadj}({\sf C})$.

Since we have assumed that $\mathscr{C}$ is local with respect to the maps in \cref{dualmorph},
\cref{simpltosegal} implies that we have a commutative diagram of weak equivalences of derived mapping spaces
$$\xymatrix{
\map(\underline{\rm Adj},{\sf C})\ar[r]_-{\simeq}^-{{\cal N}}\ar[d]_{\simeq} & \map({\cal N}(\underline{\rm Adj}),\mathscr{D})\ar[d]^-{\simeq}\ar[r]^-{\simeq} & \map({\rm Adj}_{\bf m},\mathscr{C}(V,\langle \ell\rangle))\ar[d]\ar[d]^-{\simeq}\\
{\rm leftadj}({\sf C})\ar[r]^-{\cal N} &  \map({\cal N}(\underline{f}),\mathscr{D})\ar[r]^-{\simeq} &\map(f_{\bf m},\mathscr{C}(V,\langle \ell\rangle)).\\
}$$
The top map is an equivalence since the mapping simplicial sets on both sides are derived and ${\cal N}$ is a Quillen equivalence.
By the 2-out-of-3 property, the bottom map is a weak equivalence. By definition of ${\rm leftadj}({\sf C})$, this implies that every morphism in the homotopy 2-category admits a left adjoint.
Since ${\bf m}\in \Delta^{\times (k-1)}$ was arbitrary, this proves that $\mathscr{C}(V,\langle \ell\rangle)$ has adjoints for $k$-morphisms. 

Duals for objects is essentially immediate from the fact that $\mathscr{C}$ is assumed to be a local object.
Indeed, since the model structure is simplicial and the map $f_⊗\into \dual_⊗$ in \cref{dualobj} an acyclic cofibration, the induced map $\map(\dual_⊗,\mathscr{C})\to \map(f_⊗,\mathscr{C})$ is an acyclic fibration.
In particular, it is surjective on vertices.
Unwinding the definitions, a map $f_⊗\to \mathscr{C}$ picks out an object $x\in \mathscr{C}(\langle 1\rangle)$.
A lift $\dual_⊗\to \mathscr{C}$ picks out an object $x^{\vee}$ together with a unit and counit map, which exhibit $x^\vee$ as the dual of~$x$ in the homotopy category. 
\end{proof}

In the following definition, we use \cref{power.cat}.

\begin{definition}
\label{unit.notation}
Let $d\geq 2$, $1\leq k<d$, ${\bf m}\in \Delta^{k-1}$, and $\mathscr{C}\in \smcatdual_{\infty,d}$ be a fibrant object.
Denote by $\unit_{\bf m}(\mathscr{C})\subset \Map(\eta_{\bf m},\mathscr{C})$ the subobject of the powering (see \cref{power.cat})
given by the essential image of the restriction map $\Map({\rm Adj}_{\bf m},\mathscr{C})→\Map(\eta_{\bf m},\mathscr{C})$.
\end{definition}

\section{The smooth $(\infty,d)$-category of bordisms with $d$-thin homotopies}

In this section, we review the construction of the $(\infty,d)$-category of bordisms, together with $d$-thin homotopies as higher morphisms.
For more details, we refer the reader to Grady–Pavlov \cite{GradyPavlov}.
To circumvent size issues, we will always assume the underlying set of a smooth manifold to be a subset of $\RR$.

\subsection{The smooth $(\infty,d)$-category of bordisms}

We begin with some elementary definitions, which are necessary to encode geometric structures.

\begin{definition}
Let $\cartsp$ be the category whose objects are open subsets $U\subset \RR^n$ that are diffeomorphic to $\RR^n$, for some $n$. Morphisms are smooth maps $f:U\to V$.
\end{definition}

\begin{definition}\label{fembdef}
Let $\FEmb_d$ be the site with objects submersions $p:M\to U$, with $d$-dimensional fibers and $U$ an object in $\cartsp$. 
Morphisms are smooth bundle maps $f:M\to N$ 
that restrict to open embeddings fiberwise.
Covering families are given by a collection of morphisms
$$\left\{\vcenter{
\xymatrix{
M_{\alpha}\ar[r]\ar[d] & M\ar[d]\\
U_{\alpha}\ar[r] & U\\
}}\right\}
$$ 
such that $\{M_{\alpha}\}$ is an open cover of $M$.
\end{definition}

\begin{definition}
\label{geometric.structure}
A {\it fiberwise $d$-dimensional geometric structure\/} is a simplicial presheaf on $\FEmb_d$.
\end{definition}

As explained in Grady–Pavlov \cite[\ecref{EL-examples.of.geometric.structures}]{GradyPavlov},
this notion of a geometric structure is extremely general
and can encode maps to an arbitrary target smooth manifold,
Riemannian and pseudo-Riemannian metrics (possibly with restrictions on sectional or Ricci curvature),
principal $G$-bundles with connections,
geometric spin and string structures,
as well as homotopical higher structures
used in physics, such as fivebrane structures (Sati–Schreiber–Stasheff \cite{SatiSchreiberStasheff.Fivebrane, SatiSchreiberStasheff.Twisted})
and ninebrane structures (Sati \cite{Sati}).

The basic objects in our bordism category are submersions $M\to U$ with $d$-dimensional fibers, equipped with ``cut manifolds'', which we recall below.
The geometric structures on our bordisms are given by morphisms from $M\to U$ to the moduli stack of geometric structures ${\cal S}$.
One can think that an object in our bordism category is a family of composable bordisms parametrized by a base space $U$, equipped with a family of geometric structures.

\begin{definition}\label{cut}
A \emph{cut} of an an object $p:M\to U$ in $\FEmb_d$ is a triple $(C_{<},C_=,C_>)$ of subsets of $M$ such that there is a smooth map $h:M\to \RR$ and the fiberwise-regular values of the map $(h,p)$ form an open neighborhood of $\{0\}\times U$ in $\RR \times U$. 
Moreover, $h^{-1}(-\infty,0)=C_{<}$, $h^{-1}(0)=C_=$ and $h^{-1}(0,\infty)=C_>$.
We set $$C_\le=C_<\cup C_=, \qquad C_{\ge}=C_>\cup C_=.$$
There is a functor ${\rm Cut}:(\FEmb_d)^\op\to \set$ that associates to an object $p:M\to U$ its set of cuts, and to a morphism the induced map of sets that takes preimages of cuts.
We equip the set of cuts with a natural ordering $\leq$, with $C\leq C'$ if and only if $C_{\leq}\subset C'_{\leq}$.
\end{definition}

\begin{definition}\label{cut.m.tuple}
Fix $d\geq 0$, a simplex $[m]\in \Delta$, and an object $p:M\to U$ in $\FEmb_d$. 
A \emph{cut} $[m]$-\emph{tuple} $C$ for $p:M\to U$ is a collection of cuts $C_j=(C_{<j},C_{=j},C_{>j})$ of $p:M\to U$ indexed by vertices $j\in [m]$ such that 
$$C_0\leq C_1\leq \cdots \leq C_m.$$
We set $$C_{(j,j')}=C_{>j}\cap C_{<j'}, \qquad C_{[j,j']}=C_{\ge j}\cap C_{\le j'}.$$
There is a functor ${\rm CutTuple}:\Delta^\op\times (\FEmb_d)^\op\to \set$ that associates to an object $([m],p:M\to U)$ the set of cut $[m]$-tuples of $p$. 
The functor associates to a morphism the induced map of sets that takes preimages of the cuts and reindexes them according to the map of simplices.
Thus, a face map removes a cut and a degeneracy map duplicates a cut.
\end{definition}

The following illustration depicts a surface that is cut in two directions.
The cuts on the surface are depicted as dashed curves.
There exists two height functions, which together map to $\RR^2$, depicted on the right.
The preimage of the dashed lines give the corresponding cuts on the surface.

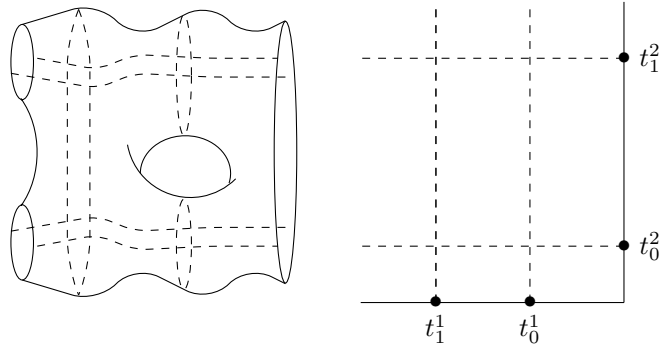
\begin{center}
\begin{tikzpicture}[scale=.5]
\draw[rounded corners=25pt](0,1.38)--(1,0)--(0,-1.38);
\draw (0.3,2.4) arc (0:360:0.3 and 1);
\draw  (0.3,-2.4) arc (0:360:0.3 and 1);
\draw[dashed,rounded corners=10pt] (1.5,3.8)--(1.8,3)--(1.8,-3)--(1.5,-3.8);
\draw[dashed,rounded corners=10pt] (1.5,3.8)--(1.2,3)--(1.2,-3)--(1.5,-3.8);
\draw[dashed,rounded corners=5pt] (0.4,2.5)--(2,2.2)--(3,2.6)--(5,2.5)--(6.8,2.5);
\draw[dashed,rounded corners=5pt] (-.3,2.1)--(2,1.7)--(3,2.1)--(5,2)--(7.2,2);
\draw[dashed,rounded corners=5pt] (0.4,-2.5)--(2,-2.2)--(3,-2.6)--(5,-2.5)--(6.8,-2.5);
\draw[dashed,rounded corners=5pt] (-.3,-2.1)--(2,-1.7)--(3,-2.1)--(5,-2)--(7.2,-2);
\draw[dashed] (4.5,2.05) arc (0:360:0.2 and 1.6);
\draw[dashed] (4.5,-2.45) arc (0:360:0.2 and 1.2);
\draw[rounded corners=10pt](0,3.4)--(2,4)--(3,3)--(5,4)--(6,3)--(7,3.5);
\draw[rounded corners=10pt](0,-3.4)--(2,-4)--(3,-3)--(5,-4)--(6,-3)--(7,-3.5);
\draw (7.3,0) arc (0:360:0.3 and 3.48);
\draw (2.8,0.2) arc (190:315:1.7cm and 1.7cm);
\draw (5.5,-0.82) arc (-15:180:1.2cm and 1cm);

\draw (9,-4)--(16,-4);
\draw (16,-4)--(16,4);
\node at (11,-4) {$\bullet$};
\node at (13.5,-4) {$\bullet$};
\node at (16,-2.5) {$\bullet$};
\node at (16,2.5) {$\bullet$};
\draw[dashed] (11,-4)--(11,4);
\draw[dashed] (13.5,-4)--(13.5,4);
\draw[dashed] (11,-4)--(11,4);
\draw[dashed] (16,-2.5)--(9,-2.5);
\draw[dashed] (16,2.5)--(9,2.5);
\node at (11,-4.7) {$t_1^1$};
\node at (13.5,-4.7) {$t_0^1$};
\node at (16.7,-2.5) {$t_0^2$};
\node at (16.7,2.5) {$t_1^2$};
\end{tikzpicture}
\end{center}

\begin{remark}
We remark that our definition of cut tuples allows for more general cuts than the ones depicted above. Cuts in the same simplicial direction may overlap, however we do not allow cuts in the same simplicial direction to cross each other transversally, as this would violate \cref{cut.m.tuple}.
\end{remark}

We define the (non-globular, or $d$-uple) $(\infty,d)$-dimensional bordism category as follows. 

\begin{definition}[$d$-uple bordisms]\label{bord}
Given $d≥0$, we specify an object $\Bord_{d,\uple}$ in the category
$$\smcatuple_{\infty,d}=\PSh_\Delta(\cartsp\times \Gamma \times \Delta^{\times d} )_{\uple}$$
as follows.
For an object $(U,\langle \ell\rangle,{\bf m})\in\cartsp⨯Γ⨯\Delta^{\times d}$,
the simplicial set $\Bord_{d,\uple}(U,\langle \ell\rangle,{\bf m})$
is the nerve of the following groupoid.

\medskip

\noindent {\bf Objects:\/}
An object of the groupoid is a {\it bordism\/} given by the following data.
\begin{enumerate}
\item A $d$-dimensional smooth manifold $M$.
\item For each $1\leq i \leq d$, a cut $[m_i]$-tuple $C^i$ for the projection $p:M\times U\to U$.
\item A choice of map $P:M\times U\to \langle \ell\rangle$, which gives a partition of the set of connected components of $M\times U$ into $\ell$ disjoint subsets
and another subset corresponding to the basepoint (the “trash bin”),
\end{enumerate}
which satisfy the transversality property: 
\begin{enumerate}
\item[$\pitchfork$.] For every subset $S\subset \{1,\ldots,d\}$ and for any $j:S→\ZZ$ such that $0\leq j_i\leq k_i$ for all $i∈S$, there is a smooth map $h_S:M\times U\to \RR^S$ such that for any $i\in S$, the map $$\pi_i\circ h_S:M\times U\to  \RR,$$
where $\pi_i:\RR^S\to \RR$ is the $i$th projection, yields the $j_i$-th cut $C^i_{j_i}$ in the cut tuple $C^i$. 
We require that the fiberwise-regular values of $(h_S,p)$ form an open neighborhood of ${\bf m}\times U\subset \RR^S\times U$, where ${\bf m}\in \Delta^S$ is the multisimplex whose $i$-th component is $[m_i]$.
\end{enumerate}

To simplify notation, we define 
$$C_{[j,{j'}]}≔\bigcap_{i\in S} C_{[j_i,j'_i]}$$
and 
$$C_{(j,{j'})}≔\bigcap_{i\in S} C_{(j_i,j'_i)},$$
where $j,j':S→\ZZ$ and $0\leq j_i\leq j'_i\leq m_i$ for all $i∈S$.
We also set
$$
\core([j,j'],P)= C_{[j,j']} \setminus P^{-1}\{\ast\}, \eqlabel{core}
$$
and we require it to be compact, for all choices of $j,j'$. We will omit $P$ in the notation when it is clear from the context.

\medskip

\noindent{\bf Morphisms:\/}
A morphism of the groupoid is a {\it cut-respecting map\/} that restricts to a diffeomorphism on the germ of each cell in the mesh obtained from the cut tuples (see Grady–Pavlov \cite[\ecref{EL-bord}]{GradyPavlov} for more details).
\end{definition}

The presheaf structure maps in the above definition are obtained in the semi-obvious way.
For maps in~$\Delta$, one simply removes or duplicates a cut, according to a face or degeneracy map.
The structure maps for~$\Gamma$ repartition connected components, possibly throwing more components into oblivion.
For smooth maps, we precompose to get new cut tuples and a new partition of connected components, parametrized over a new base space. 

Now that we have defined the $d$-uple bordism category, we can introduce geometric structures on bordisms as follows. 

\begin{definition}[$d$-uple bordisms with structure]\label{bordstr}
Fix $d≥0$ and ${\cal S}$ a simplicial presheaf on $\FEmb_d$. 
We specify an object $\Bord^{\cal S}_{d,\uple}$ in the category
$$\smcatuple_{\infty,d}=\PSh_\Delta(\cartsp\times \Gamma\times \Delta^{\times d})_{\uple}$$
by taking the diagonal of the nerve of the following presheaf of simplicial groupoids. 

\medskip
\noindent {\bf Objects:\/}
The simplicial set of objects is given by 
$${\rm Ob}≔\coprod_{(M,C,P)}{\cal S}_{\square}(M\times U),$$
where the coproduct ranges over the objects of \cref{bord} and the subscript $\square$ denotes the restriction to the germ of the core.
This groupoid is small because of our convention on manifolds (see Grady--Pavlov \cite[\ecref{EL-small}]{GradyPavlov}).
\medskip

\noindent {\bf Morphisms:\/} The simplicial set of morphisms is given by
$${\rm Mor}≔\coprod_{\varphi:(M,C,P)\to (\tilde M,\tilde C,\tilde P)}\mathcal{S}_{\square}(\tilde M\times \tilde U),$$
where the coproduct is taken over the morphisms in \cref{bord}.
The target map of the groupoid structure sends the component indexed by a germ $\varphi:(M,C,P)\to (\tilde M,\tilde{C},\tilde P)$ to itself by identity. 
The source map pulls back the structure by $\varphi$.
Composition is given by functoriality of ${\cal S}$. 
\end{definition}

The following example will be fundamental in proving the cobordism hypothesis. This is the analogue of the framed case in the topological setting. 

\begin{example}\label{geoframed}
Let us take a representable geometric structure of the form ${\cal S}=\RR^d\times U\in \FEmb_d$. 
A vertex in the simplicial set $\Bord^{\RR^d\times U}_{d}(V,\langle 1\rangle,{\bf m})$ is given by a bordism $(M,C,P)$ and a fiberwise open embeddings $M\times V\to \RR^d\times U$, restricted to the germ of the core. 
A 1-simplex is given by a cut-respecting diffeomorphism $\phi:M\to N$. In this case, the simplicial groupoid in \cref{bordstr} is discrete in both directions, 
i.e., discrete as a simplicial set and as a groupoid.

Alternatively, $\Bord^{\RR^d\times U}_{d}(V,\langle 1\rangle,{\bf m})$ is equivalent to the set whose elements are pairs $(f,N)$, where $f:V\to U$ is a smooth map and $N\subset \RR^d\times U$ is a subset such that the intersection of $N$ with each $\RR^d\times \{f(v)\}$ is an embedded submanifold $N_{v}\subset \RR^d$, varying smoothly with respect to $v\in V$.  
\end{example}

\begin{remark}\label{etalebordisms}
We observe that the codescent property for $\Bord_d^{\cal S}$ established in Grady--Pavlov \cite{GradyPavlov} implies that the geometric structure given by fiberwise open embeddings into $\RR^d$ can be replaced by fiberwise \'etale maps into $\RR^d$, which we denote by $(\RR^d\times U\to U)_{\text{\'et}}$.
Indeed, the latter is precisely the sheafification of the former.
Hence, codescent implies that the canonical map $(\RR^d\times U\to U)\to (\RR^d\times U\to U)_{\text{\'et}}$ induces an equivalence on bordism categories. 
\end{remark}

\subsection{Adding $d$-thin homotopies to bordisms}

In the main theorem, we will use the bordism category in which $d$-thin homotopies, and higher $d$-thin homotopies between these, are added as $k$-morphisms for $k>d$.
We recall the definition of $\BBord_d$ in Grady--Pavlov \cite[\ecref{EL-enrichedbordstr,EL-globbord}]{GradyPavlov}.

\begin{definition}[$d$-uple bordisms with isotopies]
\label{enrichedbordstr}
Fix $d≥0$ and ${\cal S}$ a simplicial presheaf on $\FEmb_d$. 
We specify an object $\BBord^{\cal S}_{d}$ in the category
$$\smcatuple_{\infty,d}=\PSh_\Delta(\cartsp⨯Γ⨯\Delta^{\times d})_{\uple}.$$
Fix $(U,\langle \ell\rangle,{\bf m})\in \cartsp\times \Gamma\times \Delta^{\times d}$.
We define $\Bord^{\cal S}_{d}((U,\langle \ell\rangle,{\bf m})$ by taking the diagonal of the nerve of the following simplicial groupoid.

\begin{enumerate}
\item[(GO)]
The simplicial set of objects is given by 
$${\rm Ob}≔\coprod_{(M,P)}{\cal S}_{\square}(M\times U\to U)\times\Cut_\square(M\times U\to U,P),$$
where the coproduct ranges over pairs $(M,P)$ as in \cref{bord}.
The simplicial set $\Cut(M\times U\to U,P)$
is a Kan complex that has as its $l$-simplices
$Δ^l$-indexed smooth collections~$C$ of cut tuples
such that for any $t\in Δ^l$
the triple $(M,C_t,P)$ is an object in $\Bord_{d,\uple}(U,\langle \ell\rangle,{\bf m})$,
and the subscript $\square$ denotes the germ around the union of cores of~$C_t$ inside~$M$, for all $t\in Δ^l$.
We identify two such $Δ^l$-indexed collections
if they have the same germ around the compact part of $Δ^l=\RR^{l+1}$ (given by $\{x∈\RR_{≥0}^{l+1}\mid \sum_i x_i=1\}$).
This simplicial set is a Kan complex because such germs of families can be glued along their common boundary.
\end{enumerate}

\begin{enumerate}
\item[(GM)] The simplicial set of morphisms is given by
$${\rm Mor}≔\coprod_{(\tilde M,\tilde P)}{\cal S}_{\square}(\tilde M\times U\to U)\times\DiffCut_\square(M\times U\to U,P),$$
where the coproduct is taken over the same pairs as for ${\rm Ob}$.
The simplicial set $\DiffCut(M\times U\to U,P)$
has as its $l$-simplices
an $l$-simplex of $\Cut(M\times U\to U,P)$
together with a germ of a diffeomorphism $M\to \tilde M$
taken around the union over all $t\in Δ^l$ of cores of~$C_t$ inside~$\tilde M$
(that is, the germ is taken for the codomain, not the domain) and that commutes with the maps $P$ and $\tilde P$.
The simplicial structure maps are given by the simplicial structure maps for $\Cut$
together with the appropriate restriction maps for germs of diffeomorphisms.
\end{enumerate}
The target map of the groupoid structure sends the component indexed by $(\tilde M,\tilde P)$ to itself
and discards the data of the germ of a diffeomorphism.
The source map pulls back the geometric structure and the cut tuples
via the given germ of a diffeomorphism $M\to \tilde M$,
landing in the component indexed by $(M,P)$.
Composition is given by functoriality of ${\cal S}$. 

The above simplicial sets assemble to form a presheaf on $\cartsp\times\Gamma\times\Delta^{\times d}$ in the semi-obvious way, using the presheaf structure of ${\cal S}$ on $\FEmb_d$ and restricting to the germ of the core as needed. 
\end{definition}

Deformations in the space of cut tuples can be interpreted as $d$-thin homotopies.
For example, if the geometric structure is given by smooth maps to a fixed smooth manifold,
then a 1-simplex in the space of bordisms yields a $d$-thin homotopy of smooth maps.

\begin{remark}
The presence of $d$-thin homotopies is crucial for the following reason. In order for the inductive argument in \cref{mainthm} to work, embedded cylinders must be homotopic to identities. Any homotopy deforming a $d$-dimensional cylinder to a degenerate cylinder through embeddings into $\RR^d$ necessarily has rank $\leq d$, hence is a $d$-thin homotopy. 

Thus, in the geometric framed case, the incorporation of $d$-thin homotopies is essential.
The codescent property of the $d$-thin bordism category (Grady--Pavlov \cite[\ecref{EL-mainthm}]{GradyPavlov}) implies that $d$-thin homotopies must be incorporated not only in the geometric framed case (i.e., the case of representable presheaves), but also in the case of general geometric structures.  
\end{remark}

To help elucidate the above definition, we conclude this section with a close examination of $\BBord_1^{\RR\times U\to U}$.

\begin{example}
\label{one.dimensional.bordisms}
The bordism category $\BBord_1^{\RR\times U\to U}$ has the following explicit description.

\paragraph{Objects:} Fix $V\in \cartsp$, $\langle 1\rangle\in \Gamma$, $[0]\in\Delta$, and $[l]\in \Delta$.
A simplex $x\in \BBord_1^{\RR\times U\to U}(V,\langle 1\rangle,[0])_{l}$ is given by a smooth $\Delta^l$-indexed family of cuts $C=(C_{<},C_{=},C_{>})$ on $M\times V\to V$, where $M\subset \RR$ is an open submanifold, along with a smooth map $f:V\to U$. The restriction of $C_=\subset M\times V$ to each fiber $v\in V$ is embedded in $\RR\times \{f(v)\}$, via the smooth map $f$ and the canonical inclusion $M\subset \RR$.  

\begin{figure}[ht]
\begin{center}
\tikzset{->-/.style={decoration={
  markings,
  mark=at position #1 with {\arrow[scale=1.5]{>}}},postaction={decorate}}}
\tikzset{-<-/.style={decoration={
  markings,
  mark=at position #1 with {\arrow[scale=-1.5]{>}}},postaction={decorate}}}
\begin{tikzpicture}
\draw [decorate,decoration={brace,amplitude=5pt, raise=4pt},yshift=0pt]
(-.6,.2) -- (.6,.2); 
\node at (0,.7) {germ};
\draw (-2,0) -- (2,0);
\node at (0,0) {$\bullet$};
\node at (0,-.4) {$C_{=}$};
\draw (0,-.2) -- (0,.2);
\draw (-.4,.2) arc (90:270:.2);
\draw (.4,.2) arc (90:-90:.2);
\node at (-1.5,-.4) {$C_{<}$};
\node at (1.5,-.4) {$C_{>}$};
\draw[->] (0,-.7)--(0,-1.7);
\draw[->-=.7] (-2,-2) -- (2,-2);
\node at (.2,-1) {$h$};
\node at (0,-2.4) {$0$};
\draw (0,-1.8) -- (0,-2.2);
\node at (1.5,-1.7) {$\RR$};
\node at (1.5,.3) {$M\subset \RR$};
\end{tikzpicture}
\end{center}
\caption{A point in the bordism category with $V=U=\ast$. 
The cut tuple is provided by a smooth map $h:\RR\to \RR$ and is defined by $C_<=h^{-1}(-\infty,0)$, $C_==h^{-1}(0)$ and $C_>=h^{-1}(0,\infty)$.
The core of this bordism is precisely the image of $C_=$.}
\end{figure}
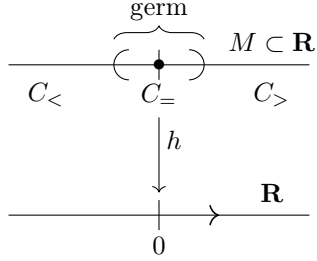

We call points \emph{negative} if the direction of the cut tuple $C$ is opposite the standard orientation on $\RR$. We call points \emph{positive} if the direction coincides with the standard orientation. 
Signs are not allowed to change in families. Indeed, such a change of sign would imply that the cut tuple $C$ must switch $C_<$ and $C_>$ at some point in the base space. This is forbidden by our definition of a bordism.

\paragraph{Morphisms:}
A simplex $[x,y]\in \BBord_1^{\RR\times U\to U}(V,\langle 1\rangle,[1])_l$ is given by a $\Delta^l$-family of cut tuples on $M\times V\to V$, with $M\subset \RR$ an open submanifold, along with a smooth function $f:V\to U$.
The core of this $\Delta^l$-family of bordisms is the union of a $\Delta^l$-family of fiberwise embedded intervals in $\RR$.
Here there is more than just the sign.
We can also join a positive point to a negative point in $\BBord_1^{\RR\times U\to U}$ by an elbow as follows. 
Fix some small $\epsilon>0$ and parametrize the (fattened) right half of the unit circle in $\RR^2$ by $x(\theta)=\cos(\theta),y(\theta)=\sin(\theta)$ where $-(\pi/2+\epsilon)<\theta<(\pi/2+\epsilon)$. Let $\tilde M$ be the image of this parametrization. 
Let $h:\tilde M\to \RR$ be the projection onto the $x$-axis.
Then the two cuts $C_{=}^0=h^{-1}(0)$ and $C_{=}^1=h^{-1}(2)$ give a 1-dimensional bordism from two points to the empty manifold.
Now consider the fiberwise embedding 
$$e_+:\tilde M\times \Delta^l\times V\to \RR\times U, \qquad e_+((\cos(\theta),\sin(\theta)),t,v)=(\theta,f(v))$$
where $f:V\to U$.
This embedding reverses the direction of one half of the cut $C^0$
and hence defines a bordism from the disjoint union of two points with opposite signs to the empty manifold.
Working in a single fiber, we denote two points of opposite sign by $+$ and~$-$.
Altogether, \cref{1bordisms.framed} depicts elementary 1-simplices in the bordism category. 
\begin{figure}[ht]
\begin{center}
\tikzset{->-/.style={decoration={
  markings,
  mark=at position #1 with {\arrow[scale=1.2]{>}}},postaction={decorate}}}
  \tikzset{-<-/.style={decoration={
  markings,
  mark=at position #1 with {\arrow[scale=-1.2]{>}}},postaction={decorate}}}
\begin{tikzpicture}
\draw[->-=.45,->-=.65,thick] (0,1.4) -- (0,3.4);
\draw[-<-=.35,-<-=.6,thick] (5,1.4) -- (5,3.4);
\node at (0,2.4) {$\bullet$};
\node at (5,2.4) {$\bullet$};
\node at (1.5,3) {$i_+$};
\node at (6.5,3) {$i_-$};
\draw[->] (1,2.4) -- (2,2.4);
\draw[->] (6,2.4) -- (7,2.4);
\draw[->-=.7,thick] (3,1.4) -- (3,3.4);
\draw[->-=.7,thick] (8,1.4) -- (8,3.4);
\node at (3,2.4) {$\bullet$};
\node at (8,2.4) {$\bullet$};

\node at (3.3,3) {$\RR$};
\node at (8.3,3) {$\RR$};
\node at (3.3,0) {$\RR$};
\node at (8.3,0) {$\RR$};

\node at (.3,3) {$M$};
\node at (5.3,3) {$M$};
\node at (.3,-.5) {$M$};
\node at (5.7,-.5) {$M$};

\draw[->-=.1,-<-=.9,thick] (0,0) arc (90:-90:1);
\draw[-<-=.4,thick] (0,0) arc (90:110:1);
\draw[-<-=.4,thick] (0,-2) arc (-90:-110:1);
\draw[->-=.7,thick] (3,-2.4) -- (3,.4);
\node at (3,0) {$\bullet$};
\node at (3,-2) {$\bullet$};

\draw (6,0) arc (90:270:1);
\draw[-<-=.05,->-=.95,thick] (6,0) arc (90:270:1);
\draw[->-=.8,thick] (6,0) arc (90:70:1);
\draw[->-=.8,thick] (6,-2) arc (270:290:1);
\node at (6,0) {$\bullet$};
\node at (6,-2) {$\bullet$};
\node at (6.5,0) {$+$};
\node at (6.5,-2) {$-$};
\draw[->-=.7,thick] (8,-2.4) -- (8,.4);
\node at (8,0) {$\bullet$};
\node at (8,-2) {$\bullet$};
\draw[->] (1.5,-1) -- (2.5,-1);
\draw[->] (6.5,-1) -- (7.5,-1);
\node at (-.6,0) {$-$};
\node at (-.6,-2) {$+$};
\node at (0,0) {$\bullet$};
\node at (0,-2) {$\bullet$};

\node at (2,-.5) {$e_+$};
\node at (7,-.5) {$e_-$};
\end{tikzpicture}
\end{center}
\caption{The image of each displayed embedding (including the cuts) gives a 1-simplex in the bordism category. 
Two morphisms joining these dual objects are depicted on bottom.
The arrows in the codomain depict the orientation of $\RR$.
}
\label{1bordisms.framed}
\end{figure}
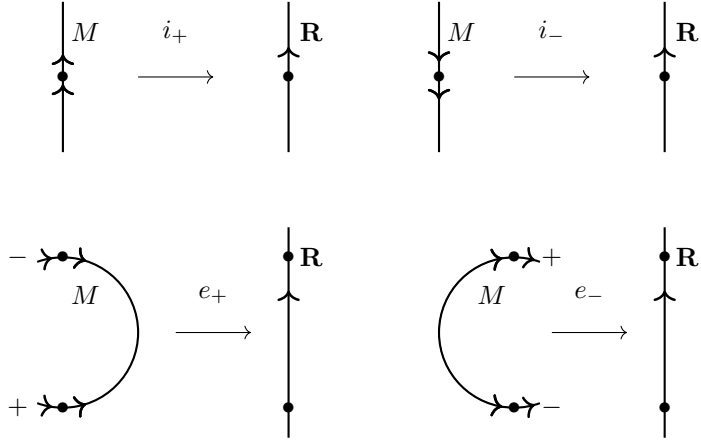

Observe that some of the bordisms in \cref{1bordisms.framed} cannot be composed strictly. For example, we cannot obtain a circle, since any bordism obtain by strict gluing is necessarily equipped with a local diffeomorphism to $\RR$. However, we do still obtain a circle in the fibrant replacement.
To see this, consider the proposed composition of bordisms in \cref{1bordisms.framed}.
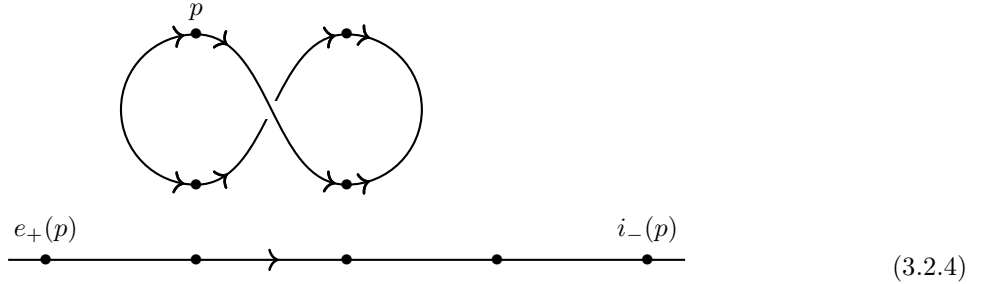
\begin{equation}\label{d1bords}
\begin{tikzpicture}
\tikzset{->-/.style={decoration={
  markings,
  mark=at position #1 with {\arrow[scale=1.2]{>}}},postaction={decorate}}}
\tikzset{-<-/.style={decoration={
  markings,
  mark=at position #1 with {\arrow[scale=-1.2]{>}}},postaction={decorate}}}
\draw[->-=.1,-<-=.9,thick] (0,0) arc (90:-90:1);
\draw[-<-=.05,->-=.95,thick] (-2,0) arc (90:270:1);
\draw[->-=.95,->-=.15,thick] (-2,-2) .. controls (-1,-2) and (-1,0) ..(0,0);
\node[fill=white] at (-1,-1) {};
\draw[-<-=.85,-<-=.05,thick] (0,-2) .. controls (-1,-2) and (-1,0) ..(-2,0);

\node at (0,0) {$\bullet$};
\node at (0,-2) {$\bullet$};
\node at (-2,0) {$\bullet$};
\node at (-2,-2) {$\bullet$};
\node at (-2,.3) {$p$};

\draw[->-=.4, thick] (-4.5,-3)--(4.5,-3);

\node at (-4,-3) {$\bullet$};
\node at (-2,-3) {$\bullet$};
\node at (0,-3) {$\bullet$};
\node at (2,-3) {$\bullet$};
\node at (4,-3) {$\bullet$};

\node at (-4,-2.6) {$e_+(p)$};
\node at (4,-2.6) {$i_-(p)$};
\end{tikzpicture}
\end{equation}
The bottom line depicts the image of the embeddings corresponding to each bordism.
Observe that the embeddings $e_+$ and $i_-$ do not match on the germ of the point labeled $p$ in the diagram.
However, the embeddings differ only by a translation by some real number $a\in \RR$.
Then the homotopy given by postcomposing the restriction of $e_+$ to the germ at $p$ with the family of translations $f_t(x)=x+ta$ defines a homotopy in the space of embeddings connecting the restriction of $e_+$ to $i_-$, hence a 1-simplex in the corresponding simplicial set of bordisms, using \cref{c.formula}.
By the Segal condition, the above defines a bordism in the fibrant replacement. 

This argument also demonstrates that 1-bordisms given by intervals are invertible in $\BBord_1^{\RR\times U\to U}$.
(This argument continues to work verbatim for cylinders over higher dimensional bordisms, which is used in the proof of \cref{contractiblecylinders}.)

All bordisms in the geometric framed case for $d=1$ are given by composing the elementary bordisms in \cref{d1bords}, although some compositions do not appear until we fibrantly replace the bordism category.
\end{example}

\subsection{Thin homotopies for geometric structures}

In this section, we will define a new model structure on $\sPSh(\FEmb_d)$.
This model category has the property that fibrant objects are invariant under $d$-thin homotopies (see \cref{fembdelta}). 
This allows us to establish a connection with a more traditional formulation of the (topological) cobordism hypothesis,
which involves spaces equipped with an action of the $\infty$-group~$\O(d)$.

\begin{definition}\label{flcms}
Let $\sPSh(\FEmb_d)_\loc$ denote the left Bousfield localization of the injective model structure at Čech covers.
The {\it fiberwise locally constant model structure\/} $\sPSh(\FEmb_d)_\lconst$ is the left Bousfield localization of $\sPSh(\FEmb_d)_\loc$
at representable morphisms of the form $(\RR^d⨯U→\RR^d⨯U,\id_U:U→U)$,
i.e., a smooth $U$-indexed family of open embeddings $\RR^d→\RR^d$,
for any $U\in \cartsp$.
\end{definition}

\begin{definition}\label{fembcartdef}
Given $d≥0$, the site $\FEmbCart_d$
is the full subcategory of $\FEmb_d$ (\cref{fembdef})
on objects isomorphic to projections $\RR^d⨯U→U$,
with the Grothendieck topology of good open covers
(meaning open covers on total spaces such that any finite intersection is empty or isomorphic to an object of $\FEmbCart_d$).
\end{definition}

\begin{proposition}
\label{fembfembcart}
The restriction functor
$$\sPSh(\FEmb_d)_\loc→\sPSh(\FEmbCart_d)_\loc$$
is a right Quillen equivalence.
The restriction functor
$$\sPSh(\FEmb_d)_\lconst→\sPSh(\FEmbCart_d)_\lconst$$
is a right Quillen equivalence.
\end{proposition}

\begin{proof}
For the first part, we apply the comparison lemma of Hoyois \cite[Lemma~C.3]{Hoyois}.
The conditions of the lemma are satisfied:
condition (a) is satisfied by definition of $\FEmbCart_d$,
condition (b) is satisfied because any open cover of an object in $\FEmbCart_d$
by objects in $\FEmb_d$ can be refined to an open cover by objects of $\FEmbCart_d$,
and condition (c) is satisfied because any object of $\FEmb_d$
admits an open cover by objects in $\FEmbCart_d$
such that any finite intersection is empty or belongs to $\FEmbCart_d$.
For the second part, it suffices to observe that the derived left adjoint functor
sends the localizing morphisms of \cref{flcms} to weak equivalences in $\sPSh(\FEmbCart_d)_\lconst$,
whereas the derived right adjoint functor preserves local objects (i.e., fiberwise homotopy locally constant presheaves).
\end{proof}

\begin{definition}
\label{fembdelta}
The site $\FEmb^\Delta_d$ has the same objects as $\FEmb_d$.
Given two objects $T→U$ and $T'→U'$,
the hom-object between them is a simplicial set whose $n$-simplices
are pairs of smooth maps $t:\Delta^n\times T\to T'$ and $u:U→U'$,
where for any $x\in\Delta^n$ the resulting map $t(x,-):T→T'$ together with $u:U→U'$ forms a morphism in $\FEmb_d$.
\end{definition}

\begin{definition}\label{fembcartdeltadef}
Given $d≥0$, the site $\FEmbCart^\Delta_d$
is the full subcategory of $\FEmb^\Delta_d$ (\cref{fembdelta})
on objects isomorphic to projections $\RR^d⨯U→U$,
with the Grothendieck topology of good open covers
(meaning open covers on total spaces such that any finite intersection is empty or isomorphic to an object of $\FEmbCart^\Delta_d$).
\end{definition}

\begin{proposition}
\label{fembdeltafembdeltacart}
The restriction functor
$$\sPSh(\FEmb^Δ_d)_\loc→\sPSh(\FEmbCart^Δ_d)_\loc$$
is a right Quillen equivalence.
\end{proposition}

\begin{proof}
The proof of \cref{fembfembcart} continues to work
once we observe that homotopy pullbacks of open inclusions in $\FEmb^Δ_d$ and $\FEmbCart^Δ_d$
can be computed as intersections.
\end{proof}

\begin{definition}
\label{c.formula}
The adjunction
$$\xymatrix{ L: \sPSh(\FEmb_d)_\lconst \ar@<.1cm>[r] & \ar@<.1cm>[l]\sPSh(\FEmb^Δ_d)_\loc : ρ}$$
is defined as follows.
The right adjoint~$ρ$ is given by restriction along the functor $\FEmb_d→\FEmb^Δ_d$.
The left adjoint $L$ is a simplicial left adjoint functor
that sends the representable presheaf of $(T→U)∈\FEmb_d$
to the representable simplicial presheaf $L(T→U)$,
i.e., $L(T→U)(V→S)=\FEmb^Δ_d(V→S,T→U)$.
The adjunction
\begin{equation}\label{fembcartadj}
\xymatrix{ L: \sPSh(\FEmbCart_d)_\lconst \ar@<.1cm>[r] & \ar@<.1cm>[l]\sPSh(\FEmbCart^Δ_d)_\loc : ρ}
\end{equation}
is defined likewise.
\end{definition}

\begin{definition}
\label{fiberwiseshapefun}
Given $F∈\sPSh(\FEmb_d)_\loc$, we denote by $$X→\mathfrak{C}_d X$$
the derived unit of the adjunction in \cref{c.formula}.
Thus, $\mathfrak{C}_d=ρLQ$ is the composition of the cofibrant replacement functor~$Q$
on $\sPSh(\FEmb_d)_\loc$ with the functor $L$ and then~$ρ$.
\end{definition}

\begin{proposition}
\label{fembfembdelta}
The adjunction of \cref{c.formula} is a Quillen equivalence.
\end{proposition}

\begin{proof}
First, the adjunctions are Quillen adjunctions because the inclusion $\FEmb_d→\FEmb^Δ_d$ (respectively $\FEmbCart_d→\FEmbCart^Δ_d$)
induces a Quillen adjunction between projective model structures.
This Quillen adjunction descends to Čech localizations
because the inclusion functor preserves covering families.
Finally, the left derived functor of~$L$ maps the localizing morphisms in \cref{flcms} to weak equivalences.

By the 2-out-of-3 property for Quillen equivalences,
it suffices to show that the adjunction \cref{fembcartadj}
is a Quillen equivalence. We conclude by invoking Ayala–Francis \cite[Proposition~2.19]{AyalaFrancis}.
The proof given there for $\Emb_d$ carries over verbatim to $\FEmb_d$.
\end{proof}

\begin{remark}
The functor 
$$\mathfrak{C}_d:\sPSh(\FEmb_d)_\loc\to \sPSh(\FEmb_d)_\loc$$
of \cref{fiberwiseshapefun}
performs the left Bousfield localization with respect to the morphisms that define the fiberwise locally constant model structure in \cref{flcms}.
Thus, for any $X\in\sPSh(\FEmb_d)_\loc$,
we have a localization morphism $X→\mathfrak{C}_d(X)$,
where $\mathfrak{C}_d(X)$ is local with respect to the maps in \cref{flcms}.
In model-categorical terms, $\mathfrak{C}_d$ is weakly equivalent to the fibrant replacement functor in $\sPSh(\FEmb_d)_\lconst$.
\end{remark}

We now relate the model category $\sPSh(\FEmb_d)_\lconst$ to the more familiar (but also more difficult in practice)
model category of presheaves on $\cartsp$ valued in simplicial sets equipped with an action of the orthogonal group $\O(d)$ (or rather its singular simplicial set).

\begin{proposition}
\label{fembcartgl}
The inclusion of simplicial categories
$$\cartsp\times\tdeloop(\Sing\GL(d))→\FEmbCart^Δ_d,\qquad U\mapsto (\RR^d\times U\to U)$$
induces a Quillen equivalence
$$\sPSh(\FEmbCart^Δ_d)_\loc→\PSh(\cartsp,\sset^{\Sing\GL(d)})_\loc,$$
where the right side denotes the projective model structure
on the category of simplicial sets equipped with an action of the simplicial group $\Sing\GL(d)$.
\end{proposition}

\begin{proof}
The inclusion of simplicial categories
$$\cartsp\times\tdeloop(\Sing\GL(d))\to\FEmbCart^\Delta_d,\qquad U\mapsto (\RR^d\times U\to U)$$
is a Dwyer–Kan equivalence of simplicial categories.
Indeed, the inclusion functor is homotopically essentially surjective by construction.
It is homotopically fully faithful by the smooth analogue of the Kister–Mazur theorem.
\end{proof}

\begin{proposition}\label{fiberwiseshape}
We have a zigzag of right Quillen equivalences of model categories
$$\xymatrix{
\sPSh(\FEmb_d)_{\lconst} \ar[d]_q &\ar[l]_{\rho} \sPSh(\FEmb^\Delta_d)_\loc \ar[d]_q \ar[r]^{r} &\PSh(\cartsp,\sset^{\Sing\GL(d)})_\loc\cr
\sPSh(\FEmbCart_d)_{\lconst} &\ar[l]_{\rho} \sPSh(\FEmbCart^\Delta_d)_\loc \ar[ur]^{r}.\cr}$$
\end{proposition}

\begin{proof}
Combine \cref{fembfembcart,fembdeltafembdeltacart,fembfembdelta,fembcartgl}.
\end{proof}

In the proof of the main theorem, we will need the following transition functors between $\FEmb_{k}$ and $\FEmb_{k-1}$.

\begin{definition}
\label{transemb}
Let $k≥1$.
Consider the functor $$\iota_{k-1}:\FEmb_{k-1}\into \FEmb_{k}$$ defined on objects by $\iota_{k-1}(T→U)=T\times \RR→U$
and on morphisms by taking the product of an embedding with the identity $\id_\RR:\RR\to \RR$.
By left Kan extension, we obtain a corresponding adjunction $(\iota_{k-1})_{!}\dashv \iota^*_{k-1}$.
$$\xymatrix@C=1.5cm{
\sPSh(\FEmb_{k-1})\ar@<.1cm>[r]^{(\iota_{k-1})_!} & \ar@<.1cm>[l]^-{\iota_{k-1}^*}  \sPSh(\FEmb_{k}).
}$$
Since $(\iota_{k-1})_!$ sends covering families to covering families, the adjunction descends to the Čech-local model structure.
\end{definition}

\begin{lemma}
The adjunction \cref{transemb} descends to the fiberwise locally constant model structure of \cref{flcms}.
\end{lemma}

\begin{proof}
By the universal property of left Bousfield localizations, we need only show that the left Kan extension $(\iota_{k-1})_!$ sends morphisms whose domain and codomain are fiberwise diffeomorphic to some $\RR^d\times U\to U$ to morphisms of the same form. But this is clear, since for any morphism $f:(M\to U)\to (N\to V)\in \FEmb_d$, 
$$(\iota_{k-1})_!(f)=f\times \id_{\RR}:(M\times \RR\to U)\to(N\times \RR\to V).\qedhere$$
\end{proof}

The following lemma will be used in the proof of \cref{mainthm}.
It is analogous to the homotopy pushout diagram in Lurie \cite[Proof of Theorem 2.4.6, Part~3]{Lurie.TFT}, adapted to our setting.
\begin{lemma}
\label{inductive.reduction}
Fix $d≥2$ and $U\in\cartsp$.
The homotopy commutative diagram 
\begin{equation}
\xymatrix{
\ldf(\iota_{d-2})_{!}\iota_{d-2}^*(\RR^{d-1}\times U\to U) \ar[r]^-{\bar\epsilon}\ar[d]_{\epsilon} & \RR^{d-1}\times U\to U\ar[d]_{\bar e_d}\\
\RR^{d-1}\times U\to U\ar[r]_{e_d} & \iota_{d-1}^*(\RR^{d}\times U\to U)
}
\end{equation}
is a homotopy pushout diagram in $\sPSh(\FEmb_d)_\lconst$, where the model structure is defined in \cref{fiberwiseshape}.
The map $e_d$ is given by the inclusion $\RR^{d-1}\times\RR→\RR^d$.
The map $\bar e_d$ is given by the composition of~$e_d$ with the map $\RR^d→\RR^d$ that changes the sign of the $(d-1)$st and $d$th coordinates.
The map $\epsilon$ is adjoint to the identity map.
The map $\bar\epsilon$ is adjoint to the map that changes the sign of the $(d-1)$st coordinate.
\end{lemma}

\begin{proof}
\cref{fiberwiseshape} produces an equivalent homotopy commutative diagram in the model category $\PSh(\cartsp,\sset^{\Sing \GL(d)})$:
$$\xymatrix{
\ldf(\kappa_{d-2})_{!}\kappa_{d-2}^*(\Sing \GL(d-1)\times U\to U) \ar[r]^-{\bar\epsilon}\ar[d]_{\epsilon} & \Sing \GL(d-1)\times U\to U \ar[d]_{\bar e_d}\\
\Sing \GL(d-1)\times U\to U \ar[r]_{e_d} & \kappa_{d-1}^*(\Sing \GL(d)\times U\to U).
}$$
where $\kappa^*_{d-1}$ restricts a $\Sing \GL(d)$ action to a $\Sing \GL(d-1)$ action.
Similarly, $\kappa_{d-2}^*$ restricts from a $\Sing \GL(d-1)$ action to a $\Sing \GL(d-2)$ action.
The derived left Kan extension $(\kappa_{d-2})_!$ takes the homotopy fiber product with $\Sing \GL(d-1)$ over $\Sing \GL(d-2)$.
Using \cref{fiberwiseshape}, it suffices to show that this diagram is a homotopy pushout.
The diagram is obtained by taking a product of $U$ with the corresponding diagram without $U$s.
Hence, we need only show that we have a homotopy pushout homotopy commutative diagram of $\Sing \GL(d-1)$-equivariant simplicial sets:
$$\xymatrix{
\Sing \GL(d-1) \times_{\Sing \GL(d-2)} \Sing \GL(d-1) \ar[r]^-{\bar p}\ar[d]_{p} & \Sing \GL(d-1) \ar[d]_{\bar e_d}\\
\Sing \GL(d-1) \ar[r]_{e_d} & \Sing \GL(d),
}$$
where
the map $p$ is induced by the multiplication in $\GL(d-1)$.
Here $\GL(d-1)$ acts on the left on all objects, including the top left corner: $α(g,g')=(αg,g')$.
Also $\Sing \GL(d-2)$ acts on $\Sing \GL(d-1) \times \Sing \GL(d-1)$ as $x(g,g')=(gx,x^{-1}g')$.
Using the Quillen equivalence between simplicial sets and topological spaces, along with the fact that the inclusion $\O(d)\into\GL(d)$ is a weak homotopy equivalence,
the above square is transformed into
$$\xymatrix{
\O(d-1) \times_{\O(d-2)} \O(d-1) \ar[r]^-{\bar p}\ar[d]_{p} & \O(d-1) \ar[d]_{\bar e_d}\\
\O(d-1) \ar[r]_{e_d} & \O(d).
}$$
Here the map $e_d$ is the inclusion that omits the $d$th basis vector, whereas $\bar e_d(g)=e_d(g)γ_1$,
where $\gamma_t\in\O(d)$ is given by $\gamma_t(f_d)=\cos(\pi t)f_d+\sin(\pi t)f_{d-1}$, $\gamma_t(f_{d-1})=-\sin(\pi t)f_d+\cos(\pi t)f_{d-1}$,
with $f_{d-1}$ and $f_d$ denoting the last two basis vectors of~$\RR^d$.
The map $p$ is given by $p(g,g')=gg'$, whereas $\bar p$ is given by $\bar p(g,g')=gδg'δ^{-1}$,
where $δ$ reverses the sign of the $(d-1)$st coordinate.
The bottom left composition is $(g,g')↦gg'$ and the top right composition is $(g,g')↦gγ_1g'$.
The homotopy between them is given by $(g,g',t)↦gγ_tg'$.

Denote by $L$ the space $\O(d-1) \times_{\O(d-2)} \O(d-1)$.
The induced map from the homotopy pushout to the bottom right corner can be computed as
$$\O(d-1)\sqcup_L (L\times[0,1])\sqcup_L \O(d-1)\to\O(d),\eqlabel{luriehomeo}$$
where the two outer factors are embedded using the inclusion $e_d:\O(d-1)\to\O(d)$ and the inclusion $\bar e_d:\O(d-1)\to\O(d)$.
The left attaching map is $p:L\to\O(d-1)$.
The right attaching map is $\bar p:L\to\O(d-1)$.
The middle map is given by $(g,g',t)\mapsto g\gamma_t g'$.
By inspection, the map \cref{luriehomeo} is a homeomorphism.
\end{proof}

\section{The geometric cobordism hypothesis}\label{codescent}

We are now ready to prove the main theorem. 
We identify some key steps of the proof.
\begin{itemize}
\item Using the codescent property of the functor $\BBord_d$ (Grady–Pavlov \cite[\ecref{EL-mainthm}]{GradyPavlov}), in \cref{mainthm} we reduce the general statement for $\BBord_d^{\cal S}$
to the case of $\BBord_d^{\RR^d⨯U→U}$, which is much easier to work with.
\item We introduce a filtration $(B_k)_{-1≤k≤d}$ (\cref{index.filtration}) on the bordism category $\BBord_d^{\RR^d⨯U→U}$,
where $B_k$ contains bordisms that admit a Morse function with critical points of index at most~$k$, while vanishing on the source cut in the $d$th direction.
\item We prove (\cref{bkpushout}) that the inclusion $B_{k-1}→B_k$ ($k≥0$) is a homotopy cobase change
of the inclusion $\tilde O_{k-1}→\tilde H_k$ of handles of index~$k-1$ into handles of index $k-1$ and~$k$ (\cref{handle,def.O}).
\item In \cref{handlecutout}, we cut off the “tails” (given by identity $d$-bordisms on some $(d-1)$-bordisms) from handles in $\tilde O_{k-1}$ and $\tilde H_k$,
which results in subcategories $O_{k-1}$ and $H_k$, whose bordisms can be interpreted as units and counits in adjunctions.
\item In \cref{handleequiv} we show that for $k≥1$ the inclusion $O_{k-1}→H_k$
is a homotopy cobase change of the inclusion of the free-walking adjunction unit into the free-walking adjunction, parametrized by a certain moduli stack.
\item These observations are combined in \cref{bkequivalence} to show that the inclusion $B_{k-1}→B_k$ is a weak equivalence in $\smcatdual_{∞,d}$ for $k≥2$.
For $k=1$, interpreting the 0-handle (a $d$-dimensional disk) in $B_0$ as a unit of an adjunction and freely adjoining the corresponding counit produces an object
weakly equivalent to~$B_1$.
\item In \cref{b0equivalence}, we prove that the inclusion $B_{-1}→B_0$ freely adds a 0-handle (given by a $d$-dimensional disk).
\item In \cref{contractiblecylinders}, we prove that $B_{-1}$ up to a weak equivalence only has identity bordisms in the $d$th direction,
and in \cref{dimensional.reduction} we show that the objects of $B_{-1}$
are weakly equivalent to a $(d-1)$-dimensional bordism category $\BBord_{d-1}$ with a certain geometric structure.
\item All of the above results are assembled together in \cref{mainthm} to prove that $\BBord_d^{\RR^d⨯U→U}$ is weakly equivalent (in $\smcatdual_{∞,d}$)
to the representable presheaf of~$U$ (interpreted as a $U$-family of 0-bordisms given by points), using induction on~$d≥0$.
\end{itemize}

\subsection{Proof of the geometric cobordism hypothesis}

Since bordisms and their cut tuples will be used frequently, we will often denote bordisms simply by $M$, leaving the cut tuples and map $P:M\to \langle \ell\rangle$ implicit. 
In what follows, we will always assume that $\mathscr{C}\in \smcatdual_{\infty,d}$ is a fibrant object. In particular, $\mathscr{C}$ has all duals. 
We begin by considering a canonical example of a point in the bordism category $\BBord_{d}^{\RR^d\times U\to U}$. This point will be used to define the evaluation functor at the level of field theories.  

\begin{example}\label{thepoint}
Fix $d\ge0$ and let $\{e_i\}_{i=1}^d$ be the standard orthonormal basis of $\RR^d$.
The following is a canonical example of an object, i.e., a point, in the fully extended bordism category. Consider the triple 
\begin{equation}\label{bord.point}
(\RR^d\times U,\{C^k:1\leq k \leq d\}, 1:\RR^d\times U\to \langle 1\rangle)\in \BBord_d^{\RR^d\times U\to U}(U,\langle 1\rangle,{\bf 0}),
\end{equation}
where the cut $C^k$ in the $k$-th direction is given by the hyperplane $C^k_{=}={\rm span}\{e_i:i\neq k\}$
and $C^k_{<}$ and $C^k_{>}$ are given by the halfspaces for which the $k$th coordinate is negative respectively positive.
The map $1:\RR^d\times U\to \{*,1\}=\langle 1\rangle$ sends every point to~$1$.
The geometric structure is given by the identity embedding $\RR^d\times U\to \RR^d\times U$.
Evaluation at the family of points as in \cref{bord.point} gives rise to corresponding evaluation functor 
$${\rm ev}:\left(\Funmon(\BBord_d^{\RR^d\times U\to U},\mathscr{C})\right)^{\times}
\to \left(\Funmon(j(U,\langle 1\rangle,{\bf 0}),\mathscr{C})\right)^{\times}
\to \Map(j(U,\langle 1\rangle),\mathscr{C}^\times)
\to \Map(U,\mathscr{C}^{\times}),$$
where $\mathscr{F}^{\times}$ denotes the core of a smooth symmetric monoidal (∞,d)-category~$\mathscr{F}$ (\cref{coreadj}),
$\Funmon(-,-)$ denotes the internal hom (\cref{ihom.cat}),
and $\Map(-,-)$ denotes the powering of $\sPSh(\cartsp\times \Gamma)$ over $\sPSh(\cartsp)$ (\cref{power.cat}).
In the Segal space formalism, $\mathscr{F}^{\times}$ is just the value of $\mathscr{F}$ on the multisimplex ${\bf 0}\in \Delta^{\times d}$, whenever $\mathscr{F}$ is fibrant.
\end{example} 

To prove the cobordism hypothesis, we will need to enhance $\mathscr{C}^{\times}$ to a simplicial presheaf on $\FEmb_d$.
This is analogous to the topological case, where $\mathscr{C}^\times$ is enhanced with an action of $\O(d)$ as a topological group.

\begin{definition}\label{extcx}
Let $\mathscr{C}\in \smcatdual_{∞,d}$ (\cref{smcatdual}) be a fibrant object
and denote by $\mathscr{E}^{\times}$ the invertible part (i.e., the core, see \cref{coreadj}) of a smooth symmetric monoidal $(∞,d)$-category~$\mathscr{E}$.
We define the presheaf 
$$\mathscr{C}_{d}^{\times}:\FEmb_d^\op\to \sPSh(\cartsp\times \Gamma),\qquad\mathscr{C}^{\times}_{d}(\RR^d\times U\to U)≔\Funmon(\BBord_d^{\RR^d\times U\to U},\mathscr{C})^\times.\eqlabel{extcxfun}$$
The right side is given by taking the core of the internal hom, which is an object in $\sPSh(\cartsp\times \Gamma\times \Delta^{\times d})$. 

Equivalently, we can regard $\mathscr{C}^⨯_d$ as a presheaf on $\FEmb_d^\op\times \cartsp\times \Gamma$, and then the notation $\mathscr{C}^{\times}_{d}(\RR^d\times U\to U)$ is the partial evaluation (\cref{partials}).
We observe that we can always forget structure, by evaluating on $\RR^0\in \cartsp$, or $\langle 1\rangle\in \Gamma$, when necessary.
\end{definition}

\begin{remark}
At first glance, \cref{extcx} may seem strange.
However, let us recall that in the topological cobordism hypothesis, $\mathscr{C}^{\times}$ acquires an action of $\O(d)$ through the framed cobordism hypothesis,
via the action of~$\O(d)$ on framings.
\cref{extcx} provides a parallel construction in the geometric setting:
through the geometric framed cobordism hypothesis, $\mathscr{C}^{\times}$ can be extended to a simplicial presheaf $\mathscr{C}_{d}^{\times}$ on $\FEmb_d$.
\end{remark}

\begin{remark}
\cref{mainthm.geometric} proves that the smooth symmetric monoidal $(∞,d)$-category $$\Funmon(\BBord_d^{\cal S},\mathscr{C})$$
is a smooth symmetric monoidal $∞$-groupoid for any geometric structure ${\cal S}$.
Thus, applying $(-)^\times$ on the right side of \cref{extcxfun} merely converts between two equivalent descriptions of $∞$-groupoids:
as homotopy constant presheaves on $\Delta^{\times d}$ and as simplicial sets.
\end{remark}

\def\ala{\mathfrak a}
\def\alg{A}

\begin{remark}
\label{refinement.deligne}
As a practical example for \cref{extcx}, consider an abelian Lie group~$\alg$ and some $d\ge0$.
The smooth symmetric monoidal $(\infty,d)$-category $\deloop^d\alg$
sends $(V,\langle\ell\rangle,{\bf m})$ to $(\tdeloop^d \sm(V,\alg))^\ell$,
equipped with the obvious structure maps.
The resulting object is an $(\infty,d)$-category with a single $k$-morphism for $0\le k<d$
and an abelian Lie group~$\alg$ of $d$-morphisms.
In a forthcoming work of Stolz, Teichner, and the second author \cite{PavlovStolzTeichner},
we compute $(\deloop^d\alg)^\times_d$ as the simplicial presheaf that sends $p:M\to V$
to the Dold–Kan functor applied to the fiberwise Deligne chain complex
$$\Omega_p^d(M,\ala)←\Omega_p^{d-1}(M,\ala)←⋯←\Omega_p^1(M,\ala)←\sm(M,\alg).$$
Thus, the space of $d$-dimensional functorial field theories with geometric structure~${\cal S}$ and target $\deloop^d\alg$ is equivalent to
the space of $A$-banded bundle $(d-1)$-gerbes with connection on the geometric structure~${\cal S}$.
\end{remark}

We are now ready to prove the main theorem.
Supporting lemmas and propositions will be proved subsequently.
We remark that the strategy of the proof is similar to Lurie \cite[Theorem 2.4.6]{Lurie.TFT}.
\cref{mainthm} is fully enriched and both sides of the equivalence are smooth symmetric monoidal $(\infty,d)$-categories. 
We remind the reader that the objects $\Funmon$ and $\Map$ denote the internal hom, respectively powering (\cref{power.cat}). 
For the statement that only encodes the categorical structure, one can simply evaluate at $\RR^0\in \cartsp$ and $\langle 1\rangle\in \Gamma$. 
Hence, for a simplified version, the reader is welcome to substitute $\Map(-,-)$ with  the $\infty$-groupoid of maps $\map(-,-)$ and to read $\Funmon(-,-)$ as the $(\infty,d)$-category of symmetric monoidal functors. 

Recall the notations $\mathscr{C}^⨯$ of \cref{coreadj}, ${\rm c}_S$ of \cref{evconst},
$\Funmon$ of \cref{ihom.cat}, $\Map$ of \cref{power.cat}.

\begin{theorem}[Geometric framed cobordism hypothesis]\label{mainthm}
Fix $d\geq 0$, $U\in\cartsp$, and a fibrant object $\mathscr{C}∈\smcatdual_{∞,d}$ (\cref{smcatdual}).
The smooth symmetric monoidal $(∞,d)$-category
$\Funmon(\BBord_d^{\RR^d\times U\to U},\mathscr{C})$ is a smooth symmetric monoidal $∞$-groupoid, i.e., the inclusion of the core yields a weak equivalence in $\smcatdual_{∞,d}$:
$${\rm c}_{\{1,\ldots,d\}}\left(\Funmon(\BBord_d^{\RR^d\times U\to U},\mathscr{C})^\times\right)\overset{\simeq}{\longrightarrow} \Funmon(\BBord_d^{\RR^d\times U\to U},\mathscr{C}).$$
Furthermore, evaluation at the point (see \cref{thepoint}) yields an equivalence
of smooth symmetric monoidal $∞$-groupoids
$${\rm ev}:\Funmon(\BBord^{\RR^d\times U\to U}_d,\mathscr{C})^{\times}\overset{\simeq}{\longrightarrow} \Map(U,\mathscr{C}^⨯),$$
where $\Map(-,-)$ denotes the powering of simplicial presheaves on $\cartsp⨯Γ$ over simplicial presheaves on $\cartsp$.
\end{theorem}

\begin{proof}
We proceed by induction on the dimension $d$.
The main proof works with $k$-bordisms,
where $k=d$, $k=d-1$, $k=d-2$, or (in \cref{bkequivalence}) $k=d-3$.
In order for these constructions to work when $d<3$,
we deloop the bordism category $\BBord_d^{\RR^d\times U\to U}$ three times,
by pulling back along the canonical functor $Γ⨯Δ^{⨯3}→Γ$
given by the composition
$$Γ⨯Δ^{⨯3}→Γ^{⨯4}→Γ,$$
where the left functor uses the canonical functor $Δ→Γ$
and the right functor is the smash product of finite pointed sets.

For $d=0$, we invoke \cref{bkpushout}.
Since $\Ot_{-1}=B_{-1}=\emptyset$ for $d=0$,
we obtain a weak equivalence
$$\Funmon(\BBord_0^{U→U},\mathscr{C})→\Funmon(\Ht_0,\mathscr{C}),$$
where $\Ht_0$ is as in \cref{def.Ht}.
This map is weakly equivalent to the map from the statement
because $\Ht_0$ is weakly equivalent to the object defined in \cref{thepoint}.
This serves as the base of the induction.

Suppose the theorem is true for all dimensions strictly less than~$d$. Then \cref{mainthm.geometric} also holds in dimensions strictly less than $d$. 

\paragraph{Step 1 (moving from $B_1$ to codimension 1):}
Recall the functor 
$${\rm c}_{\{d\}}:\smcatdual_{\infty,d-1}\into \smcatdual_{\infty,d}$$
in \cref{evconstadj} that sends a $(d-1)$-fold simplicial object $\mathscr{C}$ to the $d$-fold simplicial object ${\rm c}_{\{d\}}(\mathscr{C})$ that is constant in the $d$th simplicial direction.
Recall also the functor $\iota_{d-1}: \FEmb_{d-1}→\FEmb_d$
(\cref{transemb})
and the Quillen adjunction $(\iota_{d-1})_!\dashv \iota_{d-1}^*$.
Combining \cref{contractiblecylinders,dimensional.reduction,bkequivalence,b0equivalence}, we have a weak equivalence
$$
\Funmon(B_1,\mathscr{C}) \simeq   \Funmon\left({\rm c}_{\{d\}}\left(\BBord_{d-1}^{\iota_{d-1}^*(\RR^d\times U\to U)}\right),\mathscr{C}\right)
\times_{\Funmon(O_{-1},\mathscr{C})}\Funmon(H_0,\mathscr{C})_\unitable. \eqlabel{uniteq}
$$

In the remainder of the proof, we will abbreviate ${\rm ev}_{\{d\}}\mathscr{C}$ as $\mathscr{D}$.
Using the induction hypothesis (with arbitrary geometric structure,  \cref{mainthm.geometric})
and the adjunction ${\rm c}_{\{d\}}\dashv {\rm ev}_{\{d\}}$, we see that the first factor on the right side of \eqref{uniteq} is equivalent to 
\begin{align*}
\Funmon\left({\rm c}_{\{d\}}\left(\BBord_{d-1}^{\iota_{d-1}^*(\RR^d\times U\to U)}\right),\mathscr{C}\right) &\simeq \Funmon\left(\BBord_{d-1}^{\iota_{d-1}^*(\RR^d\times U\to U)},{\rm ev}_{\{d\}}\mathscr{C}\right) 
\\
&\simeq \Map_{\FEmb_{d-1}}\left(\iota_{d-1}^*(\RR^d\times U\to U),\mathscr{D}_{d-1}^{\times}\right).
\end{align*}
Here we recall the notation $\Map_{\FEmb_{d-1}}(-,-)$, which is indeed an object of the correct type  (\cref{power.end.cat}). 
It follows that we have an equivalence
$$\Funmon(B_1,\mathscr{C})\simeq \Map_{\FEmb_{d-1}}\left(\iota_{d-1}^*(\RR^d\times U\to U),\mathscr{D}_{d-1}^{\times}\right)\times_{\Funmon(O_{-1},\mathscr{C})}\Funmon(H_0,\mathscr{C})_\unitable. \eqlabel{unitwithmap}
$$
\paragraph{Step 2 (moving to codimension 2):}
By \cref{inductive.reduction}, we have
a homotopy pushout diagram in $\sPSh(\FEmb_{d-1})_\lconst$, where the model structure is defined in \cref{fiberwiseshape}:
\begin{equation}\label{pushoutorth}
\xymatrix{
\ldf((\iota_{d-2})_{!})\iota_{d-2}^*(\RR^{d-1}\times U\to U )\ar[r]^-{\bar\epsilon}\ar[d]_{\epsilon} & \RR^{d-1}\times U\to U\ar[d]_{\bar e_{d}}
\\
\RR^{d-1}\times U\to U\ar[r]_{e_d} & \iota^*_{d-1}(\RR^{d}\times U\to U),
}
\end{equation}
where $\ldf((\iota_{d-2})_!)$ denotes the left derived functor of the left Kan extension $(\iota_{d-2})_!$.
The left vertical map is the derived counit of the adjunction $(\iota_{d-1})_!\dashv \iota_{d-1}^*$,
whereas the top horizontal map further changes the sign of the $(d-1)$st coordinate in the source and in the target.
The maps $e_{d}$ and $\bar e_{d}$ are the two natural transformations that include the hyperplane $\RR^{d-1}$ into $\RR^d$,
and change the sign of the $(d-1)$st and $d$th coordinate of~$\RR^d$ in the case of $\bar e_d$.
Taking derived powerings of $\mathscr{D}^\times_{d-1}$ by \eqref{pushoutorth} and using the Quillen adjunction $(\iota_{d-1})_{!}\dashv \iota^*_{d-1}$ gives a corresponding homotopy pullback diagram 
\begin{equation}
\xymatrix{
\Map_{\FEmb_{d-2}}(\iota_{d-2}^*(\RR^{d-1}\times U\to U ),\iota^*_{d-2}\mathscr{D}_{d-1}^{\times}) & \ar[l]_-{\bar{\epsilon}} \Map_{\FEmb_{d-1}}(\RR^{d-1}\times U\to U,\mathscr{D}^{\times}_{d-1})\\
\Map_{\FEmb_{d-1}}(\RR^{d-1}\times U\to U,\mathscr{D}^{\times}_{d-1})\ar[u]_{{\epsilon}}  &\ar[l]_{e_d}\ar[u]_{\bar{e}_{d}} \Map_{\FEmb_{d-1}}(\iota_{d-1}^*(\RR^{d}\times U\to U),\mathscr{D}_{d-1}^{\times}).
}
\end{equation}

Reversing the direction of the $(d-1)$st factor of~$Δ$
amounts to taking inverses in the $(d-1)$st direction,
yielding a weak equivalence $\mathscr{D}^⨯_{d-1}→(\mathscr{D}^\op)^⨯_{d-1}$.
Substituting this weak equivalence into the above square yields the following commutative diagram:
\begin{equation}\label{pullbackorth}
\xymatrix{
\Map_{\FEmb_{d-2}}(\iota_{d-2}^*(\RR^{d-1}\times U\to U ),\iota^*_{d-2}\mathscr{D}_{d-1}^{\times}) & \ar[l]_-{\bar{\epsilon}} \Map_{\FEmb_{d-1}}(\RR^{d-1}\times U\to U,(\mathscr{D}^\op)^{\times}_{d-1})\\
\Map_{\FEmb_{d-1}}(\RR^{d-1}\times U\to U,\mathscr{D}^{\times}_{d-1})\ar[u]_{{\epsilon}}  &\ar[l]_{e_d}\ar[u]_{\bar{e}_{d}} \Map_{\FEmb_{d-1}}(\iota_{d-1}^*(\RR^{d}\times U\to U),\mathscr{D}_{d-1}^{\times}).
}
\end{equation}

By definition of $\mathscr{D}^{\times}_{d-1}$, we have 
$$\Map_{\FEmb_{d-1}}(\RR^{d-1}\times U\to U,\mathscr{D}_{d-1}^{\times})\cong \mathscr{D}^{\times}_{d-1}(\RR^{d-1}\times U\to U)≔\Funmon(\BBord_{d-1}^{\RR^{d-1}\times U\to U},\mathscr{D}). \eqlabel{maptofun}$$
Observe that the right side is a smooth symmetric monoidal $\infty$-groupoid by the induction hypothesis.
To simplify notation, we set $\mathscr{E}={\rm ev}_{\{d-1\}}\mathscr{D}={\rm ev}_{\{d-1,d\}}\mathscr{C}$. Let $B_{k,d-1}$ denote the $k$-th layer of the filtration in \cref{index.filtration} in dimension $d-1$. By \cref{bkequivalence}, we have a weak equivalence 
$$B_{1,d-1}\overset{\simeq}{\into}\BBord_{d-1}^{\RR^{d-1}\times U\to U}.$$  
Let $H_{0,d-1}$ denote $H_0$ in dimension $d-1$.
Combining the last equivalence with \cref{uniteq} in dimension $d-1$, we have a weak equivalence
$$\Funmon(\BBord_{d-1}^{\RR^{d-1}\times U\to U},\mathscr{D})\simeq \Funmon(\BBord_{d-2}^{\iota_{d-2}^*(\RR^{d-1}\times U\to U)},\mathscr{E})\times_{\Funmon(O_{-1,d-1},\mathscr{D})}\Funmon(H_{0,d-1},\mathscr{D})_{\unitable}. \eqlabel{luriesthing}$$

Combining \cref{maptofun} and \cref{luriesthing}, we have a weak equivalence
$$\displaylines{\quad\Map_{\FEmb_{d-1}}(\RR^{d-1}\times U\to U,\mathscr{D}_{d-1}^{\times}) \hfill\cr
\hfill{}\simeq \Map_{\FEmb_{d-2}}(\iota^*_{d-2}(\RR^{d-1}\times U\to U),\mathscr{E}^{\times}_{d-2})\times_{\Funmon(O_{-1,d-1},\mathscr{D})}\Funmon(H_{0,d-1},\mathscr{D})_{\unitable}.\quad\cr}$$
Observe that under this equivalence, the map $\bar{e}_d$ reverses the direction of $(d-1)$-morphisms.
This amounts to reversing arrows in the $(d-1)$st direction for the target category, i.e., the ordering on simplices in the $(d-1)$st simplicial direction.
In order to make the notation compact, we set 
$$Q=\Map_{\FEmb_{d-2}}(\iota^*_{d-2}(\RR^{d-1}\times U\to U),\mathscr{E}^{\times}_{d-2})\times_{\Funmon(O_{-1,d-1},\mathscr{D})}\Funmon(H_{0,d-1},\mathscr{D})_{\unitable}$$
and 
$$Q^\op=\Map_{\FEmb_{d-2}}(\iota^*_{d-2}(\RR^{d-1}\times U\to U),\mathscr{E}^{\times}_{d-2})\times_{\Funmon(O_{-1,d-1},\mathscr{D^\op})}\Funmon(H_{0,d-1},\mathscr{D}^\op)_\unitable,\eqlabel{qop}$$
where the superscript $\op$ indicates that arrows are reversed in the $(d-1)$st direction.
Substituting $Q$ and $Q^\op$ into \cref{pullbackorth} and expanding one of the homotopy pullback squares, we obtain homotopy pullback squares
$$\xymatrix@C=0cm{ 
Q^\op\ar[d]^-{\bar\delta} & \ar[l]_-{\bar f_d} \Map_{\FEmb_{d-1}}(\iota_{d-1}^*(\RR^{d}\times U\to U ),\mathscr{D}_{d-1}^{\times})\ar[d]^-{f_d}\\
\Map_{\FEmb_{d-1}}(\iota^*_{d-2}(\RR^{d-1}\times U\to U),\mathscr{E}_{d-2}^{\times})\ar[d] &  Q \ar[l]_-{\delta}\ar[d]\\
\Funmon(O_{-1,d-1},\mathscr{D}) & \ar[l]\Funmon(H_{0,d-1},\mathscr{D})_{\unitable}.
}$$
The maps $f_d$ and $\bar f_d$ are induced by the maps $e_d$ and $\bar e_d$.
The maps $\delta$ and $\bar\delta$ are induced by the maps $\epsilon$ and $\bar\epsilon$ in a similar manner.
Hence, $$\Map_{\FEmb_{d-1}}(\iota^*_{d-1}(\RR^{d}\times U\to U),\mathscr{D}^{\times}_{d-1}) \simeq Q^\op \times_{\Funmon(O_{-1,d-1},\mathscr{D})}\Funmon(H_{0,d-1},\mathscr{D})_\unitable.$$
Plugging back into \eqref{unitwithmap}, we have an equivalence
$$\Funmon(B_1,\mathscr{C})\simeq  (Q^\op \times_{\Funmon(O_{-1,d-1},\mathscr{D})}\Funmon(H_{0,d-1},\mathscr{D})_\unitable)\times_{\Funmon(O_{-1},\mathscr{D})}\Funmon(H_{0},\mathscr{C})_\unitable. \eqlabel{b1trippull}$$

\paragraph{Step 3 (restriction to the left adjoint):} Referring back to \cref{qop}, we focus on the triple pullback in the right side of \cref{b1trippull}:
$$T_{\unitable}≔\left(\Funmon(H_{0,d-1},\mathscr{D}^\op)_\unitable\times_{\Funmon(O_{-1,d-1},\mathscr{D}^\op)}\Funmon(H_{0,d-1},\mathscr{D})_\unitable\right)\times_{\Funmon(O_{-1},\mathscr{D})}\Funmon(H_{0},\mathscr{C})_\unitable,$$
which we consider as an object in the slice category over $\Funmon(O_{-1,d-1},\mathscr{D})$. We also define 
$$T≔ \left(\Funmon(H_{0,d-1},\mathscr{D}^\op)\times_{\Funmon(O_{-1,d-1},\mathscr{D}^\op)}\Funmon(H_{0,d-1},\mathscr{D})\right)\times_{\Funmon(O_{-1},\mathscr{D})}\Funmon(H_{0},\mathscr{C})$$
as the same pullback, but without the unit conditions.

We rewrite
$$T=\left(\Funmon(H_{0,d-1}^\op,\mathscr{D})\times_{\Funmon(O_{-1,d-1},\mathscr{D})}\Funmon(H_{0,d-1},\mathscr{D})\right)\times_{\Funmon(O_{-1},\mathscr{D})}\Funmon(H_{0},\mathscr{C}),$$
using the fact that $O_{-1,d-1}$ only has identity $(d-1)$-morphisms.

Using the (derived) adjunction ${\rm c}_{\{d\}}\dashv {\rm ev}_{\{d\}}$
and the fact that ${\rm c}_{\{d\}}(H_{0,d-1})$ is a subobject of $\BBord_d^{\RR^d⨯U→U}$,
we have a weak equivalence of (derived) mapping objects 
$$\Funmon\left({\rm c}_{\{d\}}(H_{0,d-1}^\op⊔_{O_{-1,d-1}} H_{0,d-1})⊔_{O_{-1}}H_0,\mathscr{C}\right)→T. \eqlabel{pushouttoT}$$

We have a weak equivalence
$$H_{0,d-1}^\op→H_{d-1,d-1}$$
that changes the sign of the $(d-1)$st coordinate for the embedding into~$\RR^{d-1}$
and also exchanges the order of $C_{<0}$ and $C_{>0}$ in every cut.
Substituting $H_{d-1,d-1}$ instead of $H_{0,d-1}^\op$ in \cref{pushouttoT}
and taking into account that the maps
$$O_{-1,d-1}→H_{0,d-1}^\op,\qquad O_{-1}→\left({\rm c}_{\{d\}}(H_{0,d-1}^\op⊔_{O_{-1,d-1}} H_{0,d-1})⊔_{O_{-1}}H_0,\mathscr{C}\right)$$
both involve a change of the sign of the $(d-1)$st coordinate for the embedding into~$\RR^{d-1}$,
we have a weak equivalence of (derived) mapping objects
$$\Funmon\left({\rm c}_{\{d\}}(H_{d-1,d-1}⊔_{O_{-1,d-1}} H_{0,d-1})⊔_{O_{-1}}H_0,\mathscr{C}\right)→T, \eqlabel{pushouttoT2}$$
where no sign changes are present anymore.

Since homotopy pushouts commute with disjoint unions,
by inspection of individual connected components of $H_{0,d-1}$, $H_{d-1,d-1}$, $O_{-1,d-1}$, $O_{-1}$, and $H_0$,
we have a weak equivalence in $\smcatdual_{∞,d}$:
$${\rm c}_{\{d\}}(H_{d-1,d-1}⊔_{O_{-1,d-1}} H_{0,d-1})⊔_{O_{-1}}H_0→{\cal H}_0. \eqlabel{pushoutunitsd}$$
Here ${\cal H}_0$ denotes the subobject of $\BBord_d^{\RR^d⨯U→U}$
given by the union of $H_0$, $H_{0,d-1}$, and $H_{d-1,d-1}$.
Observe that ${\cal H}_0$ contains simplices that decompose the boundary of the $(d-1)$-disk, in addition to the $d$-dimensional disk and its boundary.

Thus, we have a weak equivalence of (derived) internal homs
$$\Funmon({\cal H}_0,\mathscr{C})→\Funmon\left({\rm c}_{\{d\}}(H_{d-1,d-1}⊔_{O_{-1,d-1}} H_{0,d-1})⊔_{O_{-1}}H_0,\mathscr{C}\right).\eqlabel{combinetopushout}$$
Let $\Funmon({\cal H}_0,\mathscr{C})_{\unitable}\subset \Funmon({\cal H}_0,\mathscr{C})$ denote the subobject of connected components defined by the property that the restriction to $H_{0,d-1}$, $H_{d-1,d-1}$, and $H_0$, via the canonical maps into the pushout \cref{pushoutunitsd}, factors through the corresponding $\unitable$ subobject. 
That is, the unit condition is satisfied for the handle of index 0 in the top dimension, as well as the index~0 handle and its dual handle (of index $d-1$) in dimension $d-1$. 
Combining \cref{combinetopushout} and \cref{pushouttoT} together and imposing the unit condition, we get the following weak equivalence
in the slice category over $\Funmon(O_{-1,d-1},\mathscr{D})$:
$$\Funmon({\cal H}_0,\mathscr{C})_{\unitable}→
T_{\unitable}.$$
Denote by ${\cal  H}_1$ the union of ${\cal  H}_0$ and~$H_1$.
The map
$$\Funmon({\cal H}_1,\mathscr{C})
→\Funmon({\cal H}_0,\mathscr{C})_{\unitable}$$
is a weak equivalence.
Indeed, the codimension~1 cells of ${\cal H}_0$ and ${\cal H}_1$ are the same,
whereas ${\cal  H}_1$ has an additional codimension~0 cell corresponding to the counit.
The presence of a counit forces the other morphisms to fit into an adjunction.
Furthermore, the canonical restriction map 
$$\Funmon({\cal H}_0,\mathscr{C})_{\unitable}\to \Funmon(H_{0,d-1},\mathscr{C})_{\unitable}$$
is a weak equivalence, since the only cells present in ${\cal H}_0$ that are not present in $H_0$ are sent to the right adjoint and unit of an adjunction in $\mathscr{C}$, by definition (see \cref{discwbound}).  
Combining these observations, we see that the projection of $T_{\unitable}$ to the second factor $\Funmon(H_{0,d-1},\mathscr{C})_{\unitable}$ is an equivalence. 

\begin{figure}[ht]
\centering
\begin{tikzpicture}
\node (a) at (0,0) {$\bullet$};
\node (b) at (1,-.5) {$\bullet$};

\draw (a.center) to[out=85,in=90,looseness=3] (b.center);
\draw (a.center) to[out=-90,in=-95,looseness=3] (b.center);
\draw (.4,.78) to[out=185,in=180,looseness=2.5] (.55,-1.27);

\node at (.8,1) {$f$};
\node at (.8,-1.4) {$g$};
\node at (-1.4,-.2) {$\emptyset$};
\node at (-.8,.8) {$\eta$};

\draw[dashed] (.5,-2) -- (.5,2);
\draw[dashed] (.5,-2) -- (.5,2);
\draw[dashed] (-1,.5) -- (2,-1);
\end{tikzpicture}
\caption{A disc with decomposed boundary, which is present in ${\cal H}_0$. The top half of the boundary $f$ is contained in $H_{0,d-1}$. The unit condition ensures that $\eta$ is the unit of an adjunction, with a left adjoint given by $f$ and right adjoint given by $g$. The embedding of $\eta$ is given by projecting onto the plane whose axes are the dashed lines. The embeddings of $f$ and $g$ are given by projecting onto the horizontal axis.} \label{discwbound}
\end{figure}
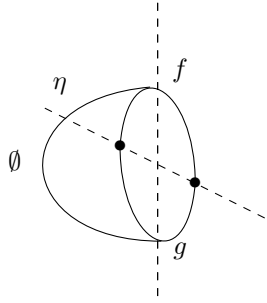

\paragraph{Step 4 (final assembly):} Finally, using \cref{contractiblecylinders,dimensional.reduction} once more in dimension $d-1$ and the induction hypothesis, it follows that we have equivalences 
\begin{align*}
\Funmon(B_1,\mathscr{C})&\simeq\Map_{\FEmb_{d-1}}(\iota^*_{d-2}(\RR^{d-1}\times U),\mathscr{E}^{\times}_{d-2})\times_{\Funmon(O_{-1,d-1},\mathscr{D})}\Funmon(H_{0,d-1},\mathscr{D})_{\unitable}\cr
&\simeq \Funmon(B_{1,d-1},\mathscr{D}).
\end{align*}
Using the equivalences $B_1\simeq \BBord_{d,}^{\RR^d\times U\to U}$ and $B_{1,d-1}\simeq \BBord_{d-1,}^{\RR^{d-1}\times U\to U}$, 
we get that 
$$
\Funmon\left(\BBord_{d}^{\RR^d\times U\to U},\mathscr{C}\right)\simeq \Funmon(B_1,\mathscr{C}), \quad \Funmon\left(\BBord_{d-1}^{\RR^{d-1}\times U\to U},{\rm ev}_{\{d\}}\mathscr{C}\right)\simeq \Funmon(B_{1,d-1},\mathscr{D}).
$$
Hence, the induction hypothesis implies 
$$
\Funmon\left(\BBord_{d}^{\RR^d\times U\to U},\mathscr{C}\right)\simeq \Funmon\left(\BBord^{\RR^{d-1}\times U\to U}_{d-1},{\rm ev}_{\{d\}}\mathscr{C}\right)\overset{\simeq}{\to} \Map(U,{\rm ev}_{\{d\}}\mathscr{C}^{\times})= \Map(U,\mathscr{C}^{\times}),
$$
which completes the induction.
\end{proof}

Now that $\mathscr{C}^{\times}$ has been promoted to a sheaf on $\FEmb_d$, via the framed cobordism hypothesis, we can prove the following general version of the geometric cobordism hypothesis.
Given the main theorem of Grady–Pavlov \cite{GradyPavlov}, the proof is rather trivial and just depends on a simply homotopy cocontinuity argument. 

Recall the notations $\mathscr{C}^⨯$ of \cref{coreadj}, ${\rm c}_S$ of \cref{evconst},
$\Funmon$ of \cref{ihom.cat}, $\Map$ of \cref{power.cat}.

\begin{theorem}[The geometric cobordism hypothesis]\label{mainthm.geometric}
Fix $d\geq 0$, a fibrant object $\mathscr{C}\in\smcatdual_{∞,d}$ (\cref{smcatdual}), and ${\cal S}\in \sPSh(\FEmb_d)$ (\cref{fembdef}).
Let $\mathscr{C}^{\times}_{d}$ be the (fibrant) simplicial presheaf in \cref{extcx}. 
The smooth symmetric monoidal $(∞,d)$-category
$\Funmon(\BBord_d^{\cal S},\mathscr{C})$ is a smooth symmetric monoidal $∞$-groupoid, i.e., the inclusion of the core yields a weak equivalence in $\smcatdual_{∞,d}$:
$${\rm c}_{\{1,\ldots,d\}}\left(\Funmon(\BBord_d^{\cal S},\mathscr{C})^\times\right)\overset{\simeq}{\longrightarrow} \Funmon(\BBord_d^{\cal S},\mathscr{C}).$$
Furthermore, we have a natural weak equivalence (in fact, an isomorphism)
$$\Funmon(\BBord_d^{\cal S},\mathscr{C})^\times\simeq \Map_{\FEmb_d}({\cal S},\mathscr{C}^{\times}_{d}),$$
where $\Map_{\FEmb_d}(-,-)$ denotes the mapping object from \cref{power.end.cat}. 
\end{theorem}

\begin{proof}
Both sides compute derived mapping objects,
since $\mathscr{C}$ is fibrant by assumption,
$\mathscr{C}^⨯_d$ is fibrant by construction,
and all objects are cofibrant in the appropriate local injective model structure.
Since $\Funmon$ and $\Map_{\FEmb_d}$ are right Quillen bifunctors
and $\BBord_d$ is a left Quillen functor by Grady--Pavlov \cite[\ecref{EL-mainthm}]{GradyPavlov},
both sides define simplicial right Quillen functors
$$\sPSh(\FEmb_d)^\op→\sPSh(\cartsp⨯Γ),$$
and when ${\cal S}$ is a representable simplicial presheaf on~$\FEmb_d$,
both sides are isomorphic by definition of $\mathscr{C}^⨯_d$ and the enriched Yoneda lemma.
Hence both sides compute derived mapping objects and are isomorphic for any ${\cal S}$.
\end{proof}

\begin{remark}
Alternatively to the above proof of \cref{mainthm.geometric}, we have a chain of natural weak equivalences
\begin{align*}
\Funmon(\BBord_d^{\cal S},\mathscr{C})&\overset{\bf 1}{\simeq} 
\Funmon(\BBord_d^{\hocolim_{\RR^d\times U\to {\cal S}}(\RR^d\times U\to U)}\mathscr{C})\\
&\overset{\bf 2}{\simeq}  \holim_{\RR^d\times U\to {\cal S}}\Funmon(\BBord_d^{\RR^d\times U\to U},\mathscr{C})\\
&=:  \holim_{\RR^d\times U\to {\cal S}}\mathscr{C}_{d}^{\times}(\RR^d\times U\to U)\\
&\overset{\bf 3}{\simeq}  \Map_{\FEmb_{d}}({\cal S},\mathscr{C}_{d}^{\times}).
\end{align*}
The equivalence ${\bf 1}$ follows by writing ${\cal S}$ as a homotopy colimit over representables.
The equivalence ${\bf 2}$ follows from Grady--Pavlov \cite[\ecref{EL-mainthm}]{GradyPavlov}.
The equivalence ${\bf 3}$ follows for formal reasons.
Indeed, recall that the evaluation is the partial evaluation defined in \cref{partials}. Equivalently, 
$$\mathscr{C}^{\times}(\RR^d\times U\to U)\cong \Map_{\FEmb_{d}}(\RR^d\times U\to U,\mathscr{C}_d^{\times}).$$
Hence, the equivalence follows from homotopy continuity of the derived bifunctor $\Map_{\FEmb_d}(-,-)$. 
\end{remark}

As a special case of presheaves constant in the direction of $\cartsp$, we obtain a proof of the topological cobordism hypothesis.

\begin{theorem}[The topological cobordism hypothesis]
\label{top.cob.hyp}
Fix $d\geq 0$ and a symmetric monoidal $(∞,d)$-category with duals given by a fibrant object $\mathscr{C}\in\catdual_{∞,d}$.
Fix ${\cal S}\in \sPSh(\Emb_d)$, a simplicial presheaf on the site $\Emb_d$ of $d$-dimensional manifolds and open embeddings,
which equivalently can be specified (\cref{fiberwiseshape}) as a simplicial set with an action of the simplicial group $\Sing\O(d)\simeq\Sing\GL(d)$.
Let $\mathscr{C}^{\times}_{d}$ be the presheaf of $Γ$-spaces on $\Emb_d$, defined like in \cref{extcx}, but without involving the site $\cartsp$:
$$\mathscr{C}^{\times}_{d}(V)≔\Funmon(\BBord_d^V,\mathscr{C})^{\times},\qquad V∈\Emb_d.$$
Then, the inclusion
$${\rm c}_{\{1,\ldots,d\}}(\Funmon(\BBord_d^{\cal S},\mathscr{C})^{\times})→\Funmon(\BBord_d^{\cal S},\mathscr{C})$$
is a natural weak equivalence for all ${\cal S}∈\sPSh(\Emb_d)$.
We have a weak equivalence
$$\mathscr{C}^⨯_d(\RR^d)≃\mathscr{C}^⨯,$$ which is not natural in $\Emb_d$.
We also have a natural weak equivalence
$$\Funmon(\BBord_d^{{\cal S}},\mathscr{C})^⨯\simeq \Map_{\Emb_d}({\cal S},\mathscr{C}^{\times}_d),$$
where again $\Map_{\Emb_d}(-,-)$ denotes the mapping object from \cref{power.end.cat}.
\end{theorem}

\begin{proof}
We invoke \cref{mainthm.geometric} in the special case when the presheaves on $\cartsp$ and $\FEmb_d$
are homotopy constant with respect to maps in $\cartsp$.
In detail, the framed geometric cobordism hypothesis (\cref{mainthm}) implies that for the representable $\RR^d\to \RR^0$, we have an equivalence of $\Gamma$-spaces
$$\Funmon(\BBord_d^{ \RR^d\to \RR^0},\mathscr{C})^⨯ \simeq \mathscr{C}^{\times}.$$
Let $i:\Emb_d\into \FEmb_d$ denote the canonical inclusion.
The general case follows by homotopy cocontinuity of the bordism category, using Grady--Pavlov \cite[\ecref{EL-mainthm}]{GradyPavlov}, along with homotopy cocontinuity of the left derived functor of the left Kan extension $i_!:\sPSh(\Emb_d)\to \sPSh(\FEmb_d)$.
\end{proof}

In \cref{top.cob.hyp}, it is not necessary to assume that the geometric structure is constant in the cartesian space direction in order to obtain topological field theories in the sense of Lurie \cite{Lurie.TFT}. In fact, if we assume that only the target category $\mathscr{C}$ is locally constant in the cartesian space direction, we still obtain topological field theories. To see this, we recall the following from Schreiber \cite{Schreiber}.

\begin{definition}
\label{shape.functor}
Recall the cohesive adjunction between simplicial presheaves on $\cartsp$ and $\sset$ (see Grady–Pavlov \cite[\ecref{EL-shapebordism}]{GradyPavlov} for a review using our notation).
In particular, we have a simplicial Quillen adjunction $(\csp\dashv \Delta)$
\begin{equation}\label{cohesiveadj}
\xymatrix{
\sPSh(\cartsp)_{\loc}\ar@<.1cm>[r]^-{\csp} & \ar@<.1cm>[l]^-{\Delta} \sset.
}
\end{equation}
We refer to (the left derived functor of) the left adjoint functor $\csp$ as the \emph{shape} functor. 
\end{definition}

\begin{remark}\label{shapeequivariant}
The above adjunction continues to hold if we take presheaves with values in a left Bousfield localization of simplicial presheaves.
For example, the adjunction $\csp\dashv \Delta$ induces a corresponding adjunction 
$$\xymatrix{
\sPSh(\cartsp,\sset^{\sing \GL(d)})_{\loc}\ar@<.1cm>[r]^-{\csp} & \ar@<.1cm>[l]^-{\Delta} \sset^{\sing \GL(d)}.
}$$ 
\end{remark}

\begin{proposition}\label{equiadj}
Fix $d\geq 0$ and a smooth symmetric monoidal $(∞,d)$-category with duals given by a fibrant object $\mathscr{C}\in\smcatdual_{∞,d}$.
Suppose that $\mathscr{C}$ is homotopy constant with respect to $\cartsp$, i.e., for all $U\in \cartsp$ the canonical map $U\to \RR^0$ induces a weak equivalence $\mathscr{C}(\RR^0)\to \mathscr{C}(U)$.
Then we have a natural weak equivalence
$$\Funmon(\BBord_d^{{\cal S}},\mathscr{C})^⨯\simeq \Map(\ldf\csp({\cal S}),\mathscr{C}_d^{\times})^{\GL(d)},$$
where $\ldf\csp$ is the left derived functor of $\csp$ in \cref{shapeequivariant}
and the superscript $\GL(d)$ denotes the homotopy fixed points object.
\end{proposition}

\begin{proof}
Observe that when $\mathscr{C}$ is homotopy constant, $\mathscr{C}^{\times}_d$ is homotopy constant, i.e. for all $U\in \cartsp$ the projection map $\RR^d\times U\to \RR^d$ induces an weak equivalence $\mathscr{C}_d^{\times}(\RR^d\to \RR^0)\to \mathscr{C}_d^{\times}(\RR^d\times U\to U)$.
Using the equivalence with presheaves of $\GL(d)$-equivariant simplicial sets, $\mathscr{C}^{\times}_d$ can be identified with a presheaf of $\GL(d)$-equivariant simplicial sets that is weakly equivalent to an object in the essential image of $\Delta$.
With this observation, the claim follows at once from the geometric cobordism hypothesis (\cref{mainthm.geometric}), along with the Quillen adjunction $\csp\dashv \Delta$ in \cref{shapeequivariant}. 
\end{proof}

\begin{remark}
\label{lurie.formulation}
\cref{equiadj} can be seen as recovering Lurie's formulation of the (topological) cobordism hypothesis \cite[Theorem~2.4.18]{Lurie.TFT}.
Indeed, the simplicial groups $\GL(d)$ and $\O(d)$ are weakly equivalent,
the action of $\O(d)$ on $\mathscr{C}^⨯$ coincides (by definition) with the action defined by Lurie \cite[Corollary~2.4.10]{Lurie.TFT},
and the $\O(d)$-space $\ldf\csp({\cal S})$ is weakly equivalent to $\tilde X$ in Lurie's setting.
\end{remark}

\subsection{Handles in the geometrically framed bordism category}
\label{handles.duals}

We now turn to the proof of supporting lemmas for the framed case. The present section is devoted to showing that the index $k$ handle and index $k-1$ handle participate in an adjunction. 
The triangle identity is witnessed by a family of generalized Morse functions with a single birth death singularity. 
The proof of this fact is surprisingly subtle and requires us to move down to codimension 2 in order to see the adjunction. 
In particular, we must obtain the canonical bordisms \cref{counitforreal}, \cref{unitforreal}, and \cref{triangleid} from the construction of Milnor \cite[Assertion~3.4]{Milnor}. This is explained in \cref{milnor.construction}. 

The following table summarizes the bordisms that appear in the present section.
The left column is the subobject of the bordism category that contains the bordisms described in the center column, which are depicted in the right column. 
\begin{center}
\begin{tabular}{c|p{5cm}|c}
Notation & \centering bordisms & illustration for $d=3,k=2$
 \\ \hline
$\Hb_k$ & Index $k$ handle with ``tails'', index $k-1$ handle with ``tails'', and handle cancellation with ``tails''. & 
\input{hbar.tikz}
\\
$\Ob_{k-1}$ & Index $k-1$ handle with ``tails''.&  \begin{tikzpicture}[scale=.2, baseline={(current bounding box.center)}]
	\begin{pgfonlayer}{nodelayer}
		\node [style=none] (144) at (-6.5, 9.5) {};
		\node [style=none] (145) at (-7, 8.5) {};
		\node [style=none] (146) at (-5.75, 8.5) {};
		\node [style=none] (147) at (-6.5, 4) {};
		\node [style=none] (148) at (-7, 5) {};
		\node [style=none] (149) at (-5.75, 5) {};
		\node [style=none] (150) at (-3, 8.5) {};
		\node [style=none] (151) at (-4.25, 8.5) {};
		\node [style=none] (152) at (-3, 5) {};
		\node [style=none] (153) at (-4.25, 5) {};
		\node [style=none] (154) at (-3.5, 7.5) {};
		\node [style=none] (155) at (-3.5, 6) {};
		\node [style=none] (156) at (-3.75, 9.5) {};
		\node [style=none] (157) at (-6.25, 9.5) {};
		\node [style=none] (158) at (-5.75, 8.5) {};
		\node [style=none] (159) at (-7, 8.5) {};
		\node [style=none] (160) at (-3.75, 4) {};
		\node [style=none] (161) at (-6.25, 4) {};
		\node [style=none] (162) at (-5.75, 5) {};
		\node [style=none] (163) at (-7, 5) {};
		\node [style=none] (164) at (-4.25, 8.5) {};
		\node [style=none] (165) at (-3, 8.5) {};
		\node [style=none] (166) at (-4.25, 5) {};
		\node [style=none] (167) at (-3, 5) {};
		\node [style=none] (168) at (-3.75, 7.5) {};
		\node [style=none] (169) at (-3.75, 6) {};
		\node [style=none] (170) at (-3.5, 9.5) {};
		\node [style=none] (171) at (-1, 9.5) {};
		\node [style=none] (172) at (-1.5, 8.5) {};
		\node [style=none] (173) at (-0.25, 8.5) {};
		\node [style=none] (174) at (-3.5, 4) {};
		\node [style=none] (175) at (-1, 4) {};
		\node [style=none] (176) at (-1.5, 5) {};
		\node [style=none] (177) at (-0.25, 5) {};
		\node [style=none] (178) at (-0.75, 9.5) {};
		\node [style=none] (179) at (-0.25, 8.5) {};
		\node [style=none] (180) at (-1.5, 8.5) {};
		\node [style=none] (181) at (-0.75, 4) {};
		\node [style=none] (182) at (-0.25, 5) {};
		\node [style=none] (183) at (-1.5, 5) {};
		\node [style=none] (184) at (-6.5, 0.25) {};
		\node [style=none] (185) at (-7, -0.75) {};
		\node [style=none] (186) at (-5.75, -0.75) {};
		\node [style=none] (187) at (-6.5, -5.25) {};
		\node [style=none] (188) at (-7, -4.25) {};
		\node [style=none] (189) at (-5.75, -4.25) {};
		\node [style=none] (190) at (-6.25, 0.25) {};
		\node [style=none] (191) at (-5.75, -0.75) {};
		\node [style=none] (192) at (-7, -0.75) {};
		\node [style=none] (193) at (-6.25, -5.25) {};
		\node [style=none] (194) at (-5.75, -4.25) {};
		\node [style=none] (195) at (-7, -4.25) {};
		\node [style=none] (196) at (-0.75, 0.25) {};
		\node [style=none] (197) at (-1.25, -0.75) {};
		\node [style=none] (198) at (0, -0.75) {};
		\node [style=none] (199) at (-0.75, -5.25) {};
		\node [style=none] (200) at (-1.25, -4.25) {};
		\node [style=none] (201) at (0, -4.25) {};
		\node [style=none] (202) at (-0.5, 0.25) {};
		\node [style=none] (203) at (0, -0.75) {};
		\node [style=none] (204) at (-1.25, -0.75) {};
		\node [style=none] (205) at (-0.5, -5.25) {};
		\node [style=none] (206) at (0, -4.25) {};
		\node [style=none] (207) at (-1.25, -4.25) {};
		\node [style=none] (208) at (-3.75, 3.25) {};
		\node [style=none] (209) at (-3.75, 0.75) {};
		\node [style=none] (215) at (-3.5, -8.25) {$\xi$};
	\end{pgfonlayer}
	\begin{pgfonlayer}{edgelayer}
		\draw [bend left=90, looseness=2.75] (145.center) to (146.center);
		\draw [bend right=90, looseness=2.75] (148.center) to (149.center);
		\draw (145.center) to (148.center);
		\draw (146.center) to (149.center);
		\draw [bend right=90, looseness=2.75] (150.center) to (151.center);
		\draw [bend left=90, looseness=2.75] (150.center) to (151.center);
		\draw [bend right=90, looseness=2.75] (152.center) to (153.center);
		\draw [bend left=90, looseness=2.75] (152.center) to (153.center);
		\draw [bend right=75, looseness=1.75] (154.center) to (155.center);
		\draw (156.center) to (157.center);
		\draw [bend right=90, looseness=2.75] (158.center) to (159.center);
		\draw (160.center) to (161.center);
		\draw [bend left=90, looseness=2.75] (162.center) to (163.center);
		\draw (158.center) to (162.center);
		\draw (159.center) to (163.center);
		\draw (153.center) to (163.center);
		\draw (152.center) to (162.center);
		\draw (151.center) to (159.center);
		\draw (150.center) to (158.center);
		\draw [bend left=90, looseness=2.75] (164.center) to (165.center);
		\draw [bend right=90, looseness=2.75] (164.center) to (165.center);
		\draw [bend left=90, looseness=2.75] (166.center) to (167.center);
		\draw [bend right=90, looseness=2.75] (166.center) to (167.center);
		\draw [bend left=75, looseness=1.75] (168.center) to (169.center);
		\draw (170.center) to (171.center);
		\draw [bend left=90, looseness=2.75] (172.center) to (173.center);
		\draw (174.center) to (175.center);
		\draw [bend right=90, looseness=2.75] (176.center) to (177.center);
		\draw (172.center) to (176.center);
		\draw (173.center) to (177.center);
		\draw (167.center) to (177.center);
		\draw (166.center) to (176.center);
		\draw (165.center) to (173.center);
		\draw (164.center) to (172.center);
		\draw [bend right=90, looseness=2.75] (179.center) to (180.center);
		\draw [bend left=90, looseness=2.75] (182.center) to (183.center);
		\draw (179.center) to (182.center);
		\draw (180.center) to (183.center);
		\draw [bend left=90, looseness=2.75] (185.center) to (186.center);
		\draw [bend right=90, looseness=2.75] (188.center) to (189.center);
		\draw (185.center) to (188.center);
		\draw (186.center) to (189.center);
		\draw [bend right=90, looseness=2.75] (191.center) to (192.center);
		\draw [bend left=90, looseness=2.75] (194.center) to (195.center);
		\draw (191.center) to (194.center);
		\draw (192.center) to (195.center);
		\draw [bend left=90, looseness=2.75] (197.center) to (198.center);
		\draw [bend right=90, looseness=2.75] (200.center) to (201.center);
		\draw (197.center) to (200.center);
		\draw (198.center) to (201.center);
		\draw [bend right=90, looseness=2.75] (203.center) to (204.center);
		\draw [bend left=90, looseness=2.75] (206.center) to (207.center);
		\draw (203.center) to (206.center);
		\draw (204.center) to (207.center);
		\draw (190.center) to (196.center);
		\draw (191.center) to (203.center);
		\draw (194.center) to (206.center);
		\draw (193.center) to (199.center);
		\draw [style=new edge style 0] (208.center) to (209.center);
		\draw [style=new edge style 2] (192.center) to (204.center);
		\draw [style=new edge style 2] (195.center) to (207.center);
	\end{pgfonlayer}
\end{tikzpicture}
 \\
 $H_{k}$ & Index $k$ handle, index $k-1$ handle, and handle cancellation. & \input{h.tikz}
\\
$O_{k-1}$ & Index $k-1$ handle & \begin{tikzpicture}[scale=.25,baseline=(current bounding box.center)]
	\begin{pgfonlayer}{nodelayer}
		\node [style=none] (145) at (-8, 8.75) {};
		\node [style=none] (146) at (-6.75, 8.5) {};
		\node [style=none] (148) at (-8, 5.25) {};
		\node [style=none] (149) at (-6.75, 5) {};
		\node [style=none] (150) at (-4.5, 8.5) {};
		\node [style=none] (151) at (-5.75, 8.75) {};
		\node [style=none] (152) at (-4.5, 5) {};
		\node [style=none] (153) at (-5.75, 5.25) {};
		\node [style=none] (154) at (-5, 7.75) {};
		\node [style=none] (155) at (-5, 6.25) {};
		\node [style=none] (158) at (-6.75, 8.5) {};
		\node [style=none] (159) at (-8, 8.75) {};
		\node [style=none] (162) at (-6.75, 5) {};
		\node [style=none] (163) at (-8, 5.25) {};
		\node [style=none] (164) at (-5.75, 8.75) {};
		\node [style=none] (165) at (-4.5, 8.5) {};
		\node [style=none] (166) at (-5.75, 5.25) {};
		\node [style=none] (167) at (-4.5, 5) {};
		\node [style=none] (168) at (-5.25, 7.75) {};
		\node [style=none] (169) at (-5.25, 6.25) {};
		\node [style=none] (172) at (-3.5, 8.75) {};
		\node [style=none] (173) at (-2.25, 8.5) {};
		\node [style=none] (176) at (-3.5, 5.25) {};
		\node [style=none] (177) at (-2.25, 5) {};
		\node [style=none] (179) at (-2.25, 8.5) {};
		\node [style=none] (180) at (-3.5, 8.75) {};
		\node [style=none] (182) at (-2.25, 5) {};
		\node [style=none] (183) at (-3.5, 5.25) {};
		\node [style=none] (185) at (-8, -0.5) {};
		\node [style=none] (186) at (-6.75, -0.75) {};
		\node [style=none] (188) at (-8, -4) {};
		\node [style=none] (189) at (-6.75, -4.25) {};
		\node [style=none] (191) at (-6.75, -0.75) {};
		\node [style=none] (192) at (-8, -0.5) {};
		\node [style=none] (194) at (-6.75, -4.25) {};
		\node [style=none] (195) at (-8, -4) {};
		\node [style=none] (197) at (-3.25, -0.5) {};
		\node [style=none] (198) at (-2, -0.75) {};
		\node [style=none] (200) at (-3.25, -4) {};
		\node [style=none] (201) at (-2, -4.25) {};
		\node [style=none] (203) at (-2, -0.75) {};
		\node [style=none] (204) at (-3.25, -0.5) {};
		\node [style=none] (206) at (-2, -4.25) {};
		\node [style=none] (207) at (-3.25, -4) {};
		\node [style=none] (208) at (-5.25, 3.5) {};
		\node [style=none] (209) at (-5.25, 1) {};
		\node [style=none] (215) at (-4.75, -7) {$\eta$};
	\end{pgfonlayer}
	\begin{pgfonlayer}{edgelayer}
		\draw (145.center) to (148.center);
		\draw (146.center) to (149.center);
		\draw [in=255, out=-90, looseness=2.50] (150.center) to (151.center);
		\draw [bend right=75, looseness=1.75] (154.center) to (155.center);
		\draw (158.center) to (162.center);
		\draw (159.center) to (163.center);
		\draw (152.center) to (162.center);
		\draw (151.center) to (159.center);
		\draw (150.center) to (158.center);
		\draw [in=75, out=90, looseness=3.00] (166.center) to (167.center);
		\draw [bend left=75, looseness=1.75] (168.center) to (169.center);
		\draw (173.center) to (177.center);
		\draw (167.center) to (177.center);
		\draw (165.center) to (173.center);
		\draw (164.center) to (172.center);
		\draw (179.center) to (182.center);
		\draw (185.center) to (188.center);
		\draw (186.center) to (189.center);
		\draw (191.center) to (194.center);
		\draw (192.center) to (195.center);
		\draw (198.center) to (201.center);
		\draw (203.center) to (206.center);
		\draw (191.center) to (203.center);
		\draw (194.center) to (206.center);
		\draw [style=new edge style 0] (208.center) to (209.center);
		\draw [style=new edge style 2] (195.center) to (207.center);
		\draw [style=new edge style 2] (204.center) to (207.center);
		\draw [style=new edge style 2] (180.center) to (183.center);
		\draw (192.center) to (204.center);
		\draw [style=new edge style 2] (163.center) to (183.center);
	\end{pgfonlayer}
\end{tikzpicture}
\end{tabular}
\end{center}

We now review the construction of Milnor \cite[Assertion~3.4]{Milnor} and describe the multisimplices in the bordism category that arise from the construction. 
\begin{construction}
\label{milnor.construction}
For each $0\leq k\leq d$, consider the function $f_k:\RR^d\to \RR$ defined by the formula 
$$f_k=-x_1^2-x_2^2+\cdots -x_k^2+x_{k+1}^2+\cdots +x_d^2.$$
Fix a smooth cutoff function $\mu:\RR\to \RR$ satisfying 
$\mu(0)>1$, $\mu(x)=0$ for $x\geq 2$, and $-1<\mu'(x)\leq 0$. Define the modified $F$ function 
$$F_k=f_k-\mu(x_1^2+x_2^2+\cdots +x_k^2+2x_{k+1}^2+\cdots + 2x_d^2), \eqlabel{milnorF}$$
considered in Milnor \cite{Milnor}.
The function $F$ is obtained by modifying $f$ in the ellipsoid $$G(x)=x_1^2+x_2^2+\cdots +x_k^2+2x_{k+1}^2+\cdots + 2x_d^2-2=0.$$
The intersection of this ellipsoid with the level set of $f_k$ at $-1$ is a product of spheres $S^{d-k-1}\times S^{k-1}$.
The function $F$ can be used to construct a canonical bordism of the form \cref{counitforreal}, including the embedding into $\RR^d$.
The preimage $f^{-1}(-1)$ is the composition of two discs, embedded in $\RR^d$.
The preimage $F^{-1}(-1)$ is an identity bordism, embedded in $\RR^d$.
The entire bordism is depicted below
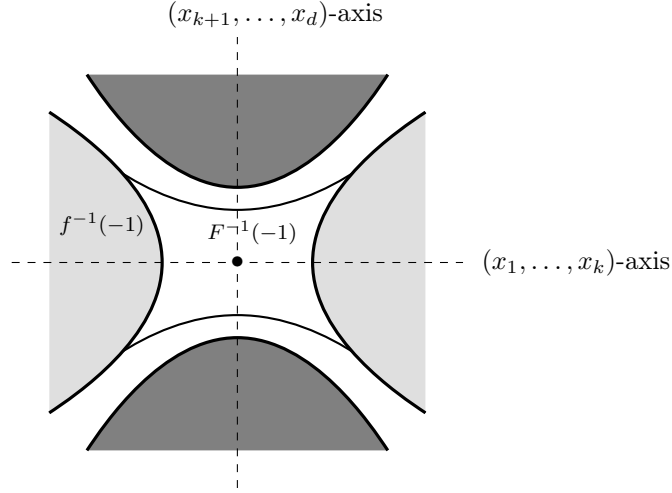
\begin{center}
\begin{tikzpicture}
\draw[fill=gray, very thick] (-2,2.5) parabola bend (0,1) (2,2.5);
\draw[fill=gray, very thick] (-2,-2.5) parabola bend (0,-1) (2,-2.5);
\draw[thick] (-2.5,2) parabola bend (0,.7) (2.5,2);
\draw[thick, rotate=180] (-2.5,2) parabola bend (0,.7) (2.5,2);

\draw[rotate=90, fill=gray!25, very thick] (-2,2.5) parabola bend (0,1) (2,2.5);
\draw[rotate=90, fill=gray!25, very thick] (-2,-2.5) parabola bend (0,-1) (2,-2.5);

\draw[dashed] (-3,0) -- (3,0);
\draw[dashed] (0,-3) -- (0,3);

\node at (0,0) {$\bullet$};
\node at (.2,.4) {\footnotesize{$F^{-1}(-1)$}};
\node at (-1.8,.5) {\footnotesize{$f^{-1}(-1)$}};
\node at (4.5,0) {$(x_1,\ldots,x_k)$-axis};
\node at (.5,3.3) {$(x_{k+1},\ldots,x_d)$-axis};
\end{tikzpicture}
\end{center}
Here $f^{-1}(-1)$ is the union of the two curves bounding the light gray region.
The preimage $F^{-1}(-1)$ is the two curves bounding the white region: the Milnor bridge.
The intersection of the ellipsoid $G^{-1}(0)$ with $f^{-1}(-1)$ is a product of spheres $S^{d-k-1}\times S^{k-1}$.

\paragraph{The Milnor bridge as a multisimplex:}
The above construction yields the bordism with
cuts in the $d$th direction being $f^{-1}(-1)$ and $F^{-1}(-1)$
and cuts in the $(d-1)$st direction being $∅$ (with the entire bordism being inside $C_{>0}$)
and the ellipsoid $G^{-1}(0)$:
$$
\begin{tikzcd}[column sep={6em,between origins}, row sep={2em,between origins}]
	\emptyset && \emptyset \\
	& {D^{d-k}\times D^k} \\
	{S^{d-k-1}\times S^{k-1}} && {S^{d-k-1}\times S^{k-1}}.
	\arrow["{D^{d-k}\times S^{k-1}}"', from=1-1, to=3-1]
	\arrow["{S^{d-k-1}\times D^k}", from=1-3, to=3-3]
	\arrow[from=1-1, to=1-3]
	\arrow[from=3-1, to=3-3]
\end{tikzcd} \eqlabel{biunitwcaps}
$$
Here the horizontal direction corresponds to the $d$th direction and the vertical direction corresponds to the $(d-1)$st direction. The bordism \cref{biunit} can be written as a the following composition in the $(d-1)$st direction, where the top cell is degenerate in the $d$th direction, i.e., the horizontal direction:
$$
\begin{tikzcd}[column sep={6em,between origins}, row sep={2em,between origins}]
	\emptyset && \emptyset \\
	{} &  \\
	{S^{d-k-1}\times S^{k-1}} && {S^{d-k-1}\times S^{k-1}} \\
	& \zeta \\
	{S^{d-k-1}\times S^{k-1}} && {S^{d-k-1}\times S^{k-1}}
	\arrow["{D^{d-k}\times S^{k-1}}"', from=1-1, to=3-1]
	\arrow[from=3-1, to=3-3]
	\arrow["B"', from=3-1, to=5-1]
	\arrow["R", from=3-3, to=5-3]
	\arrow[from=5-1, to=5-3]
	\arrow["{D^{d-k}\times S^{k-1}}", from=1-3, to=3-3]
	\arrow[from=1-1, to=1-3]
\end{tikzcd} \eqlabel{biunit}$$
For $k\geq 2$, the bordism $B$ is a composition of bordisms in the $(d-1)$st and $(d-2)$nd directions. Explicitly, 
$$
B=\begin{tikzcd}[column sep={6.5em,between origins}, row sep={2.5em,between origins}]
	\emptyset && \emptyset && \emptyset \\
	{} & {{\rm id}} && {{\rm id}} \\
	{D^{d-k-1}\times S^{k-2}} && {D^{d-k-1}\times S^{k-2}} && {D^{d-k-1}\times S^{k-2}} \\
	& {D^{d-k}\times D^{k-1}} && {D^{d-k}\times D^{k-1}} \\
	{D^{d-k-1}\times S^{k-2}} && {D^{d-k-1}\times S^{k-2}} && {D^{d-k-1}\times S^{k-2}} \\
	& {{\rm id}} && {{\rm id}} \\
	\emptyset && \emptyset && \emptyset,
	\arrow["{{\rm id}}"{description}, from=3-1, to=3-3]
	\arrow["{S^{d-k-1}\times D^{k-1}}"', from=3-1, to=5-1]
	\arrow["{D^{d-k}\times S^{k-2}}"{description}, from=3-3, to=5-3]
	\arrow["{{\rm id}}"{description}, from=5-1, to=5-3]
	\arrow["{{\rm id}}"{description}, from=3-3, to=3-5]
	\arrow["{{\rm id}}"{description}, from=5-3, to=5-5]
	\arrow["{S^{d-k-1}\times D^{k-1}}", from=3-5, to=5-5]
	\arrow["{D^{d-k-1}\times D^{k-1}}"', from=1-1, to=3-1]
	\arrow[from=1-1, to=1-3]
	\arrow[from=1-3, to=1-5]
	\arrow["{D^{d-k-1}\times D^{k-1}}"{description}, from=1-3, to=3-3]
	\arrow["{D^{d-k-1}\times D^{k-1}}", from=1-5, to=3-5]
	\arrow["{D^{d-k-1}\times D^{k-1}}"', from=5-1, to=7-1]
	\arrow[from=7-1, to=7-3]
	\arrow[from=7-3, to=7-5]
	\arrow["{D^{d-k-1}\times D^{k-1}}"{description}, from=5-3, to=7-3]
	\arrow["{D^{d-k-1}\times D^{k-1}}", from=5-5, to=7-5]
\end{tikzcd} \eqlabel{biunittoeps}
$$
where the $(d-1)$st direction is the horizontal direction and the $d$th direction is the vertical direction. For $k<2$, we have just the composition of the middle two cells. For $k\geq 2$, the bordism $R$ is the composition
$$
R=\begin{tikzcd}[column sep={8em,between origins}, row sep={2em,between origins}]
	\emptyset && \emptyset \\
	{} & {{\rm id}} \\
	{D^{d-k-1}\times S^{k-2}} && {D^{d-k-1}\times S^{k-2}} \\
	{} & {S^{d-k-1}\times D^{k-1}\times[0,1]} \\
	{D^{d-k-1}\times S^{k-2}} && {D^{d-k-1}\times S^{k-2}} \\
	& {{\rm id}} \\
	\emptyset && \emptyset.
	\arrow["{{\rm id}}"{description}, from=3-1, to=3-3]
	\arrow["{S^{d-k-1}\times D^{k-1}}"', from=3-1, to=5-1]
	\arrow["{{\rm id}}"{description}, from=5-1, to=5-3]
	\arrow["{D^{d-k-1}\times D^{k-1}}"', from=1-1, to=3-1]
	\arrow[from=1-1, to=1-3]
	\arrow["{D^{d-k-1}\times D^{k-1}}", from=1-3, to=3-3]
	\arrow["{D^{d-k-1}\times D^{k-1}}"', from=5-1, to=7-1]
	\arrow[from=7-1, to=7-3]
	\arrow["{D^{d-k-1}\times D^{k-1}}", from=5-3, to=7-3]
	\arrow["{S^{d-k-1}\times D^{k-1}}", from=3-3, to=5-3]
\end{tikzcd}
$$
For $k<2$, we have just the center cell. Finally, the bordism $\zeta$ is the bordism obtained by pre and post composing the bordism $\epsilon$ in \cref{counitforreal} with identity on the identity on $D^{d-k-1}\times D^{k-1}$. For an illustration of the case $d=3,k=2$, see the left column in \cref{d3k2bords}.

\paragraph{Handle cancellation:}
Consider also the 1-parameter family of generalized Morse functions $f_{k-1,k}:\RR^d\to \RR$ defined by the formula
$$f_{k-1,k}=-x_1^2-x_2^2-\cdots- x_{k-1}^2 -x_{k}^3-tx_k+x_{k+1}^2+\cdots +x_d^2.$$
At $t=-1$, there is a pair of critical points, one of index $k$ and one of index $k-1$.
At $t=0$ the critical points are canceled through a birth-death singularity.
At $t=1$, there are no critical points. The construction of $F$ construction can be modified to suit $f_{k,k-1}$, by choosing a cutoff functions and modifying $f_{k,k-1}$ in elliptical regions around each critical point. For example, when $k=1$, and $t=-1$, we have the embedded bordism depicted below. 
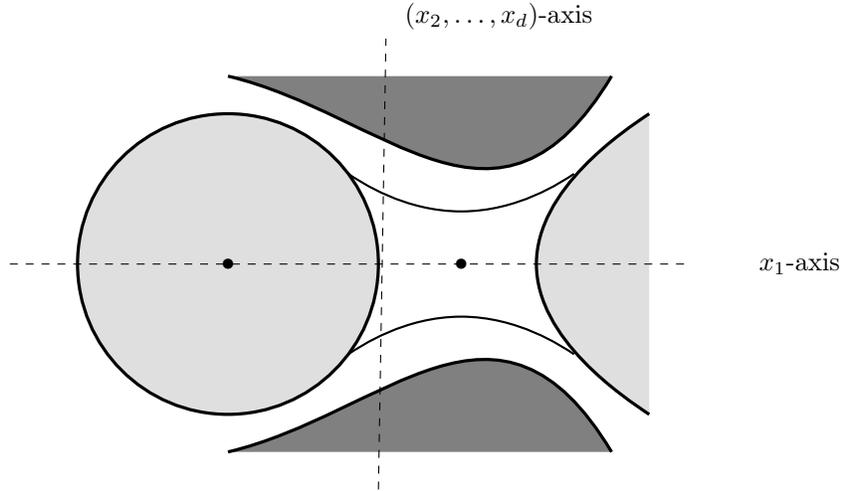
\begin{center}
\begin{tikzpicture}
\draw[fill=gray, very thick] (-3.1,2.5) .. controls (-1,2) and (.5,0) .. (2,2.5);
\draw[fill=gray, very thick] (-3.1,-2.5) .. controls (-1,-2) and (.5,0) .. (2,-2.5);
\draw[thick] (-1.5,1.2) parabola bend (0,.7) (1.5,1.2);
\draw[thick, rotate=180] (-1.5,1.2) parabola bend (0,.7) (1.5,1.2);

\draw[fill=gray!25,very thick] (-3.1,0) circle (2cm);
\draw[rotate=90, fill=gray!25, very thick] (-2,-2.5) parabola bend (0,-1) (2,-2.5);

\draw[dashed] (-6,0) -- (3,0);
\draw[dashed] (-1.1,-3) -- (-1,3);

\node at (0,0) {$\bullet$};
\node at (-3.1,0) {$\bullet$};
\node at (4.5,0) {$x_1$-axis};
\node at (.5,3.3) {$(x_{2},\ldots,x_d)$-axis};
\end{tikzpicture}
\end{center}
2-cell of the form \cref{counitforreal}, along with a canonical embedding into $\RR^d$. 

\begin{figure}[ht]
\begin{center}
\begin{tikzpicture}[scale=.25]
	\begin{pgfonlayer}{nodelayer}
		\node [style=none] (0) at (-9.25, -0.75) {};
		\node [style=none] (1) at (-8, -0.75) {};
		\node [style=none] (2) at (-9.25, -4.25) {};
		\node [style=none] (3) at (-8, -4.25) {};
		\node [style=none] (4) at (-8.75, -1.75) {};
		\node [style=none] (5) at (-8.75, -3.25) {};
		\node [style=none] (6) at (-8.5, 0.25) {};
		\node [style=none] (7) at (-4.75, 0.25) {};
		\node [style=none] (8) at (-5.25, -0.75) {};
		\node [style=none] (9) at (-4, -0.75) {};
		\node [style=none] (11) at (-8.5, -5.25) {};
		\node [style=none] (12) at (-4.75, -5.25) {};
		\node [style=none] (13) at (-5.25, -4.25) {};
		\node [style=none] (14) at (-4, -4.25) {};
		\node [style=none] (15) at (0, -0.75) {};
		\node [style=none] (16) at (-1.25, -0.75) {};
		\node [style=none] (17) at (0, -4.25) {};
		\node [style=none] (18) at (-1.25, -4.25) {};
		\node [style=none] (19) at (-0.5, -1.75) {};
		\node [style=none] (20) at (-0.5, -3.25) {};
		\node [style=none] (21) at (-0.75, 0.25) {};
		\node [style=none] (22) at (-4.5, 0.25) {};
		\node [style=none] (23) at (-4, -0.75) {};
		\node [style=none] (24) at (-5.25, -0.75) {};
		\node [style=none] (25) at (-0.75, -5.25) {};
		\node [style=none] (26) at (-4.5, -5.25) {};
		\node [style=none] (27) at (-4, -4.25) {};
		\node [style=none] (28) at (-5.25, -4.25) {};
		\node [style=none] (29) at (-9.25, 7.75) {};
		\node [style=none] (30) at (-8, 7.75) {};
		\node [style=none] (31) at (-1.25, 7.75) {};
		\node [style=none] (32) at (0, 7.75) {};
		\node [style=none] (33) at (-8.5, 8.75) {};
		\node [style=none] (35) at (-0.75, 8.75) {};
		\node [style=none] (36) at (-8.5, 6.75) {};
		\node [style=none] (37) at (-0.75, 6.75) {};
		\node [style=none] (38) at (-9.25, 4.25) {};
		\node [style=none] (39) at (-8, 4.25) {};
		\node [style=none] (40) at (-1.25, 4.25) {};
		\node [style=none] (41) at (0, 4.25) {};
		\node [style=none] (42) at (-8.5, 5.25) {};
		\node [style=none] (43) at (-0.75, 5.25) {};
		\node [style=none] (44) at (-8.5, 3.25) {};
		\node [style=none] (45) at (-0.75, 3.25) {};
		\node [style=none] (46) at (-1.25, -0.75) {};
		\node [style=none] (47) at (0, -0.75) {};
		\node [style=none] (48) at (-1.25, -4.25) {};
		\node [style=none] (49) at (0, -4.25) {};
		\node [style=none] (50) at (-0.75, -1.75) {};
		\node [style=none] (51) at (-0.75, -3.25) {};
		\node [style=none] (52) at (-0.5, 0.25) {};
		\node [style=none] (56) at (-0.5, -5.25) {};
		\node [style=none] (66) at (0, 7.75) {};
		\node [style=none] (67) at (-1.25, 7.75) {};
		\node [style=none] (68) at (0, 4.25) {};
		\node [style=none] (69) at (-1.25, 4.25) {};
		\node [style=none] (70) at (-0.5, 6.75) {};
		\node [style=none] (71) at (-0.5, 5.25) {};
		\node [style=none] (72) at (-0.75, 8.75) {};
		\node [style=none] (73) at (-0.75, 3.25) {};
		\node [style=none] (74) at (-1.25, 7.75) {};
		\node [style=none] (75) at (0, 7.75) {};
		\node [style=none] (76) at (-1.25, 4.25) {};
		\node [style=none] (77) at (0, 4.25) {};
		\node [style=none] (78) at (-0.75, 6.75) {};
		\node [style=none] (79) at (-0.75, 5.25) {};
		\node [style=none] (80) at (-0.5, 8.75) {};
		\node [style=none] (84) at (-0.5, 3.25) {};
		\node [style=none] (142) at (-2.75, 2.75) {};
		\node [style=none] (143) at (-2.75, 1) {};
		\node [style=none] (146) at (-11.75, -0.75) {};
		\node [style=none] (147) at (-10.5, -0.75) {};
		\node [style=none] (148) at (-11.25, -1.75) {};
		\node [style=none] (149) at (-11, 0.25) {};
		\node [style=none] (150) at (-11.75, -4.25) {};
		\node [style=none] (151) at (-10.5, -4.25) {};
		\node [style=none] (152) at (-11.25, -5.25) {};
		\node [style=none] (153) at (-11, -3.25) {};
		\node [style=none] (162) at (-11.75, 7.75) {};
		\node [style=none] (163) at (-10.5, 7.75) {};
		\node [style=none] (164) at (-11.25, 6.75) {};
		\node [style=none] (165) at (-11, 8.75) {};
		\node [style=none] (166) at (-11.75, 4.25) {};
		\node [style=none] (167) at (-10.5, 4.25) {};
		\node [style=none] (168) at (-11.25, 3.25) {};
		\node [style=none] (169) at (-11, 5.25) {};
		\node [style=none] (170) at (-11.75, 4.25) {};
		\node [style=none] (171) at (-10.5, 4.25) {};
		\node [style=none] (172) at (-11.25, 3.25) {};
		\node [style=none] (173) at (-11, 5.25) {};
		\node [style=none] (174) at (-11.75, -0.75) {};
		\node [style=none] (175) at (-10.5, -0.75) {};
		\node [style=none] (176) at (-11.25, -1.75) {};
		\node [style=none] (177) at (-11, 0.25) {};
		\node [style=none] (178) at (-11.75, -4.25) {};
		\node [style=none] (179) at (-10.5, -4.25) {};
		\node [style=none] (180) at (-11.25, -5.25) {};
		\node [style=none] (181) at (-11, -3.25) {};
		\node [style=none] (182) at (-11.75, -4.25) {};
		\node [style=none] (183) at (-10.5, -4.25) {};
		\node [style=none] (184) at (-11.25, -5.25) {};
		\node [style=none] (185) at (-11, -3.25) {};
	\end{pgfonlayer}
	\begin{pgfonlayer}{edgelayer}
		\draw [bend left=90, looseness=2.75] (0.center) to (1.center);
		\draw [bend right=90, looseness=2.75] (0.center) to (1.center);
		\draw [bend left=90, looseness=2.75] (2.center) to (3.center);
		\draw [bend right=90, looseness=2.75] (2.center) to (3.center);
		\draw [bend left=75, looseness=1.75] (4.center) to (5.center);
		\draw (6.center) to (7.center);
		\draw [bend left=90, looseness=2.75] (8.center) to (9.center);
		\draw (11.center) to (12.center);
		\draw [bend right=90, looseness=2.75] (13.center) to (14.center);
		\draw (8.center) to (13.center);
		\draw (9.center) to (14.center);
		\draw (3.center) to (14.center);
		\draw (1.center) to (9.center);
		\draw [bend right=90, looseness=2.75] (15.center) to (16.center);
		\draw [bend left=90, looseness=2.75] (15.center) to (16.center);
		\draw [bend right=90, looseness=2.75] (17.center) to (18.center);
		\draw [bend left=90, looseness=2.75] (17.center) to (18.center);
		\draw [bend right=75, looseness=1.75] (19.center) to (20.center);
		\draw (21.center) to (22.center);
		\draw [bend right=90, looseness=2.75] (23.center) to (24.center);
		\draw (25.center) to (26.center);
		\draw [bend left=90, looseness=2.75] (27.center) to (28.center);
		\draw (23.center) to (27.center);
		\draw (24.center) to (28.center);
		\draw (18.center) to (28.center);
		\draw (17.center) to (27.center);
		\draw (16.center) to (24.center);
		\draw (15.center) to (23.center);
		\draw [bend left=90, looseness=2.75] (29.center) to (30.center);
		\draw [bend right=90, looseness=2.75] (29.center) to (30.center);
		\draw [bend left=90, looseness=2.75] (31.center) to (32.center);
		\draw [bend right=90, looseness=2.75] (31.center) to (32.center);
		\draw (33.center) to (35.center);
		\draw (36.center) to (37.center);
		\draw [bend left=90, looseness=2.75] (38.center) to (39.center);
		\draw [bend right=90, looseness=2.75] (38.center) to (39.center);
		\draw [bend left=90, looseness=2.75] (40.center) to (41.center);
		\draw [bend right=90, looseness=2.75] (40.center) to (41.center);
		\draw (42.center) to (43.center);
		\draw (44.center) to (45.center);
		\draw (30.center) to (32.center);
		\draw (39.center) to (41.center);
		\draw [bend left=90, looseness=2.75] (46.center) to (47.center);
		\draw [bend right=90, looseness=2.75] (46.center) to (47.center);
		\draw [bend left=90, looseness=2.75] (48.center) to (49.center);
		\draw [bend right=90, looseness=2.75] (48.center) to (49.center);
		\draw [style=new edge style 2] (0.center) to (24.center);
		\draw [style=new edge style 2] (2.center) to (28.center);
		\draw [style=new edge style 2] (38.center) to (40.center);
		\draw [style=new edge style 2] (29.center) to (31.center);
		\draw [bend right=90, looseness=2.75] (66.center) to (67.center);
		\draw [bend left=90, looseness=2.75] (66.center) to (67.center);
		\draw [bend right=90, looseness=2.75] (68.center) to (69.center);
		\draw [bend left=90, looseness=2.75] (68.center) to (69.center);
		\draw [bend left=90, looseness=2.75] (74.center) to (75.center);
		\draw [bend right=90, looseness=2.75] (74.center) to (75.center);
		\draw [bend left=90, looseness=2.75] (76.center) to (77.center);
		\draw [bend right=90, looseness=2.75] (76.center) to (77.center);
		\draw [style=new edge style 0] (142.center) to (143.center);
		\draw [bend left=90, looseness=2.75] (146.center) to (147.center);
		\draw [bend right=90, looseness=2.75] (146.center) to (147.center);
		\draw [bend left=90, looseness=2.75] (150.center) to (151.center);
		\draw [bend right=90, looseness=2.75] (150.center) to (151.center);
		\draw [bend left=90, looseness=2.75] (162.center) to (163.center);
		\draw [bend right=90, looseness=2.75] (162.center) to (163.center);
		\draw [bend left=90, looseness=2.75] (166.center) to (167.center);
		\draw [bend right=90, looseness=2.75] (166.center) to (167.center);
		\draw [bend left=270, looseness=1.75] (165.center) to (164.center);
		\draw [bend left=90, looseness=2.75] (170.center) to (171.center);
		\draw [bend right=90, looseness=2.75] (170.center) to (171.center);
		\draw [bend left=270, looseness=1.75] (173.center) to (172.center);
		\draw [bend left=90, looseness=2.75] (174.center) to (175.center);
		\draw [bend right=90, looseness=2.75] (174.center) to (175.center);
		\draw [bend left=90, looseness=2.75] (178.center) to (179.center);
		\draw [bend right=90, looseness=2.75] (178.center) to (179.center);
		\draw [bend left=270, looseness=1.75] (177.center) to (176.center);
		\draw [bend left=90, looseness=2.75] (182.center) to (183.center);
		\draw [bend right=90, looseness=2.75] (182.center) to (183.center);
		\draw [bend left=270, looseness=1.75] (185.center) to (184.center);
	\end{pgfonlayer}
\end{tikzpicture} \qquad
\input{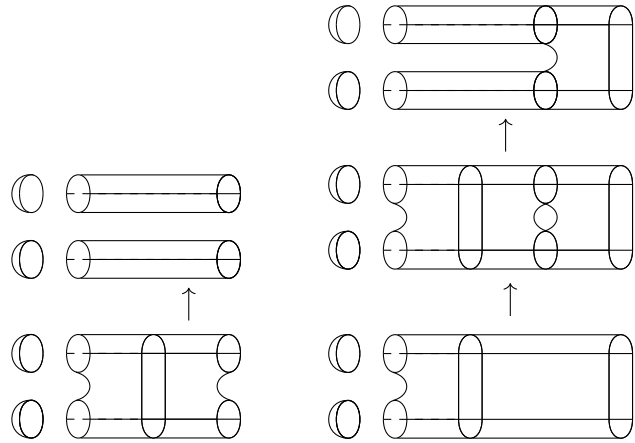}
\end{center}
\caption{The left column depicts the bordisms $B$ and $R$ in \cref{biunit} for $d=3,k=2$. The top bordism in the left column is $R$ and the bottom bordism is $B$. The caps on the left are the top square in \cref{biunit}. The bordism $\zeta$ fills in the center of the two cylinders. Further cutting the bordism $\zeta$ (in the $(d-2)$nd direction) along the horizontal lines, i.e. the composition of the two center horizontal edges in \cref{biunittoeps} gives the bordism $\epsilon$ in \cref{handle}. The right column depicts the handle cancellation, giving the bordism $\beta$ in \cref{handle} after cutting along the horizontal lines.}
\label{d3k2bords}
\end{figure}

\paragraph{Handle cancellation as a multisimplex:}
The above construction yields a bordism that can be decomposed in the $d$th and $(d-1)$st direction as follows. 
\[\delta=\begin{tikzcd}[column sep={6em,between origins}, row sep={2em,between origins}]
	\emptyset && \emptyset && \emptyset \\
	& {{\rm id}} && {{\rm id}} \\
	{S^{d-k-1}\times S^{k-1}} && {S^{d-k-1}\times S^{k-1}} && {S^{d-k-1}\times S^{k-1}} \\
	& {{\rm id}_Q} & {} \\
	{D^{d-k-1}\times S^{k-1}} && {D^{d-k-1}\times S^{k-1}} & \zeta \\
	&& {} \\
	& \xi & {S^{d-k-1}\times S^{k-1}} && {S^{d-k-1}\times S^{k-1}} \\
	&&& {{\rm id}_Q} \\
	{D^{d-k-1}\times S^{k-1}} && {D^{d-k-1}\times S^{k-1}} && {D^{d-k-1}\times S^{k-1},}
	\arrow["{D^{d-k}\times S^{k-1}}"', from=1-1, to=3-1]
	\arrow["Q"', from=3-1, to=5-1]
	\arrow["{D^{d-k}\times S^{k-1}}"{description}, from=1-3, to=3-3]
	\arrow["Q", from=3-3, to=5-3]
	\arrow["P", from=5-3, to=7-3]
	\arrow["{D^{d-k}\times S^{k-1}}"{description}, from=1-5, to=3-5]
	\arrow[from=3-1, to=3-3]
	\arrow[from=3-3, to=3-5]
	\arrow[from=7-3, to=7-5]
	\arrow[from=1-1, to=1-3]
	\arrow[from=1-3, to=1-5]
	\arrow["Q", from=7-3, to=9-3]
	\arrow[from=9-3, to=9-5]
	\arrow["Q", from=7-5, to=9-5]
	\arrow[from=9-1, to=9-3]
	\arrow[from=5-1, to=5-3]
	\arrow["{R}", from=3-5, to=7-5]
	\arrow["{S}"', from=5-1, to=9-1]
\end{tikzcd}\eqlabel{bibeta}\]
where $\zeta$ and $R$ are as in \cref{biunit}, $Q$ is the composition of the bordisms occurring in the first column of $B$, $P$ is the composition of bordisms occurring in the second column of $B$, $S$ is the composition of bordisms in the right column of \cref{lefthan}, and $\xi$ is the bordism from the right column of \cref{lefthan} to the composition of the right two columns of \cref{centerhan}. The bordism $\delta$ is the composable chain of bordisms $S\circ_{d-1}Q\to  Q\circ_{d-1}P\circ_{d-1}Q\to Q\circ_{d-1}R$. The compositions $S\circ_{d-1}Q$, $Q\circ_{d-1}P\circ_{d-1}Q$ and $Q\circ_{d-1}R$ are depicted below. 
\[
\adjustbox{scale=.75,center}{
\begin{tikzcd}[column sep={8.5em,between origins}, row sep={2.5em,between origins}]
	\emptyset && \emptyset && \emptyset \\
	& {{\rm id}} && {{\rm id}} \\
	{D^{d-k-1}\times S^{k-2}} && {D^{d-k-1}\times S^{k-2}} && {D^{d-k-1}\times S^{k-2}} \\
	& {D^{d-k}\times D^{k-1}} && {D^{d-k}\times S^{k-2}\times [0,1]} \\
	{D^{d-k-1}\times S^{k-2}} && {D^{d-k-1}\times S^{k-2}} && {D^{d-k-1}\times S^{k-2}} \\
	& {{\rm id}} && {{\rm id}} \\
	\emptyset && \emptyset && \emptyset,
	\arrow["{D^{d-k-1}\times D^{k-1}}"', from=1-1, to=3-1]
	\arrow["{S^{d-k-1}\times D^{k-1}}"', from=3-1, to=5-1]
	\arrow["{D^{d-k-1}\times D^{k-1}}"', from=5-1, to=7-1]
	\arrow[from=1-1, to=1-3]
	\arrow[from=1-3, to=1-5]
	\arrow["{D^{d-k-1}\times D^{k-1}}", from=1-5, to=3-5]
	\arrow["{D^{d-k-1}\times D^{k-1}}"{description}, from=1-3, to=3-3]
	\arrow[from=3-1, to=3-3]
	\arrow[from=3-3, to=3-5]
	\arrow["{D^{d-k}\times S^{k-2}}", from=3-5, to=5-5]
	\arrow["{D^{d-k}\times S^{k-2}}"{description}, from=3-3, to=5-3]
	\arrow[from=5-1, to=5-3]
	\arrow[from=5-3, to=5-5]
	\arrow["{D^{d-k-1}\times D^{k-1}}"{description}, from=5-3, to=7-3]
	\arrow["{D^{d-k-1}\times D^{k-1}}", from=5-5, to=7-5]
	\arrow[from=7-1, to=7-3]
	\arrow[from=7-3, to=7-5]
\end{tikzcd}}\eqlabel{lefthan}
\]
\[
\adjustbox{scale=.75,center}{
\begin{tikzcd}[column sep={8em,between origins}, row sep={2.5em,between origins}]
	\emptyset && \emptyset && \emptyset && \emptyset \\
	& {{\rm id}} && {{\rm id}} && {{\rm id}} \\
	{D^{d-k-1}\times S^{k-2}} && {D^{d-k-1}\times S^{k-2}} && {D^{d-k-1}\times S^{k-2}} && {D^{d-k-1}\times S^{k-2}} \\
	& {D^{d-k}\times D^{k-1}} && {D^{d-k}\times D^{k-1}} && {D^{d-k}\times D^{k-1}} \\
	{D^{d-k-1}\times S^{k-2}} && {D^{d-k-1}\times S^{k-2}} && {D^{d-k-1}\times S^{k-2}} && {D^{d-k-1}\times S^{k-2}} \\
	& {{\rm id}} && {{\rm id}} && {{\rm id}} \\
	\emptyset && \emptyset && \emptyset && \emptyset,
	\arrow[from=3-1, to=3-3]
	\arrow["{D^{d-k-1}\times D^{k-1}}"', from=1-1, to=3-1]
	\arrow["{D^{d-k-1}\times D^{k-1}}"{description}, from=1-3, to=3-3]
	\arrow["{D^{d-k-1}\times D^{k-1}}"{description}, from=1-5, to=3-5]
	\arrow["{D^{d-k-1}\times D^{k-1}}", from=1-7, to=3-7]
	\arrow[from=3-3, to=3-5]
	\arrow[from=3-5, to=3-7]
	\arrow[from=5-5, to=5-7]
	\arrow["{D^{d-k-1}\times D^{k-1}}", from=5-7, to=7-7]
	\arrow["{D^{d-k-1}\times D^{k-1}}"{description}, from=5-5, to=7-5]
	\arrow["{S^{d-k-1}\times D^{k-1}}"{description}, from=3-5, to=5-5]
	\arrow["{D^{d-k}\times S^{k-2}}", from=3-7, to=5-7]
	\arrow[from=5-3, to=5-5]
	\arrow["{D^{d-k-1}\times D^{k-1}}"{description}, from=5-3, to=7-3]
	\arrow["{D^{d-k}\times S^{k-2}}"{description}, from=3-3, to=5-3]
	\arrow["{S^{d-k-1}\times D^{k-1}}"', from=3-1, to=5-1]
	\arrow[from=5-1, to=5-3]
	\arrow["{D^{d-k-1}\times D^{k-1}}"', from=5-1, to=7-1]
	\arrow[from=1-1, to=1-3]
	\arrow[from=1-3, to=1-5]
	\arrow[from=1-5, to=1-7]
	\arrow[from=7-1, to=7-3]
	\arrow[from=7-3, to=7-5]
	\arrow[from=7-5, to=7-7]
\end{tikzcd}}\eqlabel{centerhan}
\]

\[
\adjustbox{scale=.75,center}{
\begin{tikzcd}[column sep={8.5em,between origins}, row sep={2.5em,between origins}]
	\emptyset && \emptyset && \emptyset \\
	& {{\rm id}} && {{\rm id}} \\
	{D^{d-k-1}\times S^{k-2}} && {D^{d-k-1}\times S^{k-2}} && {D^{d-k-1}\times S^{k-2}} \\
	& {S^{d-k-1}\times D^{k-1}\times [0,1]} && {D^{d-k}\times D^{k-1}} \\
	{D^{d-k-1}\times S^{k-2}} && {D^{d-k-1}\times S^{k-2}} && {D^{d-k-1}\times S^{k-2}} \\
	& {{\rm id}} && {{\rm id}} \\
	\emptyset && \emptyset && \emptyset.
	\arrow["{D^{d-k-1}\times D^{k-1}}"', from=1-1, to=3-1]
	\arrow[from=1-1, to=1-3]
	\arrow["{D^{d-k-1}\times D^{k-1}}"{description}, from=1-3, to=3-3]
	\arrow[from=3-1, to=3-3]
	\arrow[from=1-3, to=1-5]
	\arrow["{D^{d-k-1}\times D^{k-1}}", from=1-5, to=3-5]
	\arrow[from=3-3, to=3-5]
	\arrow["{D^{d-k}\times S^{k-2}}", from=3-5, to=5-5]
	\arrow["{S^{d-k-1}\times D^{k-1}}"{description}, from=3-3, to=5-3]
	\arrow[from=5-3, to=5-5]
	\arrow["{D^{d-k-1}\times D^{k-1}}"{description}, from=5-3, to=7-3]
	\arrow[from=7-3, to=7-5]
	\arrow["{D^{d-k-1}\times D^{k-1}}", from=5-5, to=7-5]
	\arrow[from=5-1, to=5-3]
	\arrow["{S^{d-k-1}\times D^{k-1}}"{description}, from=3-1, to=5-1]
	\arrow["{D^{d-k-1}\times D^{k-1}}"', from=5-1, to=7-1]
	\arrow[from=7-1, to=7-3]
\end{tikzcd}}
\]

\end{construction}

Throughout the remainder of the section, we will frequently refer to the various bordisms constructed above. The following table provides a summary of the bordisms we will use, along with an informal description. For the formal description, see \cref{handle} and \cref{handletails}.

\begin{remark}
In order to ensure that certain bordisms can be glued, we will use the geometric structure given by taking fiberwise etale maps to $\RR^d$.
This is the fibrant replacement of the geometric structure given by open embeddings.
Since the bordism category satisfies codescent (see Grady--Pavlov \cite[\ecref{EL-mainthm}]{GradyPavlov}), the bordism category with fiberwise open embeddings as the geometric structure is canonically equivalent (via sheafification of the geometric structure) to the bordism category with etale maps as the geometric structure. 
Therefore, the main theorem holds for both geometric structures. 
\end{remark}

\begin{definition}\label{handle}
Given $d≥0$ and $0≤k≤d$, we define a subobject $H_k\subset \BBord_d^{\RR^d\times U\to U}$ 
as follows. 
Fix $V\in \cartsp$, $\langle \ell\rangle\in \Gamma$ and ${\bf m}\in \Delta^{\times (d-2)}$.
For $k>0$, we define $H_k(V,\langle \ell\rangle,{\bf m})$
to be the subobject that is generated from the following bisimplices.

For $k=0$, we only take the bordism given by the second diagram, setting $k=1$ in that diagram, so that we get the index~0 handle as a bordism $D^d:\emptyset\to S^{d-1}$.
We require that the whole bordism maps to a single element of $\langle\ell\rangle$.
We omit $V$ from the notation below and understand the bordisms to be the corresponding trivial family of bordisms parametrized over~$V$.
We also omit the simplex~$l$ (responsible for smooth deformations of cut tuples as in \cref{enrichedbordstr}) from the notation below,
taking for every indicated bordism the whole connected component in the corresponding simplicial set.
\begin{enumerate}
\item[\fbox{$\epsilon$}] The handle of index~$k$ (constructed above),
interpreted as a counit and given explicitly as a bordism in bidegree $([1],[1])$ given by \cref{milnor.construction}:
\begin{equation}\label{counitforreal}
\xymatrix@C=.2cm{
S^{d-k-1}\times D^{k-1} \ar[rr]^{\id}\ar[dd]_{D^{d-k}\times S^{k-1}} && S^{d-k-1}\times D^{k-1}\ar[dd]^{S^{d-k-1}\times D^{k-1}\times[0,1]}\\
&D^{d-k}\times D^{k} &\\
S^{d-k-1}\times D^{k-1} \ar[rr] _{\id} && S^{d-k-1}\times D^{k-1}\\
}
\end{equation}
together with all bordisms
occurring in the following diagram, provided that its composition in the 1st simplicial direction (i.e., the vertical direction) can be extended to a diagram of the above form:
$$
\xymatrix@C=2cm{
S^{d-k-1}\times D^{k-1} \ar[r]^{\id}\ar[d]^{\fbox{$g$}}_{D^{d-k}\times D^{k-1}} & S^{d-k-1}\times D^{k-1}\ar[d]^{S^{d-k-1}\times D^{k-1}\times[0,1]}\\
D^{d-k}\times S^{k-2} \ar[d]^{\fbox{$f$}}_{D^{d-k}\times D^{k-1}} &  S^{d-k-1}\times D^{k-1}\ar[d]^{S^{d-k-1}\times D^{k-1}\times[0,1]}\\
S^{d-k-1}\times D^{k-1} \ar[r] _{\id}& S^{d-k-1}\times D^{k-1}\\
}
$$
\item[\fbox{$\eta$}] The handle of index~$k-1$ (constructed above), interpreted as a unit and given explicitly as a bordism in bidegree $([1],[1])$ given by \cref{milnor.construction}:
\begin{equation}\label{unitforreal}
\xymatrix@C=.2cm{
\ar[rr]^{\id}  D^{d-k}\times S^{k-2}\ar[dd]_{D^{d-k}\times S^{k-2}\times[0,1]} &&  D^{d-k}\times S^{k-2}\ar[dd]^{ S^{d-k}\times D^{k-1}}\\
&D^{d-k+1}\times D^{k-1} &\\
\ar[rr] ^{\id}  D^{d-k}\times S^{k-2} && D^{d-k}\times S^{k-2}}
\end{equation}
together with all bordisms
occurring in the following diagram, provided that its composition in the 1st simplicial direction (i.e., the vertical direction) can be extended to a diagram of the above form:
$$\xymatrix@C=2cm{
\ar[r]^{\id}  D^{d-k}\times S^{k-2}\ar[d]_{D^{d-k}\times S^{k-2}\times[0,1]} & D^{d-k}\times S^{k-2} \ar[d]_{\fbox{$f$}}^{D^{d-k}\times D^{k-1}}\\
D^{d-k}\times S^{k-2}\ar[d]_{D^{d-k}\times S^{k-2}\times[0,1]} & S^{d-k-1}\times D^{k-1} \ar[d]_{\fbox{$g$}}^{D^{d-k}\times D^{k-1}}\\
\ar[r] ^{\id}  D^{d-k}\times S^{k-2} & D^{d-k}\times S^{k-2}\\
}$$
\item[\fbox{$\beta$}] Consider the bordisms in the following diagram
given by \cref{milnor.construction}:
\begin{equation}\label{triangleid}
\xymatrix@C=2cm{
S^{d-k-1}\times D^{k-1}\ar[r]^\id \ar[d]_{D^{d-k}\times D^{k-1}}^g & S^{d-k-1}\times D^{k-1} \ar[r]^{\id}\ar[d]_{D^{d-k}\times D^{k-1}}^g & S^{d-k-1}\times D^{k-1}\ar[d]^{S^{d-k-1}\times D^{k-1}\times[0,1]}\\
\ar[r]^{\id}  D^{d-k}\times S^{k-2}\ar[d]_{D^{d-k}\times S^{k-2}\times[0,1]} & D^{d-k}\times S^{k-2} \ar[d]^{D^{d-k}\times D^{k-1}}_f \ar@{{}{ }{}}[r]|{\displaystyle \epsilon} & S^{d-k-1}\times D^{k-1}\ar[d]^{S^{d-k-1}\times D^{k-1}\times[0,1]}\\
D^{d-k}\times S^{k-2}\ar[d]_{D^{d-k}\times S^{k-2}\times[0,1]} \ar@{{}{ }{}}[r]|{\displaystyle η}  & S^{d-k-1}\times D^{k-1} \ar[d]^{D^{d-k}\times D^{k-1}}_g \ar[r] _{\id} &  S^{d-k-1}\times D^{k-1}\ar[d]^{D^{d-k}\times D^{k-1}}_g\\
\ar[r] ^{\id}  D^{d-k}\times S^{k-2} & D^{d-k}\times S^{k-2}\ar[r]_{\id} & D^{d-k}\times S^{k-2}.\\
}
\end{equation}
We include bordisms given by the $([2],[1])$-bisimplices appearing in the left and right rectangles above.
We also include bordisms given by the $([1],[1])$-bisimplices given by their compositions (in the vertical direction, i.e., the $(d-1)$st direction).
Finally we include bordisms given by the $([1],[2])$-bisimplex given by the left and right bordisms above and the $([1],[1])$-bisimplex given by their composition. 
\end{enumerate}
The fiberwise embedding into $\RR^d\times U$ of each vertex and each edge in the above diagrams are induced by restricting the fiberwise embedding of the corresponding 2-cell.

Summarizing, $H_k$ contains bordisms diffeomorphic to $f$,~$g$,~$\epsilon$,~$\eta$,~$\eta\circ_{d-1} \id_g$,~$\id_g\circ_{d-1} \epsilon$, and (the domain of) the triangle identity $\beta:=(\id_g\circ_{d-1} \epsilon)\circ_d (\eta\circ_{d-1} \id_g)$.
Here $\circ_i$ denotes the composition in the $i$th direction.
\end{definition}

\begin{definition}
\label{def.O}
Given $d≥0$ and $0<k≤d$,
let $O_{k-1}\subset H_k$ denote the subobject of $H_k$ (\cref{handle}) defined as follows.
Fix $V\in \cartsp$, $\langle \ell\rangle\in \Gamma$ and ${\bf m}\in \Delta^{\times (d-2)}$.
For $k>0$, we define
$$O_{k-1}(V,\langle \ell\rangle,{\bf m})\subset H_k(V,\langle \ell\rangle,{\bf m})$$
to be the bisimplicial subspace generated by bisimplices present in the left column of $\beta$ in \cref{handle}.
(We remark that this definition of $O_{k-1}$ automatically includes bisimplices of degree $([1],[0])$ in $\epsilon$ of \cref{handle}.)
For $k=0$, we define
$$O_{-1}(V,\langle \ell\rangle,{\bf m})\subset H_0(V,\langle \ell\rangle,{\bf m})$$
to be the bisimplicial subspace generated by bisimplices of degree $([1],[0])$ (i.e., vertical bordisms) from $\eta$ of \cref{handle}.
\end{definition}

The subobject $H_k$ contains precisely the data necessary to witness an adjunction in $B_k$. To see this, let ${\bf 1}_{d-2}\in \Delta^{\times d-2}$ be the multisimplex with components ${\bf 1}_i=[1]$ for all $1\leq i\leq d-2$. We recall that ${\rm Adj}_{{\bf 1}_{d-2}}$ satisfies a universal property: given an $(\infty,d)$-category $\mathscr{C}$, and an adjunction in the homotopy 2-category of $\mathscr{C}$, there exists a full homotopy coherent adjunction
$${\rm Adj}_{{\bf 1}_{d-2}}\to \mathscr{C}$$
that sends $f,g,\epsilon$, $\eta$  to the corresponding left adjoint, right adjoint, counit and unit of the adjunction. Moreover, the triangle identities and higher coherences for the triangle identities are sent to the corresponding data in $\mathscr{C}$. Applying the universal property to the bordism category, we obtain the following.

\begin{proposition}\label{adjinbord}
Fix $d\geq 0$. Let ${\bf 1}_{d-2}\in \Delta^{\times d-2}$ be the multisimplex with components ${\bf 1}_i=[1]$ for all $1\leq i\leq d-2$. Let $R$ denote the fibrant replacement functor in $\smcat$. The bordisms labeled by $f,g, \epsilon$ and $\eta$ in Part~3 of \cref{handle} are part of a homotopy coherent adjunction in the fibrant replacement of $\BBord_d^{\RR^d\times U\to U}$. That is, there is a map 
$${\rm Adj}_{{\bf 1}_{d-2}}\to R(\BBord_d^{\RR^d\times U\to U}) \eqlabel{cohadjinbord}$$
that sends $f,g,\epsilon,\eta$ and the triangle identity $\beta$, given by the bisimplex of degree $([1],[2])$
$$
\beta:=\vcenter{\xymatrix@C=2cm{
+\ar[r]^-{\id} \ar[d]_g & + \ar[r]^{\id}\ar[d]_g & +\ar[d]^{\id}
\\
 -\ar[d]_{\id}\ar[r]_{\id} & - \ar[d]^{f} \ar@{{}{ }{}}[r]|{\displaystyle \epsilon} & +\ar[d]^{\id}
 \\
-\ar[d]_{\id} \ar@{{}{ }{}}[r]|{\displaystyle η}  & + \ar[d]^{g}\ar[r]^{\id}  &  +\ar[d]^{g}\\
\ar[r] ^{\id} - & -\ar[r]_{\id} & -.\\
}}
$$
to the corresponding bordisms depicted in \cref{handle}.
Observe that in the Reedy fibrant replacement, the homotopy that contracts the cylinder given by composing the diagram \cref{counitforreal}
can be glued together with the $([1],[2])$ multisimplex that composes this diagram in the horizontal direction,
resulting in a $([1],[2])$-multisimplex whose two outer faces are given by the two columns
and the inner face is degenerate.
Hence, it makes sense to map identities to such simplices.
\end{proposition}
\proof
First, observe that $\BBord_d^{\RR^d\times U\to U}$ is closed under composition of the bordisms in Part~3 of \cref{handle}.
It is also closed under composition of the second triangle identity diagram, obtained by shuffling the squares in Part~3. 

By the universal property of ${\rm Adj}_{{\bf 1}_{d-1}}$, it suffices to show that the bordisms in \cref{handle} define an adjunction in the homotopy 2-category of the bordism category, after evaluating at ${\bf 1}_{d-2}\in \Delta^{\times d-2}$. To see that the triangle identities are satisfied in the homotopy 2-category, observe that the bordism in Part~3 is obtained by the family of generalized Morse functions $f_{k,k-1}:\RR^d\to \RR$ in \cref{milnor.construction}. The member of the family at $t=1$ is a Morse function having no critical points, which implies that the bordism is a cylinder. Since cylinders are homotopic to their source cut by a homotopy of rank $d$, i.e., by the Morse flow, the bordism is an identity in the homotopy 2-category. Finally, since all these bordisms and their honest compositions are also present in the fibrant replacement, we indeed have an adjunction. 
\endproof

We now define subobjects of the bordism category that will be used in the proof of the main theorem. These subobjects will serve as an intermediate step between bordisms in the $k$th level of the index filtration in \cref{index.filtration} and bordisms in $H_k$. The next definition includes the bordisms obtained from \cref{milnor.construction}, before cutting in the $(d-2)$nd direction.

\begin{definition}\label{handletails}
Given $d≥0$ and $0≤k≤d$, we define a subobject $\Hb_k\subset \BBord_d^{\RR^d\times U\to U}$ 
as follows. 
Fix $V\in \cartsp$, $\langle \ell\rangle\in \Gamma$ and ${\bf m}\in \Delta^{\times (d-2)}$.
For $k>0$, we define $\Hb_k(V,\langle \ell\rangle,{\bf m})$
to be the subobject that is generated from the following bisimplices.

For $k=0$, we only take the bordism given by the second diagram, setting $k=1$ in that diagram, so that we get the index~0 handle as a bordism $D^d:\emptyset\to S^{d-1}$.
We require that the whole bordism maps to a single element of $\langle\ell\rangle$.
We omit $V$ from the notation below and understand the bordisms to be the corresponding trivial family of bordisms parametrized over~$V$.
We also omit the simplex~$l$ (responsible for smooth deformations of cut tuples as in \cref{enrichedbordstr}) from the notation below,
taking for every indicated bordism the whole connected component in the corresponding simplicial set.
\begin{enumerate}
\item[\fbox{$\zeta$}] The handle with tails of index~$k$, given explicitly as a bordism in bidegree $([1],[1])$ given by \cref{biunit} in  \cref{milnor.construction}:
\begin{equation}\label{bicounitforreal}
\xymatrix@C=.4cm{
S^{d-k-1}\times S^{k-1} \ar[rr]^{\id}\ar[dd]_{B} && S^{d-k-1}\times S^{k-1}\ar[dd]^{R}\\
&\zeta &\\
S^{d-k-1}\times S^{k-1} \ar[rr] _{\id} && S^{d-k-1}\times S^{k-1}\\
}
\end{equation}
together with all bordisms
occurring in the following diagram, provided that its composition in the 1st simplicial direction (i.e., the vertical direction) can be extended to a diagram of the above form:
$$
\xymatrix@C=2cm{
S^{d-k-1}\times S^{k-1} \ar[r]^{\id}\ar[d]_{\fbox{$Q$}}& S^{d-k-1}\times S^{k-1}\ar[d]^{S^{d-k-1}\times S^{k-1}\times[0,1]}\\
D^{d-k-1}\times S^{k-1} \ar[d]_{\fbox{$P$}} &  S^{d-k-1}\times S^{k-1}\ar[d]^{S^{d-k-1}\times S^{k-1}\times[0,1]}\\
S^{d-k-1}\times S^{k-1} \ar[r] _{\id}& S^{d-k-1}\times S^{k-1}\\
}
$$
\item[\fbox{$\xi$}] The handle with tails of index~$k-1$, explicitly as a bordism in bidegree $([1],[1])$ given by the bordism $\xi$ in \cref{bibeta} of \cref{milnor.construction}:
\begin{equation}\label{biunitforreal}
\xymatrix@C=.4cm{
D^{d-k-1}\times S^{k-1} \ar[rr]^{\id}\ar[dd]_{S} && D^{d-k-1}\times S^{k-1}\ar[dd]^{B'}\\
&\xi&\\
D^{d-k-1}\times S^{k-1} \ar[rr] _{\id} && D^{d-k-1}\times S^{k-1}\\
}
\end{equation}
together with all bordisms
occurring in the following diagram, provided that its composition in the 1st simplicial direction (i.e., the vertical direction) can be extended to a diagram of the above form:
$$\xymatrix@C=2cm{
\ar[r]^{\id} D^{d-k-1}\times S^{k-1}\ar[d]_{D^{d-k-1}\times S^{k-1}\times[0,1]} & D^{d-k-1}\times S^{k-1} \ar[d]^{\fbox{$P$}}\\
D^{d-k-1}\times S^{k-1}\ar[d]_{D^{d-k-1}\times S^{k-1}\times[0,1]} & S^{d-k-1}\times S^{k-1} \ar[d]^{\fbox{$Q$}}\\
\ar[r] ^{\id}  D^{d-k-1}\times S^{k-1} & D^{d-k-1}\times S^{k-1}\\
}$$
\item[\fbox{$\delta$}] Consider the handle cancellation with tails 
given by \cref{bibeta} in \cref{milnor.construction}:
\begin{equation}\label{trianglebeta}
\begin{tikzcd}[column sep={6em,between origins}, row sep={2em,between origins}]
	{S^{d-k-1}\times S^{k-1}} && {S^{d-k-1}\times S^{k-1}} && {S^{d-k-1}\times S^{k-1}} \\
	& {{\rm id}_Q} & {} \\
	{D^{d-k-1}\times S^{k-1}} && {D^{d-k-1}\times S^{k-1}} & \zeta \\
	&& {} \\
	& \xi & {S^{d-k-1}\times S^{k-1}} && {S^{d-k-1}\times S^{k-1}} \\
	&&& {{\rm id}_Q} \\
	{D^{d-k-1}\times S^{k-1}} && {D^{d-k-1}\times S^{k-1}} && {D^{d-k-1}\times S^{k-1}}
	\arrow["Q"', from=1-1, to=3-1]
	\arrow["Q", from=1-3, to=3-3]
	\arrow["P", from=3-3, to=5-3]
	\arrow[from=1-1, to=1-3]
	\arrow[from=1-3, to=1-5]
	\arrow[from=5-3, to=5-5]
	\arrow["Q", from=5-3, to=7-3]
	\arrow[from=7-3, to=7-5]
	\arrow["Q", from=5-5, to=7-5]
	\arrow[from=7-1, to=7-3]
	\arrow[from=3-1, to=3-3]
	\arrow["R", from=1-5, to=5-5]
	\arrow["S"', from=3-1, to=7-1]
\end{tikzcd}
\end{equation}
We include bordisms given by the $([2],[1])$-bisimplices appearing in the left and right rectangles above.
We also include bordisms given by the $([1],[1])$-bisimplices given by their compositions (in the vertical direction, i.e., the $(d-1)$st direction).
Finally we include bordisms given by the $([1],[2])$-bisimplex given by the left and right bordisms above and the $([1],[1])$-bisimplex given by their composition. 
\end{enumerate}
The fiberwise embedding into $\RR^d\times U$ of each vertex and each edge in the above diagrams are induced by restricting the fiberwise embedding of the corresponding 2-cell.

Summarizing, $\Hb_k$ contains bordisms diffeomorphic to $P$,~$Q$,~$\zeta$,~$\xi$,~$\xi\circ_{d-1} \id_Q$,~$\id_Q\circ_{d-1} \zeta$, and (the domain of) the triangle identity (with tails) $\delta:=(\id_Q\circ_{d-1} \zeta)\circ_d (\xi\circ_{d-1} \id_Q)$.
Here $\circ_i$ denotes the composition in the $i$th direction.
\end{definition}

We also define a further subobject $\Ob_{k-1}$ of $\Hb_k$, similar to the subobject $O_{k-1}$ of $H_k$. 

\begin{definition}
\label{def.Obar}
Given $d≥0$ and $0<k≤d$,
let $\Ob_{k-1}\subset \Hb_k$ denote the subobject of $\Hb_k$ (\cref{handletails}) defined as follows.
Fix $V\in \cartsp$, $\langle \ell\rangle\in \Gamma$ and ${\bf m}\in \Delta^{\times (d-2)}$.
For $k>0$, we define
$$\Ob_{k-1}(V,\langle \ell\rangle,{\bf m})\subset \Hb_k(V,\langle \ell\rangle,{\bf m})$$
to be the bisimplicial subspace generated by bisimplices present in the left column of $\delta$ in \cref{handle}.
(We remark that this definition of $\Ob_{k-1}$ automatically includes bisimplices of degree $([1],[0])$ in $\zeta$ of \cref{handle}.)
For $k=0$, we define
$$\Ob_{-1}(V,\langle \ell\rangle,{\bf m})\subset \Hb_0(V,\langle \ell\rangle,{\bf m})$$
to be the bisimplicial subspace generated by bisimplices of degree $([1],[0])$ (i.e., vertical bordisms) from $\xi$ of \cref{handle}.
\end{definition}

\begin{definition}
\label{def.Ht}
Given $d\geq 0$ and $0\leq k\leq d$, let $\Ht_k\subset \BBord_d^{\RR^d\times U\to U}$ denote the subobject defined as follows. 
Fix $V\in \cartsp$, $\langle \ell\rangle\in \Gamma$ and ${\bf m}\in \Delta^{\times (d-2)}$.
We define 
$$\Ht_{k}(V,\langle \ell\rangle,{\bf m})\subset \BBord_d^{\RR^d\times U\to U}(V,\langle \ell\rangle, {\bf m})$$ 
to be the simplicial subset generated by the following simplices. With respect to the cartesian space direction, the corresponding family of bordisms is trivial.
\begin{enumerate}
\item[\fbox{$\widetilde{\zeta}$}] We take 1-simplices such that the closure of the interior of the core is diffeomorphic to the Milnor bridge \cref{biunitwcaps}, constructed in \cref{milnor.construction}. 
\item[\fbox{$\widetilde{\xi}$}]  We take 1-simplices such that the closure of the interior of the core is diffeomorphic to the Milnor bridge obtained by replacing $k$ by $k-1$ in \cref{biunitwcaps}. 
\item[\fbox{$\widetilde{\delta}$}] We take 2-simplices such that the closure of the interior of the core is diffeomorphic to the bordism obtained by composing the two vertical columns in \cref{bibeta}.
\end{enumerate}
The structure maps for $\Ht_k$ are given by restricting the corresponding structure maps for $\BBord_d^{\RR^d\times U\to U}$. Similarly, we define $\widetilde{O}_{k-1}$ to be the subobject 
$$\Ot_{k-1}(V,\langle \ell\rangle,{\bf m})\subset \BBord_d^{\RR^d\times U\to U}(V,\langle \ell\rangle,{\bf m})$$ 
generated by the 1-simplices 
\begin{enumerate}
\item[\fbox{$\widetilde{\xi}$}] We take 1-simplices such that the closure of the interior of the core is diffeomorphic to the Milnor bridge obtained by replacing $k$ by $k-1$ in \cref{biunitwcaps}. 
\end{enumerate}
\end{definition}

Our next goal is to obtain handles with tails, i.e. bordisms in $\widetilde{H}_k$ from handles via a homotopy pushout diagram. Essentially, the proof amounts to ``cutting off tails'' from the Milnor bridge in \cref{milnor.construction} and then cutting out the counit \cref{counitforreal} from the bordism \cref{biunit}, by cutting off codimension 2 tails (see \cref{handlecutfig}). 

\begin{figure}[ht]
$$
\input{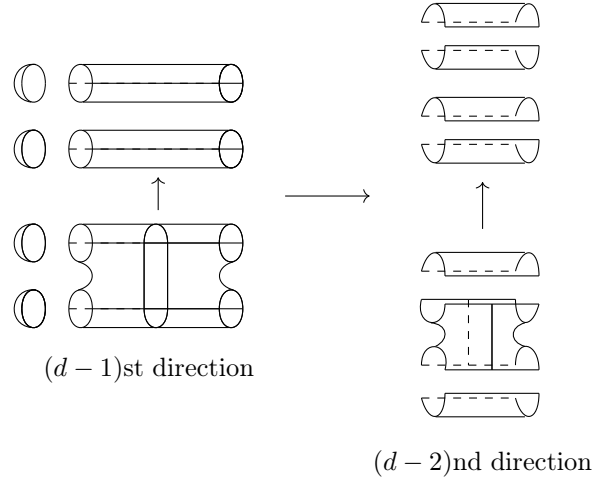}
$$
\caption{The proof of \cref{handlecutout} proceeds by first cutting out bordisms of the form $\zeta$ in \cref{biunit}, by cutting off caps in the $(d-1)$st direction, and then cutting out bordisms of the form $\epsilon$, by cutting in the $(d-2)$nd direction as indicated. The figure described the process for the special case $d=3$,$k=2$.} \label{handlecutfig}
\end{figure}

\begin{proposition}\label{handlecutout}
For any $d≥0$ and $0≤k≤d$,
the following diagrams are homotopy pushout diagrams in $\smcatdual_{\infty,d}$:
$$
\xymatrix{
O_{k-1}\ar[r]\ar[d] & \Ob_{k-1}\ar[d]\ar[r] & \Ot_{k-1}\ar[d] \\
H_{k}\ar[r] & \Hb_{k}\ar[r] & \Ht_k.
}$$
The objects in the diagrams are as in \cref{handle} ($H_k$), \cref{def.O} ($O_{k-1}$), 
\cref{handletails} ($\Hb_k$), \cref{def.Obar} $\Ob_{k-1}$), \cref{def.Ht} ($\Ht_k$, $\Ot_{k-1}$). 
\end{proposition}

\begin{proof}
The right pushout square requires cutting in the $(d-1)$st direction, while the left pushout square requires cutting in the $(d-2)$nd direction. We begin with the right pushout square. 

\paragraph{The right pushout square:}

We will show that after passing to stalks on $\cartsp$, evaluating on an arbitrary object $\langle \ell\rangle\in \Gamma$ and multisimplex ${\bf m}\in \Delta^{S}$, with $S=\{1,\ldots, d-2,d\}$, the canonical map from the objectwise pushout $f:\Ot_{k-1}\sqcup_{\Ob_{k-1}}\Hb_k\to \Ht_k$ is a weak equivalence in the Rezk model structure on simplicial spaces.
This is sufficient to prove that $f$ is a weak equivalence by Grady--Pavlov \cite[\ecref{EL-multiple.single}]{GradyPavlov}.
Furthermore, by Grady--Pavlov \cite[\ecref{EL-levelwise}]{GradyPavlov}, it suffices to show that $f$ induces an equivalence in the Joyal model structure on simplicial sets at each simplicial level $[l]\in \Delta$ indexing the target simplicial direction (i.e., the ``space'' direction).

Observe that $f$ is bijective on vertices, since vertices are bordisms with a single cut in the $(d-1)$st simplicial direction, which are present in both the pushout and in $\Ht_k$ by definition.
By Grady--Pavlov \cite[\ecref{EL-necklaces.prop}]{GradyPavlov}, it suffices to show that for all vertices $x,y\in \widetilde H_k$, the induced map on Dugger--Spivak necklace categories   
$$\Cnec(f):\Cnec(\Ot_{k-1}\sqcup_{\Ob_{k-1}}\Hb_k)(x,y)\to \Cnec(\Ht_k)(x,y)$$
is a weak equivalence in the classical model structure on simplicial sets.

Explicitly, the Dugger–Spivak necklace categories associated to the pushout can be described as follows.

\begin{itemize}
\setlength\itemsep{0em}
\item[$\circ$] A necklace of length $t$ in the category $\Cnec(\Ot_{k-1}\sqcup_{\Ob_{k-1}}\Hb_k)(x,y)$ is a composable tuple of embedded bordisms $(M_s,C_s)_{s=1}^t$ called \emph{beads}. The tuple satisfies the following property.
\begin{enumerate}
\setlength\itemsep{0em}
\item[(P)] Each bead $(M_s,C_{s})$ is either in $\Ot_{k-1}$ or $\Hb_{k}$. 
\end{enumerate}

\item[$\circ$] A morphism in  $\Cnec(\Ot_{k-1}\sqcup_{\Ob_{k-1}}\Hb_k)(x,y)$ is a composition of two basic types of morphisms. 
\begin{itemize}
\setlength\itemsep{0em}
\item The first type inserts a cut $\Upsilon$ (Grady--Pavlov \cite[\ecref{cut}]{GradyPavlov}) in a bead, obtaining a new bead $(M_s,C_s\cup\Upsilon)$,
provided that Property~(P) continues to hold. 

\item The second type glues consecutive beads together, reducing the number of beads by one,
provided that Property~(P) continues to hold. 
\end{itemize}
\end{itemize}
The necklace categories of $\Ht_{k}$ are identical, except Property P above is
replaced with the condition that each bead belongs to~$\Ht_k$.
The map $\Cnec(f)$ simply regards a necklace in $\Ot_{k-1}\sqcup_{\Ob_{k-1}}\Hb_k$ as a necklace in $\Ot_k$ (by forgetting Property P). 

The map $\Cnec(f)$ splits as a coproduct, indexed by the germs of cores of ambient bordisms with endcuts $x$ and~$y$.
(In other words, morphisms in the necklace category can only exist between necklaces whose underlying bordisms from~$x$ to~$y$ coincide.)
We claim that each disjoint summand of ${\cal N}(\Cnec(\Ht_{k})(x,y))$ and ${\cal N}(\Cnec(\Ot_{k-1}\sqcup_{\Ob_{k-1}}\Hb_k)(x,y))$ corresponding to the ambient bordism is contractible.
To this end, we use
the simplicial Whitehead theorem (see, for example, Grady–Pavlov \cite[\ecref{EL-simplicial.whitehead}]{GradyPavlov}). We prove the claim for the pushout, the proof for $\Ht_k$ is identical.

Choose an arbitrary map from the subdivided sphere
$$f:{\rm Sd}^k\partial \Delta[n]\to {\cal N}(\Cnec(\Ot_{k-1}\sqcup_{\Ob_{k-1}}\Hb_k)(x,y)).$$
Since the closure of the interior of the core of the each bordism of the form $\widetilde{\zeta}$ (\cref{def.Ht}) picked out by the map~$f$ is compact and there are only finitely many cut tuples picked out by the subdivided sphere,
there is a minimum distance, say $\epsilon$, from the closure of the interior of the core of $\widetilde{\zeta}$ to the nearest cut. This continues to hold for $l>0$, since the smooth $l$-simplex is compact, so we can choose a single $\epsilon$ for all points $t∈Δ^l$. The center horizontal cut $z=S^{d-k-1}\times S^{k-1}$ and bottom horizontal cut $w=S^{d-k-1}\times S^{k-1}$ in \cref{biunit} can be inserted at a distance of say $\epsilon/2$ from the closure of the interior of the core of $\widetilde{\zeta}$. Hence, the cuts $z$ and $w$ can be inserted simultaneously in each $\widetilde{\zeta}$ in each necklace picked out by the subdivided sphere.

Let $i$ index the bordisms of the form $\widetilde{\zeta}$ occurring in a given bordism and let $z_i,w_i$ denote the cuts constructed above, corresponding to the $i$th $\widetilde{\zeta}$. 
We will produce a zig-zag of simplicial homotopies from $f$ to a constant map,
which inserts the cut $z_{i}$ and $w_i$ defined above. The zig-zag of simplicial homotopies consists of three simplicial homotopies.
The first homotopy inserts the cut $z_{i}$ and $w_i$ for each vertex in the subdivided sphere simultaneously.
The second converts all bead cuts to ordinary cuts, simultaneously for all vertices in the subdivided sphere.
The last homotopy removes all cuts but $z_{i}$, $w_i$, $x$, and~$y$. 	

To complete the proof, we need only show that the map $\Cnec(f)$ is bijective on the disjoint summands indexed by ambient bordisms.
Since $\Cnec(f)$ is  just an inclusion on such disjoint summands, by definition, we need only prove surjectivity.
But applying the above construction to a single necklace in $\Cnec(\widetilde{H}_k)$ produces a zig-zag connecting this necklace to a necklace in the image of $\Cnec(f)$, which proves surjectivity.

\paragraph{The left pushout square:}

The proof is similar to the proof for the right square, except that we cut in the $(d-2)$nd direction. 
Again we pass to stalks on $\cartsp$ and evaluate on an arbitrary object $\langle \ell\rangle\in \Gamma$. We also evaluate on an arbitrary multisimplex ${\bf m}\in \Delta^{S}$, with $S=\{1,\ldots, d-3,d-1,d\}$, isolating the $(d-2)$nd direction. We again claim that $f$ induces an equivalence in the Joyal model structure on simplicial sets at each simplicial level $[l]\in \Delta$ indexing the target simplicial direction (i.e., the ``space'' direction).

Observe that $f$ is again bijective on vertices, since vertices are bordisms with a single cut in the $(d-2)$st simplicial direction, which are present in both the pushout and in $\Hb_k$ by definition. The Dugger–Spivak necklace categories associated to the pushout and $\Hb_k$ are described precisely as in the proof of the right pushout square, but with $\Ob_{k-1}$ replacing $\Ot_{k-1}$, $\Hb_k$ replacing $\Ht_k$ and $O_{k-1}$ replacing $\Ob_{k-1}$. 

The map $\Cnec(f)$ again splits as a coproduct, indexed by the germs of cores of ambient bordisms with endcuts $x$ and~$y$. We claim that each disjoint summand of ${\cal N}(\Cnec(\Ht_{k})(x,y))$ and ${\cal N}(\Cnec(\Ob_{k-1}\sqcup_{O_{k-1}}H_k)(x,y))$ corresponding to the ambient bordism is contractible and again invoke the simplicial Whitehead theorem (see, for example, Grady–Pavlov \cite[\ecref{EL-simplicial.whitehead}]{GradyPavlov}). We prove the claim for the pushout, the proof for $\Hb_k$ is identical.

Choose an arbitrary map from the subdivided sphere
$$f:{\rm Sd}^k\partial \Delta[n]\to {\cal N}(\Cnec(\Ot_{k-1}\sqcup_{\Ob_{k-1}}H_k)(x,y)).$$
Since the closure of the interior of the core of the each bordism of the form $\zeta$ (\cref{handletails}), picked out by the map~$f$ is compact and there are only finitely many cut tuples picked out by the subdivided sphere,
there is a minimum distance, say $\epsilon$, from the closure of the interior of the core of $\zeta$ to the nearest cut. This continues to hold for $l>0$, since the smooth $l$-simplex is compact, so we can choose a single $\epsilon$ for all points $t∈Δ^l$. The two center horizontal cuts $z=D^{d-k-1}\times S^{k-2}$ and $w=D^{d-k-1}\times S^{k-2}$ in \cref{biunittoeps} can be inserted at a distance of say $\epsilon/2$ from the closure of the interior of the core of each $\zeta$. Hence, the cuts $z$ and $w$ can be inserted simultaneously in each $\zeta$ in each necklace picked out by the subdivided sphere.

Let $i$ index the bordisms of the form $\zeta$ occurring in a given bordism and let $z_i,w_i$ denote the cuts constructed above, corresponding to the $i$th $\zeta$. 
We will produce a zig-zag of simplicial homotopies from $f$ to a constant map,
which inserts the cuts $z_{i}$ and $w_i$ defined above. The zig-zag of simplicial homotopies consists of three simplicial homotopies.
The first homotopy inserts the cut $z_{i}$ and $w_i$ for each vertex in the subdivided sphere simultaneously.
The second converts all bead cuts to ordinary cuts, simultaneously for all vertices in the subdivided sphere.
The last homotopy removes all cuts but $z_{i},w_i$, $x$, and~$y$. 	

To complete the proof, we need only show that the map $\Cnec(f)$ is bijective on the disjoint summands indexed by ambient bordisms.
Since $\Cnec(f)$ is  just an inclusion on such disjoint summands, by definition, we need only prove surjectivity.
But applying the above construction to a single necklace in $\Cnec(\widetilde{H}_k)$ produces a zig-zag connecting this necklace to a necklace in the image of $\Cnec(f)$, which proves surjectivity.
\end{proof}

\def\sC{\mathscr{C}}
\def\cM{{\cal M}}
\def\cA{{\cal A}}
\def\Adj{{\rm Adj}}
\def\bAdj{\overline{\rm Adj}}

\begin{proposition}\label{handleequiv}
Fix $d≥1$ and $k≥1$.
Denote by $\cM$ the coproduct summand of the (derived) mapping simplicial presheaf
$$\Map_{Γ⨯Δ^d}(\Adj_{{\bf 1}_{d-2}}, R(\BBord_d^{\RR^d\times U\to U}))$$
(\cref{adjinbord}),
corresponding to maps that send the left adjoint 1-morphism to the connected component of bordism~$f$ (\cref{handle}).
The adjoint map $\cM⨯\Adj_{{\bf 1}_{d-2}}→R(\BBord_d^{\RR^d\times U\to U})$ restricts to maps $\cM⨯\bAdj_{{\bf 1}_{d-2}}→H_k$ and $\cM⨯\bar\eta_{{\bf 1}_{d-2}}→O_{k-1}$.
Below, we abbreviate $\bAdj_{{\bf 1}_{d-2}}$ as $\bAdj$ and likewise for $\bar\eta$.
The commutative square
$$\xymatrix{
\cM⨯\bar\eta \ar[r] \ar[d] & O_{k-1} \ar[d]\cr
\cM⨯\bAdj \ar[r] & H_k.\cr
}$$
is homotopy cocartesian in $\smcatdual_{\infty,d}$. 
\end{proposition}

\begin{proof}
The object $\bAdj$ contains $\bar\eta$ together with the following nondegenerate bisimplices:
\begin{itemize}
\item A $([1],[1])$-bisimplex $ε$.
\item A $([1],[1])$-bisimplex $gε$.
\item A $([2],[1])$-bisimplex witnessing $gε$ as the composition $\id_g ∘ ε$ of $\id_g$ and $ε$ in the $(d-1)$st direction, where $\id_g$ denotes the identity on~$g$ in the $d$th direction.
\item A $([1],[2])$-bisimplex witnessing $\id_g$ as the composition of $ηg$ and $gε$ in the $d$th direction.
\end{itemize}

The object $H_k$ contains $O_{k-1}$ together with the same types of nondegenerate bisimplices,
but with different moduli stacks.

Denote by $P$ the objectwise homotopy pushout $$P=\cM⨯\bAdj⊔_{\cM⨯\bar\eta}O_{k-1}.$$
Denote by $R(P)$ the fibrant replacement of~$P$ in the model category $\smcatdual_{∞,d}$.
We have an induced chain of maps
$$P\to R(P)\to R(H_k),$$
which induces maps $R(P)_ε→R(H_{k})_{ε}$, $R(P)_{g\epsilon}\to R(H_{k})_{gε}$, $R(P)_β→R(H_{k})_{β}$,
where the subscripts denote the corresponding coproduct summands of moduli stacks of bisimplices in $R(P)$ respectively $R(H)_k$.
It suffices to show that these induced maps are weak equivalences.

We have a map $P_ε→H_{k,ε}$ to the moduli stack given by the coproduct summand of bordisms in the connected component of $ε$ in $H_k([1],[1])$
that extracts the $ε$-cell from a map $\bAdj→H_k$.
Fix a vertex $v∈H_{k,ε}$ and consider the (objectwise) homotopy fibers $F_v$ and ${\cal F}_v$ of $P_ε→H_{k,ε}$ and $R(P)_ε→R(H_{k})_{ε}$ over~$v$.
We have an induced chain of maps
$$F_v\to {\cal F}_v\to \{v\}.$$
It suffices to show that the right map is a weak equivalence for any $v$, since this implies that $R( P)_ε→R(H_{k})_{ε}$ is a weak equivalence.

A point in $F_v$ is given by a choice of adjunction data such that the image of its counit in $H_{k,ε}$ is given by the vertex $v\in H_{k,ε}$.
Fix a point $a_0\in F_v$.
Then $a_0$ determines a $([2],[0])$-bisimplex with outer faces $f_0$, $g_0$ and inner face given by the composition, a $([1],[1])$ bisimplex $\eta_0$ with target $g_0f_0=k$, a $([1],[1])$ bisimplex $\epsilon_0$ with source $f_0g_0=h$.
Denote by $G_0$ the connected component of the identity map in the moduli stack of bordisms from~$g_0$ to itself.
Consider the map ${\cal F}_v→G_0$ that restricts to the counit and inserts the given $([1],[1])$-bisimplex $\epsilon$ in ${\cal F}_v$
into the following diagram:
\[\begin{tikzcd}[column sep=small]
	&& {} \\
	\\
	&&& {S^{d-k-1}\times D^{k-1}} && {S^{d-k-1}\times D^{k-1}} \\
	{D^{d-k}\times S^{k-2}} && {D^{d-k}\times S^{k-2}} && \epsilon \\
	& {\eta_0} && {S^{d-k-1}\times D^{k-1}} && {S^{d-k-1}\times D^{k-1}} \\
	{D^{d-k}\times S^{k-2}} && {D^{d-k}\times S^{k-2}}
	\arrow["{g_0}", from=3-4, to=4-3]
	\arrow[from=3-4, to=3-6]
	\arrow["{f_0}"', from=4-3, to=5-4]
	\arrow[from=5-4, to=5-6]
	\arrow["{S^{d-k-1}\times D^{k-1}\times [0,1]}", from=3-6, to=5-6]
	\arrow["h", from=3-4, to=5-4]
	\arrow[from=4-1, to=4-3]
	\arrow["{D^{d-k}\times S^{k-2}\times [0,1]}"', from=4-1, to=6-1]
	\arrow["k", from=4-3, to=6-3]
	\arrow[from=6-1, to=6-3]
	\arrow["{g_0}", from=5-4, to=6-3]
\end{tikzcd}\]

The moduli stack~$G_0$ is contractible by \cref{contractiblecylinders}, i.e., cylinders can be deformed by a Morse flow to their source cut.
Consider the map $G_0→{\cal F}_v$ that inserts a bordism in $G_0$ into the following diagram:
\[\begin{tikzcd}
	{} \\
	& {D^{d-k}\times S^{k-2}} \\
	{D^{d-k}\times S^{k-2}} & {S^{d-k-1}\times D^{k-1}} && {S^{d-k-1}\times D^{k-1}} \\
	{D^{d-k}\times S^{k-2}} && \epsilon_0 \\
	& {S^{d-k-1}\times D^{k-1}} && {S^{d-k-1}\times D^{k-1}}
	\arrow["{g_0}", from=3-2, to=4-1]
	\arrow[from=3-2, to=3-4]
	\arrow["{f_0}"', from=4-1, to=5-2]
	\arrow[from=5-2, to=5-4]
	\arrow["{S^{d-k-1}\times D^{k-1}\times [0,1]}", from=3-4, to=5-4]
	\arrow["h", from=3-2, to=5-2]
	\arrow[from=3-1, to=4-1]
	\arrow[from=2-2, to=3-2]
	\arrow[from=2-2, to=3-1]
\end{tikzcd}\]
The composition ${\cal F}_v→G_0→{\cal F}_v$ sends $\epsilon$ to the composition of the following diagram: 
$$
\begin{tikzcd}[column sep=small]
	&& {} \\
	\\
	&&& {D^{d-k}\times S^{k-2}} && {S^{d-k-1}\times D^{k-1}} \\
	{D^{d-k}\times S^{k-2}} && {D^{d-k}\times S^{k-2}} && \epsilon \\
	& {\eta_0} && {S^{d-k-1}\times D^{k-1}} && {S^{d-k-1}\times D^{k-1}} \\
	{D^{d-k}\times S^{k-2}} && {D^{d-k}\times S^{k-2}} && {\epsilon_0} \\
	&&& {S^{d-k-1}\times D^{k-1}} && {S^{d-k-1}\times D^{k-1}}
	\arrow["{g_0}", from=3-4, to=4-3]
	\arrow[from=3-4, to=3-6]
	\arrow["{f_0}"', from=4-3, to=5-4]
	\arrow[from=5-4, to=5-6]
	\arrow["{S^{d-k-1}\times D^{k-1}\times [0,1]}", from=3-6, to=5-6]
	\arrow["h", from=3-4, to=5-4]
	\arrow["{g_0}", from=5-4, to=6-3]
	\arrow["{f_0}"', from=6-3, to=7-4]
	\arrow["h", from=5-4, to=7-4]
	\arrow[from=7-4, to=7-6]
	\arrow[from=5-6, to=7-6]
	\arrow["k"', from=4-3, to=6-3]
	\arrow[from=4-1, to=4-3]
	\arrow["{D^{d-k}\times S^{k-2}\times[0,1]}"', from=4-1, to=6-1]
	\arrow[from=6-1, to=6-3]
\end{tikzcd}
$$
The tuple $(f_0,g_0,η_0,ε_0,β_0)$ provides the data for a map $\bAdj→P$.
Once composed with the fibrant replacement map $P\to R(P)$,
the map $\bAdj→P→R(P)$ extends (uniquely up to a contractible choice) to a map $\Adj→R(P)$.
In particular, we obtain the other triangle identity for $(f_0,g_0,η_0,ε_0)$,
which contracts the whole diagram to~$\epsilon_0$,
so the composition ${\cal F}_v→G_0→ {\cal F}_v$ is homotopic to the identity map.

Since the moduli stack $G_0$ is contractible, this forces ${\cal F}_v$ to also be contractible.
Hence, the map $R( P)_ε→H_{k,ε}$ is a weak equivalence.

The argument for $gε$ is analogous, using the same diagrams with $gε$ instead of~$ε$.
Hence, the map $R(P)_{gε}→R(H_{k})_{gε}$ is a weak equivalence.
Finally, the map $R(P)_β→R(H_{k})_{β}$ is a weak equivalence
because its fiber over any vertex in the codomain
is the moduli stack of adjunctions with $f$ as its left adjoint.
This moduli stack is contractible.
\end{proof}

\subsection{Cutting out handles}\label{index.filtration.section}

Throughout this section, we fix an arbitrary multisimplex ${\bf m}\in \Delta^{\times d-1}$ and an object $\langle \ell\rangle\in \Gamma$.
By convention, we set 
$$\BBord_{d}^{\RR^d\times U\to U}≔\BBord_{d}^{\RR^d\times U\to U}(\langle \ell\rangle,{\bf m}).$$
We will refer to all objects in this section as smooth simplicial spaces, since the partial evaluation leaves only the cartesian space direction and the last simplicial direction.

We now define a filtration on $\BBord_{d}^{\RR^d\times U\to U}$ as follows. 
\begin{definition}\label{index.filtration}
Let $-1\leq k\leq d$.
We define a subobject $B_k\into \BBord_{d}^{\RR^d\times U\to U}$
whose bordisms, once composed with a cylinder, admit a fiberwise Morse function $g$ that satisfies the following properties.
\begin{itemize}
\item $g$ induces the source cut and the target cut, in particular, $g^{-1}(0)=C_=^0$ and $g^{-1}(0,\infty)=C_>^0$ (Grady–Pavlov \cite[\ecref{EL-cut}]{GradyPavlov}).
\item The nondegenerate critical points of~$g$
are located in the interior of the core and have index at most $k$.
\end{itemize}
The case $k=-1$ corresponds to the case where the bordism has no fiberwise critical points. We have a filtration 
$$B_{-1}\into B_0\into \cdots \into B_k\into B_d=\BBord_{d}^{\RR^d\times U\to U},$$
where the last equality holds since any bordism admits a Morse function that is compatible with the source cut.
\end{definition}

In the following proofs, we will need to use the codescent property (Grady--Pavlov \cite[\ecref{EL-mainthm}]{GradyPavlov}) at each level of the index filtration \cref{index.filtration}. 

\begin{proposition}\label{bkpushout}
Fix $d≥0$ and $0\leq k\leq d$.
We recall the subobjects $\Ot_{k-1}$ and $\Ht_k$ from \cref{def.Ht}.
The diagram of inclusions
$$\xymatrix{
\Ot_{k-1}\ar[d]\ar[r] & B_{k-1}\ar[d]\\
\Ht_k\ar[r] & B_k\\
}$$
is a homotopy pushout in $\smcatdual_{∞,d}$.
\end{proposition}
\begin{proof}
As in \cref{handlecutout},
we will show that after passing to stalks on $\cartsp$, evaluating on an arbitrary object $\langle \ell\rangle\in \Gamma$ and multisimplex ${\bf m}\in \Delta^{d-1}$,
as well as a simplex $[l]∈Δ$ (corresponding to $l$-dimensional deformations of cut tuples),
the canonical map $f:\Ht_k\sqcup_{\Ot_{k-1}}B_{k-1}\to B_k$ is a weak equivalence in the Joyal model structure on simplicial sets.

In the resulting simplicial set $B_k$, vertices are $l$-dimensional families of bordisms with a single cut in the $d$th simplicial direction.
Furthermore, after evaluating at $[0]∈Δ$ (corresponding to the $d$th factor of~$Δ$),
the intersection of $B_{k-1}$ and $\tilde H_k$ equals $\tilde O_{k-1}$ (as a set).
Thus, $f$ is bijective on vertices.
By Grady--Pavlov \cite[\ecref{EL-necklaces.prop}]{GradyPavlov}, it suffices to show that for all vertices $x,y\in B_k$, the induced map on Dugger--Spivak necklace categories   
$$\Cnec(f):\Cnec(\Ht_k\sqcup_{\Ot_{k-1}}B_{k-1})(x,y)\to \Cnec(B_k)(x,y)$$
is a weak equivalence in the classical model structure on simplicial sets.

Explicitly, the Dugger–Spivak necklace categories associated to the pushout can be described as follows.

\begin{itemize}
\setlength\itemsep{0em}
\item[$\circ$] A necklace of length $t$ in the category $\Cnec(\Ht_k\sqcup_{\Ot_{k-1}}B_{k-1})(x,y)$ is a composable tuple of embedded bordisms $(M_s,C_s)_{s=1}^t$ called \emph{beads}. The tuple satisfies the following property.
\begin{enumerate}
\setlength\itemsep{0em}
\item[(P)] Each bead $(M_s,C_{s})$ is either in $\Ht_k$ or $B_{k-1}$. 
\end{enumerate}

\item[$\circ$] A morphism in  $\Cnec(\Ht_k\sqcup_{\Ot_{k-1}}B_{k-1})(x,y)$ is a composition of two basic types of morphisms. 
\begin{itemize}
\setlength\itemsep{0em}
\item The first type inserts a cut $\Upsilon$ in a bead, obtaining a new bead $(M_s,C_s\cup\Upsilon,e_s,f_s)$;

\item The second type glues consecutive beads together, reducing the number of beads by one.  
\end{itemize}
\end{itemize}
The necklace categories of $B_{k}$ are identical, except Property P above is replaced by the requirement that each bead belongs to~$B_k$.
The map $\Cnec(f)$ simply regards a necklace in $\Ht_k\sqcup_{\Ot_{k-1}}B_{k-1}$ as a necklace in $B_k$ (by forgetting Property P). 

The map $\Cnec(f)$ splits as a coproduct, indexed by the germs of cores of ambient bordisms with endcuts $x$ and~$y$.
(In other words, morphisms in the necklace category can only exist between necklaces whose underlying bordisms from~$x$ to~$y$ coincide.)
We claim that each disjoint summand of ${\cal N}(\Cnec(B_{k})(x,y))$ and ${\cal N}(\Cnec(\Ht_k\sqcup_{\Ot_{k-1}}B_{k-1})(x,y))$
corresponding to the ambient bordism, is contractible.

For ${\cal N}(\Cnec(B_k)(x,y))$
we observe that $B_k$ is a subobject of $\BBord_d^{\RR^d⨯U→U}$ that is closed under composition of bordisms,
and, moreover, when evaluated on some fixed object of $\cartsp⨯Γ⨯Δ^{⨯d}$,
the inclusion $B_k→\BBord_d^{\RR^d⨯U→U}$
becomes an inclusion of connected components
(namely, connected components indexed by pairs $(M,P)$, where $M$ admits a Morse function with critical points of index at most~$k$).
Thus, the simplicial sets ${\cal N}(\Cnec(B_{k})(x,y))$ coincide
with the simplicial sets ${\cal N}(\Cnec(\BBord_d^{\RR^d⨯U→U})(x,y))$,
and the latter were proved to be contractible in Grady–Pavlov \cite[\ecref{EL-contractible.cuts}]{GradyPavlov}.

For ${\cal N}(\Cnec(\Ht_k\sqcup_{\Ot_{k-1}}B_{k-1})(x,y))$,
the proof of Grady–Pavlov \cite[\ecref{EL-contractible.cuts}]{GradyPavlov} continues to work with $\Ht_k\sqcup_{\Ot_{k-1}}B_{k-1}$ replacing $B_{i-1}$.
As explained there, we invoke the simplicial Whitehead theorem (Grady–Pavlov \cite[\ecref{EL-simplicial.whitehead}]{GradyPavlov}) and show that any map $$f:\Sd^k ∂Δ[n]→{\cal N}(\Cnec(\Ht_k\sqcup_{\Ot_{k-1}}B_{k-1})(x,y))$$
admits a zigzag of simplicial homotopies to a map~$f'$ of the same type that factors through
$${\cal N}(\Cnec(\Ht_k\sqcup_{\Ot_{k-1}}B_{k-1})(x,x)),$$
i.e., the target cut coincides with the source cut~$x$.
Since the latter simplicial set is contractible (by virtue of being the nerve of a category with a terminal object),
it remains to construct the claimed zigzag.

The zigzag is constructed in the same way as in Grady–Pavlov \cite[\ecref{EL-contractible.cuts}]{GradyPavlov},
by constructing a cut tuple $Ψ$ from~$x$ to~$y$ (i.e., $Ψ=(Ψ_0=x,Ψ_1,…,Ψ_{|Ψ|-1},Ψ_{|Ψ|}=y)$),
and homotoping the map~$f$ inductively so that after the $j$th step of the induction,
the homotoped map~$f$ factors through
$${\cal N}(\Cnec(\Ht_k\sqcup_{\Ot_{k-1}}B_{k-1})(x,Ψ_j)).$$
Here $j$ runs in the decreasing order from~$|Ψ|$ to~0, with $j=|Ψ|$ being the base of the induction, which is satisfied by the assumption on~$f$
and $j=0$ being the final step of the induction, proving the desired claim.

The cut tuple~$Ψ$ is constructed in
Grady–Pavlov \cite[\ecref{EL-morse.decomposition}]{GradyPavlov},
with the input data for that result constructed as follows.
Fix a Morse function~$\phi$ on the bordism from~$x$ to~$y$
such that $\phi$ is compatible with the source cut and the nondegenerate critical points of $\phi$ have index $\leq k$.
Fix an open cover~$W$ of the bordism from~$x$ to~$y$
such that every critical point~$p$ of the Morse function~$\phi$ is contained in a unique element of the open cover~$W$,
with some open ball around~$p$ being disjoint from all other elements of~$W$.

The argument of Grady–Pavlov \cite[\ecref{EL-contractible.cuts}]{GradyPavlov}
then constructs the desired zigzag of simplicial homotopies connecting~$f$ and~$f'$.
All beads created in this construction contain either no critical points of~$\phi$ or exactly one critical point of~$\phi$.
Therefore, all such beads belong to $\Ht_k$ or $B_{k-1}$.

To complete the proof, we need only show that the map $\Cnec(f)$ is bijective on the disjoint summands indexed by ambient bordisms.
Since $\Cnec(f)$ is  just an inclusion on such disjoint summands, by definition, we need only prove surjectivity.
But the construction in Grady–Pavlov \cite[\ecref{EL-morse.decomposition}]{GradyPavlov}
decomposes any ambient bordism in $B_k$ into a composition of bordisms in $\Ht_k$ and $B_{k-1}$
by cutting out strips containing the critical points of the Morse function, one critical point at a time,
which implies surjectivity. 
\end{proof}

\subsection{The index filtration}

\begin{notation}\label{jmorphisms}
Fix $d\geq 1$. Let ${\bf 1}_j$ be the multisimplex with components ${\bf 1}_i=[1]$ for $i\leq j$ and ${\bf 1}_i=[0]$ for $d\geq i>j$. Let $X$ be a $d$-fold complete Segal space. We will refer to vertices in $X({\bf 1}_j)$ as \emph{$j$-morphisms}. We will often denote a $j$-morphism as an arrow $f:a\to b$, where $a,b\in X({\bf 1}_{j-1})$, and $f\in X({\bf 1}_j)$ satisfies $d^j_1(f)=a$ and $d^j_0(f)=b$, where $d^j_l$, $l\in \{0,1\}$, denotes the simplicial map in the $j$th direction.
\end{notation}

In the following proof, we will make use of the exchange principle of Lurie \cite[Lemma 3.4.21]{Lurie.TFT}, adapting its statement to our setting.
In particular, we will define the relevant objects, morphisms, 2-morphisms and 3-morphisms as multisimplices in a $d$-fold Segal space.
With these modifications, the proof is verbatim as in Lurie \cite[Lemma 3.4.21]{Lurie.TFT}

\begin{proposition}\label{exchange}
Fix $d\geq 2$ and let $X$ be a $d$-fold complete Segal space. 
Let $f:x\to y$ and $f^{\dagger}:y\to x$ be multisimplices in degree ${\bf 1}_1$ (see \cref{jmorphisms}).
Consider multisimplices of degree ${\bf 1}_2$ of the form
$$\xymatrix@C=.5cm@R=.5cm{
x \ar[rr]^{\id}\ar[dd]_{\id} && x \ar[dd]^{f^{\dagger}\circ f }\\
&u &\\
x \ar[rr] _{\id} && x,\\
} \qquad
\xymatrix@C=.5cm@R=.39cm{
y \ar[rr]^{\id}\ar[dd]_{\id} && y \ar[dd]^{f\circ f^{\dagger} }\\
&u' &\\
y \ar[rr] _{\id} && y.\\
}$$
The right edges are given by a choice of a composition, which exists in any Segal space.
Using \cref{jmorphisms}, we will denote these 2-morphisms by $u:\id_x\to f^{\dagger}\circ f$ and $u':\id_{y}\to f\circ f^{\dagger}$. Suppose the following are satisfied.
\begin{enumerate}
\item The 2-morphism $u:\id_x\to f\circ f^{\dagger}$ is the unit of an adjunction in $X$, with a compatible counit~$v$.
That is, there is a map $(u,v,t):{\rm Adj}_1\to X$ that sends the generators $\epsilon\mapsto v$, $\eta\mapsto u$ and $t\mapsto t'$, where $t'$ implements the triangle identity.  

The triangle identities imply that every 3-morphism $\alpha:f^{\dagger}u'\to uf^{\dagger}$, which is a multisimplex  of degree ${\bf 1}_3$ of the form 
$$\xymatrix@C=1cm@R=.6cm{
y\ar[dddd]_-{\id}\ar[drr]^-{\id}\ar[rrr]^-{\id} &&& y
\ar@{}[dd]^<<<<{f^{\dagger}}
\ar[dd]|<<<<<<<<{\strut}
\ar[drr]^-{\id} && \\
&& y\ar[dddd]^-{f\circ f^{\dagger}}
\ar[rrr]^<<<<{\id}
&& \save []+<0em,-1em>*+{\id_{f^{\dagger}}}\restore & y\ar[dd]^-{f^{\dagger}}\\
&\save []+<0em,-1em>*+{u'} \restore & & x\ar[dddd]^-{\id}\ar[drr]^-{\id} && \\
&&&&& x\ar[dddd]^-{f^{\dagger}\circ f} \\
y\ar[dd]_-{f^{\dagger}}\ar[drr] &  & & & \save []+<0em,-1em>*+{u} \restore &\\
& \save []+<0em,-1em>*+{\id_{f^{\dagger}}} \restore & y\ar[dd]^<<<<{f^{\dagger}}&& \\
x\ar[drr]^-{\id}\ar@{}[rrr]^-{\id} \ar[rrr]|<<<<<<<<<<<<<<<<<<<<<<<<{\quad} &&& x\ar[drr]^-{\id} &&\\
&&x \ar[rrr]^-{\id} &&& x,\\
}$$
has an adjoint $\beta:u'\circ v\to \id_{f\circ f^{\dagger}}$, which is a multisimplex of degree ${\bf 1}_3$ of the form 
$$\xymatrix@C=1cm@R=.6cm{
y\ar[drr]^-{\id}\ar[ddd]_-{f\circ f^{\dagger}} \ar[rrr]^-{\id} &  && y\ar[ddrrrr]^-{\id}\ar@{}[ddd]^>>>>>>{f\circ f^{\dagger}} \ar[ddd]|<<<<<<<<<<<<<{\strut} &&& \\
& \save []+<0em,-3em>*+{v} \restore & y\ar[ddd]^-{\id}\ar[drr]^>>>>>>>>{\id} &&& & \\
& & & & y\ar[ddd]^<<<<<<<<<<<{f\circ f^{\dagger}}\ar[rrr]^-{\id}  &&& y\ar[ddd]^-{f\circ f^{\dagger}} \\
y \ar[drr]^-{\id}\ar@{}[rrr]^-{\id} \ar[rrr]|<<<<<<<<<<<<<<<<<<<<<<<<{\quad} &  & & y \ar[ddrrrr]|<<<<<<<<<<<<<{\phantom{ABC}} & \save []+<-2.5em,-1.5em>*+{u'} \restore &   &&\\
&  & y\ar[drr]^-{\id} & &&& &&\\
&  & & & y\ar[rrr]^-{\id} && &y&\\
}$$

Explicitly, this correspondence is the following composite map
\begin{align*}
\map(u'\circ v,\id_{f\circ f^{\dagger}})&\overset{\id_{f^{\dagger}}}{\longrightarrow} \map(f^{\dagger}(u'\circ v),\id_{f^{\dagger}\circ f\circ f^{\dagger}})
\\
&\overset{uf^{\dagger}}{\longrightarrow} \map(f^{\dagger} u'\circ f^{\dagger}v\circ u f^{\dagger}, \id_{f^{\dagger}\circ f\circ f^{\dagger}}\circ  (uf^{\dagger}))
\\
&\overset{t'}{\longrightarrow}\map(f^{\dagger}u' ,uf^{\dagger}),
\end{align*}
where the first two maps are obtained by whiskering with the corresponding 2-morphism. The last map precomposes with the triangle identity $t':f^{\dagger}u\circ v f^{\dagger}\to \id_{f^{\dagger}\circ f\circ f^{\dagger}}$. The composition is depicted diagrammatically in \cref{nastyfig}. 
\item The 2-morphism $u'$ is the unit of an adjunction in $\mathscr{C}$, with compatible counits $v'$. The triangle identity implies that every 3-morphism $\alpha:\id\times u'\to u\times \id$ is adjoint to a morphism $\gamma:\id_{f^{\dagger}\circ f}\to u\circ v'$, as in 1. 
\item Both $u$ and $v$ admit left adjoints. 
\end{enumerate}
Then $\beta$ is the counit of an adjunction in the homotopy 2-category if and only if $\gamma$ is the unit of an adjunction. 
\end{proposition}
\begin{proof}
The proof is verbatim Lurie \cite[Lemma 3.4.21]{Lurie.TFT}, interpreting the morphisms $u,u',v,v',\alpha,\beta$ and $\gamma$ as corresponding multisimplices in $X$, as described above.
\end{proof}

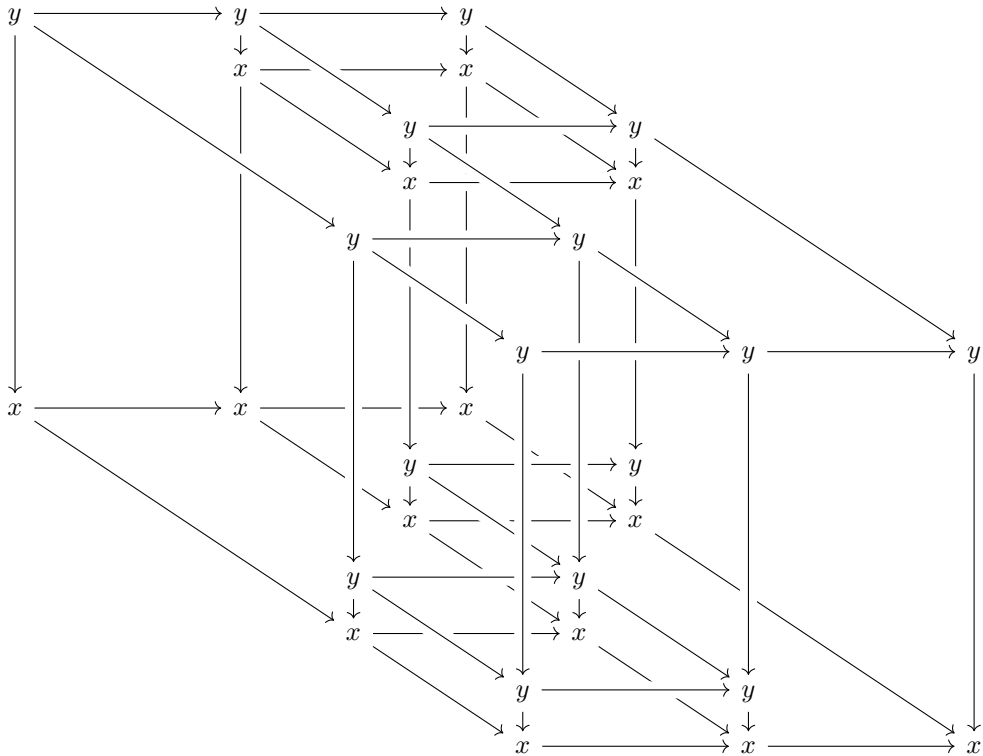
\begin{figure}[ht]
$$
\begin{tikzcd}[sep={between origins,.75cm}]
y \ar[ddddddd] \ar[rrrr] &  &   &   & y \ar[d] \ar[rrrr]&   &   &   & y \ar[d] \ar[ddrrr]  &   &   &   &   &   &   &   &   &   & \\
  &   &   &   & x \ar[ddrrr]\ar[dddddd]\ar[rrrr] &   &   &   & x \ar[ddrrr]\ar[dddddd] &   &   &   &   &   &   &   &   &   & \\
  &   &   &   &   &   &   & y \ar[from=uulll, crossing over] \ar[rrrr, crossing over]\ar[d] &   &   &   & y \ar[d]\ar[ddddrrrrrr] &   &   &   &   &   &   & \\
  &   &   &   &   &   &   & x\ar[rrrr, crossing over]  &   &   &   & x \ar[ddddd]  &   &   &   &   &   &   & \\
  &   &   &   &   &   & y\ar[from=uuuullllll, crossing over] &   &   &   & y\ar[from=uulll, crossing over] &   &   &   &   &   &   &   & \\
  &   &   &   &   &   &   &   &   &   &   &   &   &   &   &   &   &   &  \\
  &   &   &   &   &   &   &   &   & y &   &   &   & y \ar[from=uulll, crossing over]\ar[rrrr]&   &   &   & y\ar[ddddddd] & \\  
x \ar[rrrr]\ar[ddddrrrrrr] &   &   &   & x \ar[rrrr]\ar[ddrrr] &   &   &   & x\ar[ddrrr] &   &   &   &   &   &   &   &   &   & \\  
  &   &   &   &   &   &   & y \ar[rrrr, crossing over]\ar[d] \ar[from=uuuuu, crossing over] &   &   &   & y \ar[d] &   &   &   &   &   &   & \\
  &   &   &   &   &   &   & x \ar[rrrr] &   &   &   & x\ar[ddddrrrrrr] &   &   &   &   &   &   & \\
  &   &   &   &   &   & y\ar[from=uuuuuu, crossing over]\ar[d]&   &   &   & y \ar[from=uulll, crossing over]\ar[d]\ar[ddrrr]\ar[from=uuuuuu, crossing over]  &   &   &   &   &   &   &   & \\
  &   &   &   &   &   & x\ar[ddrrr] &   &   &   & x \ar[from=uulll, crossing over]\ar[ddrrr] &   &   &   &   &   &   &   & \\
  &   &   &   &   &   &   &   &   & y\ar[d]\ar[rrrr, crossing over] &   &   &   & y\ar[d]\ar[from=uuuuuu, crossing over]  &   &   &   &   & \\  
  &   &   &   &   &   &   &   &   & x\ar[rrrr] &   &   &   & x\ar[rrrr]  &   &   &   & x &  
\arrow[from=11-7, to=11-11, crossing over]
\arrow[from=12-7, to=12-11, crossing over]
\arrow[from=11-7, to=13-10, crossing over]
\arrow[from=5-7,to=7-10, crossing over]
\arrow[from=5-7,to=5-11, crossing over]
\arrow[from=7-10,to=7-14, crossing over]
\arrow[from=7-10,to=13-10, crossing over]    
\end{tikzcd}
$$
\caption{A diagram describing the adjunction $\beta\mapsto \alpha$ in \cref{exchange}. Starting with $\beta:u'\circ v\to \id_{f\circ f^{\dagger}}$, depicted by the right front 3-morphism in the above diagram, we obtain a 3-morphism $\alpha:f^{\dagger}u'\to uf^{\dagger}$. The left face is the composition of $\id_{f^{\dagger}\circ f\circ f^{\dagger}}$ with $f^{\dagger}u'$. The right face is the composition of $f^{\dagger}u$ with $\id_{f^{\dagger}\circ f\circ f^{\dagger}}$. The 3-morphism appearing in the back left corner implements the triangle identity. The center bordism is the composition $uf^{\dagger}\circ f^{\dagger}v\circ f^{\dagger}u'$. }
\label{nastyfig}
\end{figure}

We will use the following notation in the proof of \cref{bkequivalence}.

\begin{definition}
\label{def.summands}
Given a fibrant smooth symmetric monoidal $(∞,d)$-category with duals $\sC$,
we denote by
$$\Funmon(H_k,\sC)_\Adj⊂\Funmon(H_k,\sC)$$
the coproduct summand
consisting of connected components containing maps $H_k→\sC$ (\cref{handle})
that send the quintuple of bordisms $(f,g,η,ε,β)$ (\cref{handle}) to a quintuple in~$\sC$ that extends to an adjunction.

Denote by
$$\Funmon(O_{k-1},\sC)_\unit⊂\Funmon(O_{k-1},\sC)$$
the coproduct summand 
consisting of connected components containing maps $O_{k-1}→\sC$ (\cref{def.O})
that send the triple of bordisms $(f,g,η)$ (\cref{handle}) to a triple in~$\sC$ that extends to an adjunction.
\end{definition}

\begin{proposition}\label{bkequivalence}
Fix $d≥1$ and a fibrant object $\mathscr{C}\in \smcatdual_{∞,d}$.
For $k\geq 2$, the inclusion $B_{k-1}\into B_k$ is a weak equivalence in $\smcatdual_{∞,d}$,
hence it induces a weak equivalence of derived internal homs
$$\Funmon(B_k,\mathscr{C})\to\Funmon(B_{k-1},\mathscr{C}).$$
Moreover, we have a weak equivalence of derived mapping spaces 
$$\Funmon(B_1,\mathscr{C})\to \Funmon(B_0,\mathscr{C})_\unitable$$
where the subscript $\unitable$ is defined in the proof.
\end{proposition}

\begin{proof}
From \cref{bkpushout}, for any $k≥0$ we have a homotopy pushout diagram in $\smcatdual_{∞,d}$
$$\xymatrix{
\Ot_{k-1}\ar[d]\ar[r] & B_{k-1}\ar[d]\\
\Ht_k\ar[r] & B_k.\\
}$$
Let $\mathscr{C}\in \smcatdual_{∞,d}$ be a local object.  Taking derived internal homs from the above homotopy pushout into $\mathscr{C}$ gives a corresponding homotopy pullback
$$\xymatrix{
\Funmon(\Ot_{k-1},\mathscr{C})& \ar[l]\Funmon(B_{k-1},\mathscr{C})\\
\Funmon(\Ht_k,\mathscr{C})\ar[u] &\ar[l]\ar[u]\Funmon(B_k,\mathscr{C}).\\
}$$
By \cref{handlecutout}, we have a weak equivalence $\Ht_k\simeq \Ot_{k-1}\sqcup_{O_{k-1}}H_k$.
Hence, for any $k≥0$ we can further expand the above homotopy fiber product to a diagram of homotopy pullbacks 
\begin{equation}
\xymatrix{
\Funmon(O_{k-1},\mathscr{C})&\ar[l]\Funmon(\Ot_{k-1},\mathscr{C})& \ar[l]\Funmon(B_{k-1},\mathscr{C})\\
\Funmon(H_{k},\mathscr{C}) \ar[u] &\ar[l]\Funmon(\Ht_k,\mathscr{C})\ar[u] &\ar[l]\ar[u]\Funmon(B_k,\mathscr{C}).\\
}\end{equation}

By the pasting lemma for homotopy pullbacks, the bottom left square is a homotopy pullback square.
Hence the outer bottom square is also a homotopy pullback square by pasting.
For $k≥1$, by \cref{handleequiv},
the map $O_{k-1}→H_k$ fits into a homotopy cocartesian square 
$$\xymatrix{
\cM⨯\bar\eta \ar[r] \ar[d] & O_{k-1} \ar[d]\cr
\cM⨯\bAdj \ar[r] & H_k.\cr
}.$$
Mapping out into a fibrant object $\mathscr{C}$, it follows from Riehl--Verity \cite[Proposition 4.4.7]{RiehlVerity} that we have an equivalence on disjoint summands
$$\Funmon({\cal M}\times \overline{\rm Adj},\mathscr{C})_{\rm Adj}\to \Funmon({\cal M}\times \overline{\eta},\mathscr{C})_{\unit}.$$
The above homotopy cocartesian square implies that we also have an equivalence corresponding disjoint summands 
$$\Funmon(H_{k},\mathscr{C})_{\rm Adj}\to \Funmon(O_{k-1},\mathscr{C})_{\unit},$$ where subscripts are defined in \cref{def.summands}.

We claim that the above diagram factors as 
\begin{equation}\label{factorunitsadj}
\xymatrix{
\Funmon(O_{k-1},\mathscr{C})\\
\Funmon(O_{k-1},\mathscr{C})_{\unit}\ar@{>->}[u]^-{l} &  \ar[lu]^-{g} \ar[l]\Funmon(B_{k-1},\mathscr{C})\\
\Funmon(H_{k},\mathscr{C})_{\rm Adj}\ar[u]^-{f} & \ar[u] \ar[l]\ar[u]\Funmon(B_k,\mathscr{C}).\\
}\end{equation}
for $k\geq 2$. The bottom map factors by definition, since the bordisms in $H_k$ participate in an adjunction in $B_{k}$. 
For the top map, we apply the exchange principle from \cref{exchange}. 
Fix a vertex $Z:B_{k-1}\to \mathscr{C}$.
Fix a bordism in the moduli stack ${\cal M}$, i.e., a bordism of the form \cref{unitforreal}, which determines bordisms 
$$D^{k-2}\times D^{d-k}:D^{k-2}\times S^{d-k-1}\to S^{k-3}\times D^{d-k}, \qquad D^{k-2}\times D^{d-k}: S^{k-3}\times D^{d-k} \to D^{k-2}\times S^{d-k-1},$$
that are degenerate in the last two multisimplicial directions.  Recalling \cref{jmorphisms}, we set 
$$x=Z(D^{k-2}\times S^{d-k-1}),\quad y=Z(S^{k-3}\times D^{d-k}),$$ $$f^{\dagger}=Z(D^{k-2}\times D^{d-k}):y\to x,\quad f=Z(D^{k-2}\times D^{d-k}):x\to y.$$
A choice of homotopy that contracts the composite bordism $D^{k-1}\times S^{d-k-1}=D^{k-2}\times D^{d-k}\circ D^{k-2}\times D^{d-k} $ to $D^{k-2}\times S^{d-k-1}$ gives rise to a simplicial homotopy $\id_x\simeq Z(D^{k-1}\times S^{d-k-1})$. We also have an identification $f^{\dagger}\circ f= Z(S^{k-2}\times D^{d-k})$. 

Now define $u=Z(D^{d-k}\times D^{k-1}):\id_{x}\to f^{\dagger}\circ f$. This is the value of $Z$ on the upper left edge of the second diagram in \cref{unitforreal}. Define also $u'≔Z(D^{k-2}\times D^{d-k+1}):\id_y\to f\circ f^{\dagger}$. Finally, let $\gamma$ be the value of $Z$ on the bordism given by replacing $k$ with $k+1$ in \cref{unitforreal}. Then $\beta$ is the value of $Z$ on the bordism \cref{counitforreal}, which can be seen using the triangle identity \cref{triangleid}.

We are now in a position to apply \cref{exchange}. By definition, the composition 
$$\eta_{{\bf 1}_{d-2}}\to B_{k-1}\overset{Z}{\to} \mathscr{C}$$
picks out $\gamma:\id_{f^{\dagger}\circ f}\to u\circ v'$. Since $\beta$ is obtained as the image of the composite map 
$$\epsilon_{{\bf 1}_{d-2}}\to {\rm Adj}_{{\bf 1}_{d-2}}\to B_{k-1}\to \mathscr{C},$$
it is the counit of an adjunction. Then \cref{exchange} implies that $\gamma$ is also the unit of an adjunction. Hence the restriction map $\Funmon(B_{k-1},\mathscr{C})\to \Map(\eta_{{\bf 1}_{d-2}},\mathscr{C})$ indeed factors through $\unit_{{\bf 1}_{d-2}}(\mathscr{C})$.

Thus, for $k≥2$, by Riehl--Verity \cite[Proposition 4.4.12]{RiehlVerity}, the map $f$ in \cref{factorunitsadj} is a weak equivalence.
Since the bottom square in \cref{factorunitsadj} is homotopy cartesian, this shows that $B_{k-1}\into B_k$ is a local equivalence, completing the proof for $k≥2$.

For $k=1$, denote by $\Funmon(B_0,\mathscr{C})_\unitable$
the subobject given by union of the connected components of $\Funmon(B_0,\mathscr{C})$
such that the restriction of the map~$g$ factors through the map~$l$ in \cref{factorunitsadj}.
Replacing $\Funmon(B_0,\mathscr{C})$ with $\Funmon(B_0,\mathscr{C})_\unitable$
in the upper right corner of \cref{factorunitsadj} produces a commutative diagram
whose bottom square is homotopy cartesian.
Thus, the induced map $\Funmon(B_1,\mathscr{C})→\Funmon(B_0,\mathscr{C})_\unitable$ is a weak equivalence.
\end{proof}

\begin{proposition}\label{b0equivalence}
Fix $d≥1$ and $\mathscr{C}\in \smcatdual_{∞,d}$.
We have a weak equivalence $$\Funmon(B_0,\mathscr{C})_\unitable\to \Funmon(B_{-1},\mathscr{C}) \times_{\Funmon(O_{-1},\mathscr{C})} \Funmon(H_0,\mathscr{C})_{\unitable},$$
where the subscript $\unitable$ on both sides is defined in the proof of \cref{bkequivalence}:
we require that the morphism given by the image of $\emptyset\to S^{d-1}$
is the unit of an adjunction involving morphisms given by evaluating on bordisms $D^{d-1}:\emptyset\to S^{d-2}$ and $D^{d-1}:S^{d-2}\to\emptyset$ (which compose to $S^{d-1}$).
\end{proposition}

\begin{proof}
We apply \cref{bkpushout} for $k=0$,
which produces a weak equivalence
$$\Funmon(B_0,\mathscr{C})\to\Funmon(B_{-1},\mathscr{C}) \times_{\Funmon(\tilde O_{-1},\mathscr{C})} \Funmon(\tilde H_0,\mathscr{C})=\Funmon(B_{-1},\mathscr{C}) \times_{\Funmon(O_{-1},\mathscr{C})} \Funmon(H_0,\mathscr{C}).$$
Taking the connected components on both sides corresponding to the unit condition completes the proof.
\end{proof}

\begin{proposition}\label{contractiblecylinders}
The object $B_{-1}\in\smcatdual_{\infty,d}$ is homotopy constant with respect to the $d$th factor of~$\Delta$.
\end{proposition}

\begin{proof}
The idea of the proof is that bordisms in $B_{-1}$ are cylinders in the $d$th direction, so a Morse flow can be used to contract these cylinders and make them degenerate.

We evaluate $B_{-1}$ on an arbitrary object of $\cartsp⨯Γ⨯Δ^{⨯(d-1)}$
and work exclusively with the resulting simplicial space,
which we also denote by $B_{-1}$.

Recall (Grady–Pavlov \cite[\ecref{EL-bordism.enriched}]{GradyPavlov})
that in the case of representable structures given by some object $(\RR^d⨯U→U)∈\FEmb_d$,
the simplicial sets $B_{-1}([n])$
admit a simple description: their $l$-simplices
are smooth $l$-dimensional families of $n$-chains of embedded bordisms in~$\RR^d$.
(Rather, we have germs of such families around $Δ^l$, which we refer to as $Δ^l$-families from now on.)

Consider the simplicial subset $B'_{-1}([n])⊂B_{-1}([n])$
consisting of those simplices
for which the corresponding smooth $Δ^l$-families of embedded bordisms
admit a Morse function without critical points that vanishes on the 0th cut (corresponding to the vertex $0∈[n]$).
The inclusion $B'_{-1}([n])⊂B_{-1}([n])$ is a simplicial weak equivalence
by the simplicial Whitehead theorem
(see, for example, Grady–Pavlov \cite[\ecref{EL-simplicial.whitehead}]{GradyPavlov}).
Indeed, pick a $l$-simplex in $B_{-1}([n])$, which is given by a smooth $Δ^l$-family of embedded bordisms.
For any $t∈Δ^l$, the bordism corresponding to the point~$t$ admits a Morse function that vanishes on the source cut
(Grady–Pavlov \cite[\ecref{EL-morse.decomposition}]{GradyPavlov}).
Therefore, we can find a family of such functions in some open neighborhood of~$t$.
By repeatedly subdividing $Δ^l$, we can ensure that each of the resulting simplices admits such a family of Morse functions,
and the simplicial Whitehead theorem concludes the argument.
Thus, from now we work with $B'$ instead of~$B$.

We apply the simplicial Whitehead theorem (Grady–Pavlov \cite[\ecref{EL-simplicial.whitehead}]{GradyPavlov})
to the simplicial degeneration map $B'_{-1}([0])→B'_{-1}([1])$.
(The general case of simplicial degeneration maps $B'_{-1}([0])→B'_{-1}([l])$ is treated identically.)

Suppose we have a map $∂Δ^a → B'_{-1}([0])$
together with a filling $Δ^a → B'_{-1}([1])$.
That is to say, we have a $Δ^a$-smooth family of embedded cylinders
and the cylinders over the points of $∂Δ^a$ are degenerate.
Pick a smooth family of Morse functions without critical points that vanishes on the source cuts of cylinders (see Grady–Pavlov \cite[\ecref{EL-morse.decomposition}]{GradyPavlov}).
The family of Morse functions has germs around the cores of cylinders, so can be extended to an actual Morse function in an open neighborhood of the core.
Pick some $ε>0$ such that the preimage of $[-ε,0]$ has no critical points.

We can now glue the preimage of $[-ε,0]$, which is a family of cylinders, to the source cut, ensuring the source and target cuts are disjoint for all cylinders.
As explained in \cref{one.dimensional.bordisms}, gluing a cylinder produces a simplicial homotopy from the original map $Δ^a→B'_{-1}([0])$
to a map of the same form that encodes the glued cylinders.
The homotopy itself is implemented by the Morse flow.

The source and target cuts are now disjoint for all cylinders.
Pick a new family of Morse functions that vanishes on the source and is constant on the target bordisms (see Grady–Pavlov \cite[\ecref{EL-morse.decomposition}]{GradyPavlov}).

The Morse flow up to the value $-ε$ for the newly constructed Morse functions without critical points contracts the cylinders
and can be implemented via a simplicial homotopy in the simplicial set $\Cut$ of \cref{enrichedbordstr}, as explained in \cref{one.dimensional.bordisms}.

The resulting families are degenerate and land inside $B'_{-1}([0])$.

Finally, use the Morse flow again to move the resulting degenerate cylinders from~$-ε$ to~$0$.

We have constructed a chain of three simplicial homotopies from the given map $Δ^a→B'_{-1}([1])$ to a map $Δ^a→B'_{-1}([1])$ that lands in the image of $B'_{-1}([0])$.

This is not yet a homotopy relative boundary $∂Δ^a$, as required by the simplicial Whitehead theorem,
since the boundary is itself moved by these homotopies around a loop.
However, these loops are themselves contractible: geometrically, the loop moves the boundary cuts from~0 to~$-ε$ and then from~$-ε$ back to~0.
A smooth 2-homotopy contracts these loops to constant loops by smoothly deforming $ε$ to~0.
\end{proof}

\begin{proposition}\label{dimensional.reduction}
We have a weak equivalence
$$i:\BBord_{d-1}^{\iota^*_{d-1}(\RR^d\times U\to U)}\overset{\simeq}{\to} B_{-1}([0]).$$
\end{proposition}

\begin{proof}
We evaluate on some arbitrary $U∈\cartsp$, $\langle\ell\rangle∈Γ$, ${\bf m}∈Δ^{⨯(d-1)}$.
The resulting spaces are the diagonals of simplicial groupoids from \cref{enrichedbordstr},
and we define a morphism~$f$ of these simplicial groupoids.

Recall that the simplicial set of objects is given by 
$${\rm Ob}≔\coprod_{(M,P)}{\cal S}_{\square}(M\times U\to U)\times\Cut_\square(M\times U\to U,P),$$
where the coproduct ranges over pairs $(M,P)$ as in \cref{bord},
where $M$ has dimension~$d-1$ on the left side and dimension~$d$ on the right side.
When we map to the right side, we have to add an additional cut, corresponding to an object (0-simplex) in the $d$th direction.
The map $f$ sends $(M,P)$ to $(M⨯\RR,P')$,
where $P':M⨯\RR→M→\langle\ell\rangle$ simply projects onto~$M$ and applies~$P$.
The cut in the $d$th direction is $(M⨯\RR_{<0}⨯U→U,M⨯\{0\}⨯U→U,M⨯\RR_{>0}⨯U→U)$,
which provides a map of simplicial sets $\Cut(M⨯U→U,P)→\Cut((M⨯\RR)⨯U→U,P')$.
A geometric structure on the left is by definition of~$\iota^*_{d-1}$ a germ of an embedding $M⨯\RR⨯U→\RR^d⨯U$ around the core of~$M⨯U$,
which we map to its germ around the core of $M⨯\RR⨯U$ on the right side.

Recall that the simplicial set of morphisms is given by
$${\rm Mor}≔\coprod_{(\tilde M,\tilde P)}{\cal S}_{\square}(\tilde M\times U\to U)\times\DiffCut_\square(M\times U\to U,P),$$
where the coproduct is taken over the same pairs as for ${\rm Ob}$.
The simplicial set $\DiffCut(M\times U\to U,P)$
has as its $l$-simplices
an $l$-simplex of $\Cut(M\times U\to U,P)$
together with a germ of a diffeomorphism $M\to \tilde M$
taken around the union over all $t\in Δ^l$ of cores of~$C_t$ inside~$\tilde M$
and that commutes with the maps $P$ and $\tilde P$.

We send such a germ of a diffeomorphism $M\to \tilde M$
to the induced germ of a diffeomorphism $M⨯\RR→\tilde M⨯\RR$.
The resulting map preserves composition, again by definition of $\iota^*_{d-1}$ on morphisms.

We have defined a morphism~$f$ of simplicial groupoids.
To show that the diagonal~$g$ of the nerve of~$f$ is a weak equivalence,
observe that the codomain of~$g$ is weakly equivalent to the simplicial set of $U$-families of embedded codimension~1 bordisms in~$\RR^d$ with a choice of a normal orientation.
The domain of~$g$ is weakly equivalent to the simplicial set of $U$-families of embedded codimension~1 bordisms~$M$ in~$\RR^d$ with a choice of an embedding $M⨯\RR⨯U→\RR^d⨯U$
that extends the given embedding $M⨯U→\RR^d⨯U$.
The homotopy fiber of~$g$ over some vertex in the codomain of~$g$
is a simplicial torsor over the simplicial group~$G$ of $U$-families of diffeomorphisms $M⨯\RR⨯U→M⨯\RR⨯U$ (over~$U$) that preserve $M⨯\{0\}⨯U$, $M⨯\RR_{>0}⨯U$, and $M⨯\RR_{<0}⨯U$.

We now show that the simplicial group~$G$ is contractible.
The argument is standard and closely resembles the “make linear in a small neighborhood and zoom in” argument used in the proof of the smooth analogue of the Kister–Mazur theorem.
We apply the simplicial Whitehead theorem to~$G$.
Since any simplicial group is a Kan complex, there is no need to use subdivisions.

Suppose $g:∂Δ^a→G$ is a simplicial map.
We construct a filling $Δ^a→G$ as follows.
Given~$g$, construct the smooth map~$g'$ that for each $t∈∂Δ^a$
is given by the map $M⨯\RR⨯U→M⨯\RR⨯U$
that merely rescales the $\RR$-component
by the factor given by the derivative of the composition $M⨯\RR⨯U→M⨯\RR⨯U→\RR$ at the given point $(m,u)∈M⨯U$ in the direction of~$\RR$,
where the first map is given by~$g$ and the second map is projection.

Using the compactness of~$∂Δ^a$, choose $ε>0$ such that
the smooth homotopy that interpolates between $g$ and $g'$
via the formula $(1-tb(|x|/ε))g(x)+tb(|x|/ε)g'$
is an embedding for all~$t∈[0,1]$,
where $b:[0,∞)→\RR$ is a bump function that is 1 near~0 and 0 near $[1,∞)$.

Next, by rescaling $\RR$
using the formula $g'(x(1-t))/(1-t)$ for $t∈[0,1)$ (and extending to $t=1$ using $g'(x)$),
we can further deform~$g'$ to a family~$g''$ that is given by diffeomorphisms $M⨯\RR⨯U→M⨯\RR⨯U$
that do not change the $M$-component and are linear on the $\RR$-component.
The space of such diffeomorphisms is the space of smooth maps $M⨯U→\RR_{>0}$, which is contractible.
Thus, $G$ is contractible.
\end{proof}

\section{The homotopy type of the extended geometric bordism category}

In this section, we show how to extend the Galatius–Madsen–Tillmann–Weiss theorem \cite{GMTW}
to the case of fully extended bordism categories with geometric structures.
This extends the existing work in the case of topological structure of
Bökstedt–Madsen \cite{BokstedtMadsen} for multiple $(∞,d)$-categories
and Schommer-Pries \cite{SchommerPries.ITFT} for globular $(∞,d)$-categories (generalizing the fully extended case only).

\begin{definition}
The model category of simplicial spectra is denoted by~$\Sp$.
We also consider the left Bousfield localization of $\Sp$ with respect to morphisms
that induce isomorphisms on nonnegatively graded homotopy groups.
The resulting {\it connective\/} model structure on $\Sp$ is denoted by~$\Sp_{≥0}$.
\end{definition}

Recall the following Quillen equivalence,
where the left side models smooth symmetric monoidal $(∞,d)$-categories
in which all objects are (monoidally) invertible
and all $k$-morphisms for $0<k\le d$ are invertible.

\begin{proposition}[Bousfield–Friedlander \cite{BousfieldFriedlander}]
\label{sheaf.spectra.models}
We have a chain of left Quillen equivalences
$$\smcatinv_{\infty,d}=\sPSh(\cartsp⨯Γ⨯\Delta^{\times d})→\sPSh(\cartsp⨯\Gamma)→\PSh(\cartsp,\Sp_{≥0}),$$
where the left category $\smcatinv_{\infty,d}$ is the category $\smcatuple_{∞,d}$
equipped with the projective globular (or uple) model structure that is left Bousfield localized with respect to arbitrary maps of multisimplices
(making all $k$-morphisms invertible for any $0<k\le d$)
as well as the shear map for $\Gamma$
(making all objects monoidally invertible),
the middle category $\sPSh(\cartsp⨯\Gamma)$ is equipped with the stable model structure of Bousfield–Friedlander \cite{BousfieldFriedlander} in the direction of~$\Gamma$
and Čech structure in the direction of $\cartsp$,
and the right category $\PSh(\cartsp,\Sp_{\geq 0})$ is equipped with the Čech-local model structure valued in the stable {\it connective\/} model structure on simplicial spectra.
The left functor takes colimits with respect to $\Delta^{\times d}$ and the right functor is given by applying the Bousfield–Friedlander functor \cite{BousfieldFriedlander}
objectwise on $\cartsp$.
\end{proposition}
\begin{proof}
The statement for the left functor follows immediately from the definition of the model structures. The statement for the right functor follows from Bousfield--Friedlander \cite{BousfieldFriedlander}.
\end{proof}

\begin{definition}\label{def.K.functor}
We denote by
$$\cK:\smcatinv_{\infty,d}→\PSh(\cartsp,\Sp_{\geq 0})$$
the composition of the chain of equivalences from \cref{sheaf.spectra.models},
and by $\ldf\cK$ its left derived functor with respect to the model structures indicated there.
\end{definition}

Recall the notion of a geometric structure from \cref{geometric.structure}.

\begin{definition}\label{madtilspec}
Fix $d\geq 0$. We define the \emph{Madsen–Tillmann spectrum functor} as the composite functor 
$$\MT≔\ldf\cK\circ \BBord_d^{(-)}:\sPSh(\FEmb_d)\to \smcatinv_{\infty,d}\to \PSh(\cartsp,\Sp_{\geq 0}).$$
\end{definition}

The following result generalizes the classification of invertible topological field theories established by Schommer-Pries \cite[Theorem~6.14]{SchommerPries.ITFT}
to the case of arbitrary geometric structures.

\begin{theorem}
\label{invertible.tft}
Fix $d\geq 0$, ${\cal S}\in \sPSh(\FEmb_d)$, and a fibrant object $\mathscr{C}\in \smcatinv_{∞,d}$ (\cref{sheaf.spectra.models}).
Then the natural inclusion map
$${\rm c}_{\{1,\ldots,d\}}(\Funmon(\BBord_d^{\cal S},\mathscr{C})^⨯)→\Funmon(\BBord_d^{\cal S},\mathscr{C})$$
is a weak equivalence and
we have a natural weak equivalence
$$\ldf\cK\Funmon(\BBord_d^{\cal S},\mathscr{C})\simeq \rdf\hom(\MT({\cal S}),\ldf\cK(\mathscr{C})).$$
\end{theorem}

\begin{proof}
The first claim holds by the geometric cobordism hypothesis (\cref{mainthm.geometric}).
The second claim simply moves the left derived left Quillen equivalence $\ldf\cK$ (\cref{def.K.functor}) inside the (derived) internal hom.
\end{proof}

\begin{theorem}\label{madsensphere}
Fix $d\geq 0$ and $U∈\cartsp$.
Then we have a weak equivalence 
$$\MT(\RR^d\times U\to U)\simeq \mathbf{S}\otimes j(U),$$
where $j:\cartsp\to \sPSh(\cartsp)$ is the Yoneda embedding and $\otimes$ denotes the tensoring of presheaves of connective spectra over simplicial presheaves.
Furthermore, this weak equivalence is $\O(d)$-equivariant,
where the group $\O(d)$ acts via its canonical representation on $\RR^d$ on the left
and via the J-homomorphism on the right.
\end{theorem}

\begin{proof}
Let $j:\cartsp\times \Gamma\times \Delta^{\times d}\to \sPSh(\cartsp\times \Gamma\times \Delta^{\times d})$ denote the Yoneda embedding. The canonical map $$(\RR^d\times U,\{C^k:1\leq k \leq d\}, 1:\RR^d\times U\to \langle 1\rangle):j(U,\langle 1\rangle, {\bf 0})\to \BBord_d^{\RR^d\times U\to U},$$
described in \cref{thepoint}, induces a map 
$$p:\mathbf{S}\otimes j(U)\simeq \ldf \cK(j(U,\langle 1\rangle, {\bf 0}))\to \ldf \cK\BBord_d^{\RR^d\times U\to U}=\MT(\RR^d\times U\to U).$$
To see that this map is a weak equivalence, observe that for an arbitrary fibrant object $X∈\PSh(\cartsp,\Sp_{≥0})$ (\cref{sheaf.spectra.models}),
taking derived maps out of $p$ yields a map
$$\rdf\map(\BBord_d^{\RR^d\times U\to U},\Omega^{\infty}X)\simeq \rdf\map(\ldf\cK\BBord^{\RR^d\times U\to U},X)\to \map(\mathbf{S}\otimes j(U),X)\simeq (\Omega^{\infty}X)(U).$$ 
Here, $\Omega^{\infty}$ is the right Quillen adjoint $\Omega^{\infty}\vdash \cK$ in \cref{def.K.functor,sheaf.spectra.models},
which is derived because $X$ is fibrant.
The above map is a weak equivalence by the framed geometric cobordism hypothesis (\cref{mainthm}).
\end{proof}

To connect with the Madsen--Tillmann construction \cite{MadsenTillmann} in the general case, we will use the following observations.
Let $b:\FEmb_d\to \cartsp$ denote the base space functor, sending $M\to U$ to $U$. 
The category $\sPSh(\FEmb_d)$ is tensored over $\sPSh(\cartsp)$ via 
\begin{equation}
Y\boxtimes X≔ b^*(Y)\times X, \qquad X\in \sPSh(\FEmb_d), \quad Y\in \sPSh(\cartsp),
\end{equation}
where $\sPSh(\FEmb_d)$ is enriched in $\sPSh(\cartsp)$ by applying the right adjoint $b_*$ to the internal hom:
$$\map(b^*Y⨯X,Z)≅\map(b^*Y,\hom(X,Z))≅\map(Y,b_*\hom(X,Z)).$$
Furthermore, the functor $b^*$ is a left Quillen functor
$$b^*:\sPSh(\cartsp)_\loc→\sPSh(\FEmb_d)_\loc,$$
where both sides are equipped with the Čech-local injective model structure.

Similarly, let $p:\cartsp\times \Gamma\times \Delta^{\times d}\to \cartsp$ be the projection. Then $\sPSh(\cartsp\times \Gamma\times \Delta^{\times d})$ is tensored over $\sPSh(\cartsp)$ via 
\begin{equation}
Y\tboxtimes X≔ p^*(Y)\times X, \qquad X\in \sPSh(\cartsp\times \Gamma\times \Delta^{\times d}), \quad Y\in \sPSh(\cartsp).
\end{equation}

\begin{lemma}\label{bordtensor}
The functor 
$$\BBord_d:\sPSh(\FEmb_d)_{\inj,\loc}\to \smcatdual_{\infty,d}$$
is a left $\sPSh(\cartsp)$-Quillen functor. 
\end{lemma}

\begin{proof}
We have already established that it is left Quillen in Grady--Pavlov \cite[\ecref{EL-mainthm}]{GradyPavlov}.
Since $\BBord_d$ is cocontinuous, it suffices to prove that 
$$\BBord_d^{Y\boxtimes {\cal S}}= \BBord_d^{b^*Y\times {\cal S}}\cong p^*Y\times \BBord_d^{\cal S}=Y\tboxtimes \BBord_d^{\cal S}$$
for representable $Y\in \sPSh(\cartsp)$.
We simply unravel the definition of $\BBord_d^{b^*Y⨯{\cal S}}$.
The presheaf $b^*Y$ sends an object $T\to V$ to the set of smooth maps $\sm(V,Y)$.
For a summand of $\BBord_d^{b^*Y\times {\cal S}}$ indexed by a pair $(M,P)$, we have
$$(b^*Y⨯{\cal S})(M\times U\to U)\cong {\cal S}(M\times U\to U)\times \sm(U,Y).$$
The multisimplicial maps, which remove and duplicate cuts in $C$, are identity on the factor $\sm(U,Y)$.
Similarly, maps in $\Gamma$ induce the identity on $\sm(U,Y)$.
With these observations, the claim follows by a levelwise inspection of the bordism category with geometric structure (Grady--Pavlov \cite[\ecref{EL-bordstr}]{GradyPavlov}).
\end{proof}

\begin{lemma}
\label{MTtensor}
For any $Y\in \sPSh(\cartsp)$, we have a zig-zag of equivalences
$$\MT(Y\boxtimes^\ldf {\cal S})\simeq Y\otimes^\ldf \MT({\cal S}),$$
where on the right we have the tensoring of presheaves of connective spectra over simplicial presheaves on $\cartsp$, which is defined objectwise.
\end{lemma}

\begin{proof}
Indeed, this follows directly from the definition of $\MT$ in \cref{madtilspec} and \cref{bordtensor}:
$$\MT(Y\boxtimes^\ldf {\cal S})
=\ldf\cK\BBord_d^{Y\boxtimes^\ldf{\cal S}}
\simeq\ldf\cK(Y\boxtimes^\ldf\BBord_d^{{\cal S}})
\simeq Y\otimes^\ldf\ldf\cK(\BBord_d^{\cal S})
\simeq Y\otimes^\ldf\MT({\cal S}).\qedhere$$
\end{proof}

Let us recall the shape functor 
\begin{equation}\label{shapeofspec}
\csp:\PSh(\cartsp,\Sp_{\geq 0})\to \Sp_{\geq 0},
\end{equation}
in \cref{cohesiveadj}
(see also \cref{shapeequivariant}), which relates sheaves of connective spectra and ordinary connective spectra.
Using the above lemmas, we can recover the usual Madsen--Tillmann spectrum  as follows. 

\begin{proposition}\label{madsengstr}
Let $G$ be a Lie group and let $\rho:G\to \GL(d)$ be a representation.
Then $G$ acts on the representables of the form $\RR^d\times U\to U$ fiberwise, via the representation $\rho$.
Explicitly, the action 
$$m:\sm(-,G)\boxtimes j(\RR^d\times U\to U)\to j(\RR^d\times U\to U),$$
is defined on an object $T\to V$ by $$m_{T→V}(g:V→G,f:T→\RR^d⨯U,h:V→U)=((ρg⨯\id_U)∘f,h).$$
We have an equivalence 
$$\ldf\csp\MT((\RR^d⨯U\to U)\hq \sm(-,G))\simeq \Sigma^d\MTcla G,$$
where on the right, we have the $d$-fold suspension of the usual Madsen–Tillmann spectrum \cite{MadsenTillmann}.
\end{proposition}

\begin{proof}
In the following, we implicitly derive all functors.
By \cref{madsensphere}, we have an equivalence $$\MT(\RR^d\times U\to U)\simeq \mathbf{S}\otimes j(U),$$ which is natural in $U$.
Since $\MT$ is  homotopy cocontinuous, and $\MT$ commutes with the tensoring over $\sPSh(\cartsp)$ (\cref{MTtensor}), we have a weak equivalence
$$\MT((\RR^d⨯U\to U)\hq \sm(-,G))\simeq (\mathbf{S}⊗j(U))\hq \sm(-,G),$$
which is natural in $U$.
Applying the shape functor $\csp$ in \cref{shapeofspec}, we obtain equivalences
$$\ldf\csp\left(\MT((\RR^d⨯U\to U)\hq \sm(-,G))\right)\simeq \ldf\csp\left((\mathbf{S}⊗j(U))\hq \sm(-,G)\right)\simeq \mathbf{S}\hq \sing(G).$$
By \cref{madsensphere}, the action of $\sing(G)$ on $\mathbf{S}$ is given by the J-homomorphism.
The homotopy quotient $\mathbf{S}\hq \sing(G)$ is weakly equivalent to $Σ^d\MTcla G$.
\end{proof}

\begin{remark}
In the context of \cref{madsengstr},
a (derived) map $M\to \RR^d\hq G_\delta$ is precisely a rigid $G$-geometry on~$M$ in the sense of Stolz--Teichner (see Grady--Pavlov \cite[\ecref{EL-rigid.geometry}]{GradyPavlov}).
That is to say, the Madsen--Tillmann spectrum is the invertible bordism category for the geometric structure given by the rigid geometry $\rho:G\to\GL(d)$.
\end{remark}

\section{Examples}

In this section, we work out the dimension 1 case of the cobordism hypothesis for geometric structures like Riemannian metrics and smooth maps to a target manifold. 

\subsection{$d=1$ with Riemannian metrics}

Let ${\cal R}:\FEmb_1^\op\to \set$ be the sheaf that sends a submersion $M\to U$ with $1$-dimensional fibers to the set of fiberwise Riemannian metrics with fiberwise orientations on $M$. We now apply the cobordism hypothesis in this setting.
As expected, we recover the classification of 1-dimensional Riemannian field theories (oriented and unoriented)
in terms of 1-parameter groups of morphisms.
To the best of our knowledge, the resulting \cref{one.dimensional.riemannian} is new in this generality,
when $\mathscr{C}$ is a smooth symmetric monoidal $(∞,1)$-category with duals.
We start by treating the oriented case.
The unoriented case then follows. 

\begin{remark}
The presence of thin homotopies in the bordism category implies the existence of homotopies in the target that may seem undesirable for geometric field theories. However, after suitably interpreting the target as a complete Segal space, one sees that the homotopies themselves carry geometric information. For example, let us consider field theories of the form
$$Z:\BBord_1^{\cal R}\to {\rm Vect}.$$
For simplicity, we evaluate at $U=\ast\in \cartsp$. An interval of length $l$ in the bordism category can be deformed by a thin homotopy to an interval of arbitrary length. This would seem to imply that in the target, all intervals are mapped to the same linear map. However, this is not the case, as we now demonstrate.

The complete Segal space model for ${\rm Vect}$ is obtained by taking the Rezk nerve. The simplicial set of objects is the nerve of the groupoid ${\rm Vect}^{\times}$, i.e. taking only invertible linear maps as morphisms. The simplicial set of morphisms is the nerve of the groupoid underlying the arrow category of ${\rm Vect}$. Explicitly, a bisimplex of degree $(1,1)$ is a commutative square 
$$
\xymatrix{
V\ar[r]^-L\ar[d]^-{\cong}_-{A} & W\ar[d]_{\cong}^-{B}
\\
Y\ar[r]^-K & Z,
}
$$
where the vertical maps are isomorphisms. 

In the oriented Riemannian case, $Z$ sends an object $+\in \BBord_1^{\cal R}$ (i.e., a point) to a finite dimensional vector space $V$. The dual point $-$ is sent to $V^{\vee}$. Consider the interval $[0,l]$ of length $l$, with orientation compatible with the orientation $+$ on germs of the endpoints. The bordism $[0,l]$ is sent to an invertible linear map $L_l:V\to V$. Observe that there is a $(1,1)$-bisimplex in $\BBord_1^{\cal R}$ that contracts the interval of length $l$ to a point. However, the definition of the bordism category restricts the types of homotopies that are allowed. In particular, the homotopy $h$ that contracts the interval $[0,l]$ to a point moves the target cut of the interval to the source cut, but the ambient interval and its Riemannian length are fixed. Applying the functor $Z$ to the this $(1,1)$-bisimplex yields a commutative diagram 
$$
Z:\vcenter{\vbox{\xymatrix{
+\ar[r]_{\cong}^-{[0,l]}\ar[d]^-{\cong}_-{([0,l],h)} & +\ar[d]^-{{\rm id}}
\\
+\ar[r]^-{\rm id} & +.
}}}\mapsto \vcenter{\vbox{\xymatrix{
V\ar[r]_{\cong}^-{L_l}\ar[d]^-{\cong}_-{L_l} & V\ar[d]^-{{\rm id}}
\\
V\ar[r]^-{\rm id} & V.
}}}
$$
The left column corresponds to a nontrivial 1-simplex in the simplicial set of objects of $\BBord_1^{\cal R}$,
which moves a point inside of an interval of length~$l$.
The field theory~$Z$ maps this 1-simplex to a 1-simplex in the simplicial set of objects of $\Vect$, given by the linear automorphism~$L_l$.
In the \emph{topological} case, the length of the interval $[0,l]$ can also be rescaled by a homotopy. For any fixed $\epsilon>0$, such a rescaling can be used to define a 2-simplex $w$ in the space of objects with $d_1(w)=([0,\epsilon],h)$, $d_0(w)={\rm id}$ and $d_2(w)=([0,l],h)$. This forces $L_l=L_\epsilon$ in the diagram on the right above, so that intervals of different length are necessarily sent to the same morphism. However, in $\BBord_1^{\cal R}$, the Riemannian length of intervals cannot be rescaled and such a 2-simplex does not exist, so intervals of different length are distinguished.
\end{remark}

\begin{definition}\label{def.Cc}
Let 
$$\Cc_d:\sPSh(\FEmb_d)_\lconst\to \PSh(\cartsp,\sset^{\GL(d)})_\loc$$
denote the functor obtained
by composing the left derived left adjoint of~$\rho$ with the right derived functor of~$r$ in \cref{fiberwiseshape}.
We also have an analogous functor
$$\Cc_d:\sPSh(\FEmbCart_d)_\lconst\to \PSh(\cartsp,\sset^{\GL(d)})_\loc,$$
where $\FEmbCart_d$ is as in \cref{fembcartdef}.
\end{definition}

From now on we work with the site $\FEmbCart_1$.

\begin{lemma}\label{shapemetrics}
We have an equivalence of ${\rm GL}(1)$-equivariant simplicial presheaves 
$$\Cc_1({\cal R})\simeq \ZZ/2\times \deloop \RR,$$
where $\GL(1)$ acts on the first factor by the left action on the connected components 
$$\GL(1)\to \pi_0(\GL_1(1))\cong \ZZ/2.$$
\end{lemma}

\begin{proof}
It is convenient to replace ${\cal R}$
with a weakly equivalent subpresheaf ${\cal R}'⊂{\cal R}$ (which is not a sheaf)
given by fiberwise Riemannian metrics that have finite length on each fiber.
We write ${\cal R}'$ as a homotopy colimit
$${\cal R}\simeq \hocolim_{(\RR\times U→U)\to {\cal R}'}(\RR\times U→U),$$
where the homotopy colimit is taken over the comma category $j/{\cal R}$,
where $j$ is the Yoneda embedding.

Consider the category ${\sf C}$ with objects given by smooth functions $\alpha:U\to \RR_{>0}$
and morphisms given by pairs $(h:U\to U',\beta:U\to \RR_{\geq 0}):\alpha\to \alpha'$ such that $\alpha'h≥\alpha+\beta$.
The composition is defined by $$(h':U'→U'',β':U'→\RR_{≥0})∘(h:U→U',β:U→\RR_{≥0})=(h'h:U→U'',β'h+β:U→\RR_{≥0}).$$
Consider the functor 
$$\iota:{\sf C}\into j/{\cal R}$$
defined by sending $\alpha:U\to \RR_{>0}$ to the projection $$p:(0,\alpha)=\{(t,u)∈\RR⨯U\mid 0<t<α(u)\}\to U,$$ equipped with the standard metric.
A morphism $(h:U→U',\beta:U→\RR_{≥0}):α→α'$ in ${\sf C}$ is sent to the map $$(0,α)\to (0,α'),\qquad (t,u)↦(t+β(u),h(u)).$$
The functor $\iota$ is an equivalence of categories.
Hence, 
\begin{align*}
\Cc_1({\cal R})
\simeq \Cc_1({\cal R}')
&\simeq \Cc_1\left(\hocolim_{(\RR\times U→U)\to {\cal R}'}(\RR\times U→U)\right)\\
&\simeq \Cc_1\left(\hocolim_{\alpha:U\to \RR_{> 0}} ((0,α)→U)\right)\\
&\simeq \hocolim_{\alpha:U\to \RR_{>0}} \Cc_1((0,α)→U)\\
&\simeq \hocolim_{\alpha:U\to \RR_{>0}} \ZZ/2\times U.\\
\end{align*}
We compute the homotopy colimit objectwise (for each fixed $W\in \cartsp$) via the Grothendieck construction.
The construction produces a category ${\sf D}_{W}$ equipped with a $\GL(1)$-action, whose objects are triples
$$(\alpha:U\to \RR_{>0},f:W\to U,\sigma\in \ZZ/2)\in {\sf D}_W$$ and morphisms
are pairs $$(h:U\to U',\beta:U\to \RR_{\geq 0}):(\alpha:U→\RR_{≥0},f:W→U,\sigma∈\ZZ/2)\to (\alpha':U'→\RR_{>0},f':W→U',\sigma'∈\ZZ/2)$$ such that 
$$hf=f',\quad \sigma'=\sigma,\quad \alpha'h≥\alpha+\beta.$$
An element $\rho\in \GL(1)$ acts by 
$$\rho:(\alpha,f,\sigma)\mapsto (\alpha,f,{\rm sgn}(\rho)\sigma).$$ 

The category ${\sf D}_W$ is the direct product of $\ZZ/2$ and its full subcategory ${\sf D}'_W$ on $\sigma=0$.
Below, we compute the nerve of ${\sf D}'_W$.

Consider the category ${\sf E}_W$ whose objects are smooth maps $A:W→\RR_{>0}$
and morphisms $A→A'$ are smooth maps $B:W→\RR_{≥0}$ such that $A+B≤A'$.
We have a functor ${\sf D}'_W→{\sf E}_W$ that sends an object $(α,f)$ to $αf:W→\RR_{>0}$
and a morphism $(h,β):(α,f)→(α',f')$ to the morphism $βf:W→\RR_{≥0}$.
We have a functor ${\sf E}_W→{\sf D}'_W$ that sends an object $A:W→\RR_{>0}$ to $(A,\id_W)$
and a morphism $B:W→\RR_{≥0}$ to $(\id_W,B)$.
The composition ${\sf E}_W→{\sf D}'_W→{\sf E}_W$ is the identity functor.
The composition ${\sf D}'_W→{\sf E}_W→{\sf D}'_W$ admits a natural transformation to the identity functor
whose component indexed by an object $(α,f)$ is given by the morphism $(f,0):(αf,\id_W)→(α,f)$.
Thus, the nerves of ${\sf D}'_W$ and ${\sf E}_W$ are weakly equivalent.

Let ${\cal B}(\sm(W,\RR),+)$ denote the groupoid with a single object with morphisms given by smooth maps $a:W\to \RR$, with pointwise addition as composition.
We define a functor
\begin{equation}
F:{\sf E_{W}}\to {\cal B}(\sm(W,\RR),+)
\end{equation}
by $F(B)=B$.

We claim that $F$ induces a weak equivalence on nerves.
By Quillen's theorem A, it is sufficient to prove that the nerve of the comma category $\ast\downarrow F$ is contractible in each case.
For this, we observe that the comma category can be identified with the following category.
The objects are pairs
$$(A:W\to \RR_{>0},g:W→\RR)$$
and the morphisms $$(A:W→\RR_{>0},g:W→\RR)→(A':W→\RR_{>0},g':W→\RR)$$
are maps $$B:W\to \RR_{\geq 0},\qquad A'≥A+B,\quad g=g'+B.$$

Since $B=g-g'$, any two parallel morphisms are equal, so the comma category is a preorder.
It suffices to show that this preorder is filtered.
It is nonempty because $(1,0)$ is an object.
Given two objects $(A,g)$ and $(A',g')$, we can construct
a new object $(A+A',G)$ together with morphisms $B:(A,g)→(Ψ,G)$ and $B':(A',g')→(Ψ,G)$.
The conditions that must be satisfied are
$Ψ≥A+B$, $g=G+B$, $Ψ≥A'+B'$, $g'=G+B'$.
That is, $g-g'=B-B'$.
We can always choose smooth $B,B':W→\RR_{≥0}$ such that $g-g'=B-B'$.
Once this is done, set $G=g-B=g'-B'$
and pick any smooth $Ψ:W→\RR_{>0}$ such that $Ψ≥A+B$ and $Ψ≥A'+B'$.
Thus, the comma category is contractible.

We have shown that we have an equivalence 
$$\hocolim_{\alpha:U\to \RR_{>0}}\ZZ/2\times U\simeq \ZZ/2\times \deloop \RR,$$
where the $\ZZ/2$-action on the right is the left action on the first factor. This completes the proof. 
\end{proof}

Having completed the preparatory steps, we are now ready to prove the main theorem of this section,
which appears to be new in the stated generality of an arbitrary smooth symmetric monoidal $(\infty,1)$-category with duals $\mathscr{C}$.

\begin{theorem}
\label{one.dimensional.riemannian}
Let $\mathscr{C}$ be a smooth symmetric monoidal $(\infty,1)$-category with duals. We have an equivalence 
$$\FFT_{1,{\cal R},\mathscr{C}}=\Funmon(\BBord_1^{{\cal R}},\mathscr{C})\simeq \map(\deloop \RR,\mathscr{C}^{\times})$$
in the oriented case and
$$\FFT_{1,{\cal R}_{\rm un},\mathscr{C}}=\Funmon(\BBord_1^{{\cal R}_{\rm un}},\mathscr{C})\simeq \map^{\ZZ/2}(\deloop \RR,\mathscr{C}^{\times})$$
in the unoriented case, where the action of $\ZZ/2$ on $\deloop\RR$ is trivial and on $\mathscr{C}^\times$ it is given by dualization.
\end{theorem}

\begin{proof}
By the framed geometric cobordism hypothesis (\cref{mainthm}), $\mathscr{C}^{\times}$ is promoted to a fiberwise locally constant presheaf $\mathscr{C}^{\times}_1$ on $\FEmb_1$.
Equivalently, by \cref{fiberwiseshape}, $\mathscr{C}^{\times}_1$ can be identified with a presheaf $\Cc_1(\mathscr{C}^⨯_1)$ on $\cartsp$ with values in $\GL(1)$-equivariant spaces.
By \cref{shapemetrics}, we have a weak equivalence $\Cc_1({\cal R})\simeq \ZZ/2\times\deloop\RR$.
Thus, we have weak equivalences of derived mapping spaces 
$$\map({\cal R},\mathscr{C}^{\times}_1)\simeq \map_{\GL(1)}(\Cc_1({\cal R}),\Cc_1(\mathscr{C}^{\times}_1))\simeq \map_{\GL(1)}(\ZZ/2\times \deloop \RR,\Cc_1(\mathscr{C}^{\times}_1))\simeq \map(\deloop \RR,\mathscr{C}^{\times}).$$
The last step uses \cref{mainthm.geometric} to identify $\ZZ/2$-equivariant maps $\ZZ/2⨯\deloop\RR→\Cc_1(\mathscr{C}^⨯_1)$
with maps $\deloop\RR→\mathscr{C}^⨯$, since the underlying sheaf of $\Cc_1(\mathscr{C}^⨯_1)$ is weakly equivalent to $\mathscr{C}^⨯$ after we discard the action of $\ZZ/2$.
\end{proof}

In the previous corollary, we can take $\mathscr{C}$ to smooth stack of finite dimensional vector bundles ${\sf Vect}$. Then $\mathscr{C}^{\times}={\sf Vect}^{\times}$ is the smooth stack of vector bundles with invertible maps between them. 
Then $\map(\deloop \RR,\mathscr{C}^{\times})$ in \cref{one.dimensional.riemannian} can be identified as the nerve of the category of $\RR$-representations, with invertible morphisms between them.
That is, we get a vector bundle with an endomorphism that defined the infinitesimal generator of the corresponding one-parameter group.

This description continues to apply in the infinite-dimensional setting, provided that we make an appropriate choice of the category of values.
For example, if we take $\mathscr{C}$ to be the free symmetric monoidal category with duals on the category of Hilbert spaces (possibly infinite-dimensional),
we can encode quantum mechanics on an infinite-dimensional Hilbert space as a 1-dimensional Riemannian field theory.

\subsection{Principal bundles and vector bundles with connections}
\label{principal.bundles}

In this section we consider the case of $d=1$ with $\mathscr{C}=\deloop G$ being the delooping of a Lie group~$G$.
It turns out that $\mathscr{C}^⨯_1$ is the stack $\deloop_∇ G$ of fiberwise principal $G$-bundles with connection,
and 1-dimensional functorial field theories with geometric structure ${\cal S}$
and target $\deloop G$ are precisely principal $G$-bundles with connection on ${\cal S}$.
This result appears to be new when ${\cal S}$ is an arbitrary geometric structure.

\begin{definition}
Given a Lie group~$G$,
the simplicial presheaf $\deloop G$ on the site $\cartsp$ sends $U∈\cartsp$ to the delooping $\tdeloop\sm(U,G)$.
\end{definition}

\def\lalg{\mathfrak{g}}
\def\Ad{\mathop{\rm Ad}\nolimits}

\begin{definition}
Given a Lie group~$G$,
the simplicial presheaf $\fdeloop_∇ G$ on the site $\FEmb_1$ sends $(T→U)∈\FEmb_1$ to the nerve of the following groupoid.
Objects are fiberwise differential 1-forms on $T→U$ valued in the Lie algebra $\lalg$ of the Lie group~$G$,
denoted by $Ω^1_U(T,\lalg)$.
Morphisms are pairs $(α,h)$, where $α∈Ω^1_U(T,\lalg)$ and $h∈\sm(T,G)$.
The source of $(α,h)$ is~$α$.
The target of $(α,h)$ is $α\cdot h=h^*\theta+\Ad_{h^{-1}}α$,
where $\theta$ is the Maurer--Cartan form of~$G$,
$h^*$ denotes the fiberwise pullback of a differential form,
and $\Ad$ denotes the adjoint action.
The composition is defined as $(α',h')(α,h)=(α,h'h)$.
\end{definition}

\begin{definition}
\label{bgmap}
We define a morphism $\fdeloop_∇ G→(\deloop G)^⨯_1$ of simplicial presheaves on~$\FEmb_d$ as follows.
Fix an object $(T→U)∈\FEmb_d$.
Given a fiberwise differential 1-form on $T→U$ valued in the Lie algebra $\lalg$ of the Lie group~$G$,
we send it to the induced parallel transport functor $\BBord_1^{T→U}→\deloop G$.
Given a smooth function $T→G$, we send it to the induced natural transformation of parallel transport functors.
\end{definition}

\begin{proposition}
\label{compute.bg.refinement}
The map of \cref{bgmap} is a weak equivalence.
\end{proposition}

\begin{proof}
We use the geometric framed cobordism hypothesis (\cref{mainthm.geometric}):
the map
$$\FFT_{1,T→U,\deloop G}=\Funmon(\BBord_1^{T→U},\deloop G)→\deloop G(U)$$
given by the evaluation on $0∈\RR$ in each fiber,
is a weak equivalence for any $(T→U)∈\FEmbCart_1$ (\cref{fembcartdef}).
The simplicial presheaf $\fdeloop_∇ G$ on the site $\FEmb_1$ has a trivial~$π_0$ because any relative principal $G$-bundle with connection
on a fibered family of 1-dimensional cartesian manifolds is trivial.
Since both sides are 1-truncated,
it suffices to show that the induced map on $π_1$ is an isomorphism.
Indeed, $π_1(\fdeloop_∇ G)$ is the sheaf of fiberwise constant $G$-valued functions.
On $(T→U)∈\FEmbCart_1$, fiberwise constant $G$-valued functions can be identified with $G$-valued smooth functions on~$U$.
(This can be summarized by saying that any principal $G$-bundle with connection becomes a flat $G$-bundle once restricted to a (family of) 1-dimensional manifolds.)
\end{proof}

\cref{compute.bg.refinement} and \cref{mainthm.geometric} immediately imply the following result, which appears to be new in the stated generality.

\begin{theorem}
For any Lie group~$G$ and geometric structure ${\cal S}∈\sPSh(\FEmb_1)$,
the space $\FFT_{1,{\cal S},\deloop G}$ is weakly equivalent
to the space of principal $G$-bundles with connection on~${\cal S}$,
defined as the derived mapping space $\rdf\map({\cal S},\fdeloop_∇ G)$.
\end{theorem}

\def\doi#1{\href{https://doi.org/#1}{doi:#1}}
\def\arXiv#1{\href{https://arxiv.org/abs/#1}{arXiv:#1}}
\def\gen#1{\href{http://gen.lib.rus.ec/book/index.php?md5=#1}{PDF}}
\def\jstor#1{\href{https://www.jstor.org/stable/#1}{JSTOR:#1}}
\def\numdam#1{\href{http://www.numdam.org/item/#1}{numdam:#1}}



\end{document}